\documentclass[10pt,a4paper,leqno]{amsart}
\usepackage{amsmath}
\usepackage{amsthm}
\usepackage{amsfonts}
\usepackage{amssymb}
\usepackage{graphicx}
\usepackage{amsbsy}
\usepackage{mathrsfs}
\usepackage{dsfont}
\usepackage{bm}
\newcommand{\be}{\begin{equation}}
\newcommand{\ee}{\end{equation}}
\newcommand{\ba}{\begin{array}}
\newcommand{\ea}{\end{array}}
\newcommand{\bea}{\begin{eqnarray}}
\newcommand{\eea}{\end{eqnarray}}
\newcommand{\bee}{\begin{eqnarray*}}
\newcommand{\eee}{\end{eqnarray*}}

\newcommand{\Or}{\mathcal{O}}
\newcommand{\tgamma}{\widetilde{\gamma}}
\newcommand{\Mod}{{\mathbf {\tt Mod}}}

\newcommand{\loger}[1]{\log^{\frac 1 2} \! \left ( 2 + #1^{-1} \right )}

\newcommand{\In}{\mathfrak{I}}

\newcommand{\Pa}{\mathcal{P}}

\newcommand{\Qp}{Q_\Pa}
\newcommand{\Qbp}{\bm{Q}_\Pa}

\newtheorem{thm}{Theorem}
\newtheorem{lemma}{Lemma}
\newtheorem{prop}{Proposition}

\newtheorem{remark}{Remark}

\newtheorem*{remark*}{Remark}

\numberwithin{equation}{section}
\numberwithin{lemma}{section}
\numberwithin{prop}{section}
\numberwithin{cor}{section}
\numberwithin{thm}{section}
\numberwithin{definition}{section}
\numberwithin{remark}{section}
\numberwithin{assump}{section}

\catcode`@=11
\def\section{\@startsection{section}{1}%
  \z@{1.5\linespacing\@plus\linespacing}{.5\linespacing}%
  {\normalfont\bfseries\large\centering}}
\catcode`@=12

\newcommand{\R}{\mathbb{R}}

\newcommand{\C}{\mathbb{C}}
\newcommand{\N}{\mathbb{N}}

\renewcommand{\leq}{\leqslant}
\renewcommand{\geq}{\geqslant}
\newcommand{\D}{D}

\newcommand{\weakto}{\rightharpoonup}

\newcommand{\eps}{\varepsilon}

\newcommand{\cE}{\mathcal E}
\newcommand{\cJ}{\mathcal J}

\def\pa{\partial}
\def\fref{\eqref}
\def\l{\lambda}
\def\e{\varepsilon}
\def\lsl{\frac{\lambda_s}{\lambda}}
\def\l{\lambda}
\def\matchal{\mathcal}

\begin{document}

\title[Nondispersive solutions to the $L^2$-critical half-wave equation]{Nondispersive solutions to \\ the $L^2$-critical Half-Wave Equation}

\author{Joachim Krieger, Enno Lenzmann, and Pierre Rapha\"el}

\address{Department of Mathematics, EPFL, CH-1051 Lausanne, Switzerland}
\email{joachim.krieger@epfl.ch}
\address{Mathematisches Institut, Universit\"at Basel, Rheinsprung 21, CH-4051 Basel, Switzerland}
\email{enno.lenzmann@unibas.ch}
\address{Institut Math\'ematiques de Toulouse $ \&$ Institut Universitaire de France, Universit\'e Paul Sabatier, Toulouse, France}
\email{pierre.raphael@math.univ-toulouse.fr}
\maketitle

\begin{abstract}
We consider the focusing $L^2$-critical half-wave equation  in one space dimension
$$
i \partial_t u = D u - |u|^2 u, 
$$
where $D$ denotes the first-order fractional derivative. Standard arguments show that there is a critical threshold $M_* > 0$ such that all $H^{1/2}$ solutions with $\| u \|_{L^2} < M_*$ extend globally in time, while solutions with $\| u \|_{L^2} \geq M_*$ may develop singularities in finite time.

In this paper, we first prove the existence of a family of traveling waves with subcritical arbitrarily small mass. We then give a second example of nondispersive dynamics and show the existence of finite-time blowup solutions with minimal mass $\| u_0 \|_{L^2} = M_*$. More precisely, we construct a family of minimal mass blowup solutions that are parametrized by the energy $E_0 >0$ and the linear momentum $P_0 \in \R$. In particular, our main result (and its proof) can be seen as a model scenario of minimal mass blowup for $L^2$-critical nonlinear PDE with nonlocal dispersion.
\end{abstract}


\section{Introduction and Main Results}


\subsection{Setting of the problem}


We consider in this paper the half-wave equation in $N=1$ space dimension with focusing $L^2$-critical nonlinearity:
\begin{equation} \label{eq:wave}
(Wave)\ \ \left \{ \begin{array}{l} i \partial_t u = D u - |u|^2 u, \\
u(t_0,x) = u_0(x), \quad u : I \times \R \to \C. \end{array} 
\right . 
\end{equation}
Here $I \subset \R$ is an interval containing the initial time $t_0 \in \R$ and
$$\widehat{(Df)}(\xi) = |\xi| \widehat{f}(\xi)$$ 
denotes the first-order nonlocal fractional derivative. Equation \eqref{eq:wave} can be seen as a canonical model for an $L^2$-critical PDE with nonlocal dispersion given by a fractional power of the Laplacian. Let us mention that evolution problems with nonlocal dispersion such as \eqref{eq:wave} naturally arise in various physical settings, which include continuum limits of lattice systems \cite{KiLeSt2011}, models for wave turbulence \cite{CaMa2001,MaMc1997}, and gravitational collapse \cite{ElSc2007,FrLe2007}. The defocusing version of this problem is at the heart of the derivation of asymptotic models of weak turbulence through the cubic Szeg\"o model studied by Gerard, Grellier \cite{GeGr2010} and Pocovnicu \cite{Po2011}. From a mathematical point of view, the absence of specific symmetries for evolution problems like \eqref{eq:wave} (e.\,g., there is no Lorentz, Galilean or pseudo-conformal symmetry) makes the analysis rather intricate and hence robust (i.\,e., symmetry-independent) arguments have to be found. 

Let us review some basic facts about the problem at hand. The Cauchy problem for \eqref{eq:wave} is locally well-posed in the energy space $H^{1/2}(\R)$; see Appendix \ref{sec:cauchy} for more details. In particular, we have the blowup alternative that if $u\in C^0(I; H^{1/2}(\R))$ is the unique corresponding solution to \eqref{eq:wave} with its maximal time of existence $t_0 < T \leq +\infty$, then 
\be
\label{blowupcrtiterion}
T<+\infty\ \ \mbox{implies}\ \ \lim_{t\to T^-}\|u(t)\|_{H^{1/2}}=+\infty.
\ee
Furthermore, equation \eqref{eq:wave} is an infinite-dimensional Hamiltonian system, which admits three conservation laws given by
\bee
&& {\rm Mass}:\ \ M(u)=\int |u(t,x)|^2dx=M(u_0),\\
&& {\rm Momentum}: \ \ P(u)= \int -i \pa_xu(t,x)\overline{u}(t,x)dx= P(u_0),\\
&& {\rm Energy}: \ \ E(u)=\frac12\int|D^{\frac12}u|^2(t,x)dx-\frac14\int|u(t,x)|^4dx=E(u_0).
\eee
For the half-wave equation \eqref{eq:wave}, one easily verifies that the mapping
\be \label{eq:symmetry}
u(t,x) \mapsto \lambda_0^{\frac 12}u(\lambda_0 t+t_0,\lambda_0 x+x_0)e^{i\gamma_0}, \ \ (\lambda_0,t_0,x_0)\in\R^*_+\times\R\times \R,
\ee
yields a group of symmetries. In particular, the scaling symmetry leaves the $L^2$-norm in space invariant, and hence the problem is {\it mass critical}. A classical criterion of global-in-time existence for $H^{1/2}$ initial data is derived by using the Gagliardo--Nirenberg with best constant $$\forall u\in H^{1/2}(\R), \ \ \|u\|_{L^4}^4\leq C_*\|\D^{\frac12}u\|_{L^2}^2\|u\|_{L^2}^2$$ which is attained at the {\it unique} (up to symmetries) {\it ground state} profile solution to
\be
\label{neoenoneo}
DQ+Q-Q^3=0, \ \ Q(x)>0, \ \ Q\in H^{1/2}(\R).
\ee Note that the existence of this object follows from standard variational techniques, but uniqueness of $Q$, which was obtained recently by Frank and Lenzmann in \cite{FrLe2012},  is a nontrivial claim in a nonlocal setting, since ODE techniques do not apply for \eqref{neoenoneo}. The outcome is the sharp lower bound on the energy 
\be \label{eq:gwpbound}
\forall u \in H^{1/2}(\R), \ \ E(u)\geq\frac12 \int |D^{\frac 12}u|^2\left[1-\frac{\|u\|_{L^2}^2}{\|Q\|_{L^2}^2}\right],
\ee
 which together with the conservation of mass and energy and the blowup criterion \fref{blowupcrtiterion} implies that initial data $u_0\in H^{1/2}(\R)$ with 
$$\|u_0\|_{L^2}< M_*=\|Q\|_{L^2}$$
 generate global-in-time solutions. For more details about the Cauchy problem \eqref{eq:wave}, we refer to Appendix \ref{sec:cauchy} below.


\subsection{The local NLS problem} 


The structure of the problem is similar to the celebrated mass critical NLS problem 
\begin{equation} \label{nls}
(NLS)\ \ \left \{ \begin{array}{l} i \partial_t u +\Delta u+|u|^{\frac4N} u=0, \\
u(t_0,x) = u_0(x), \quad u : I \times \R^N \to \C. \end{array} 
\right . 
\end{equation}
From  Weinstein \cite{We1983}, we recall that initial data $u_0\in H^1(\R^N)$ with $\|u_0\|_{L^2}<\|Q\|_{L^2}$ yield global-in-time solutions where $Q$ is from \cite{GiNiNi1979}, \cite{Kw1989} the unique up to symmetries solution to the ground state equation $$\Delta Q-Q+Q^{1+\frac 4N}=0, \ \ Q(x)>0, \ \ Q\in H^1(\R^N).$$ Moreover, solutions with $$u_0\in H^1\cap  \{xu\in L^2\}, \ \ \|u_0\|_{L^2}<\|Q\|_{L^2}$$ scatter, i.\,e., they behave asymptotically like free waves, see \cite{Ca2003}, and this result has been extended to all $L^2$ data with subcritical mass using the Kenig-Merle road map \cite{KeMe2006} in \cite{Do2011,KiTaVi2009,KiViZh2008}.
At the mass critical level, the additional pseudo-conformal symmetry of \fref{nls} yields an explicit minimal blowup element: 
\be
\label{defst}
S(t,x)=\frac{1}{|t|^{\frac N2}}Q\left(\frac xt\right)e^{\frac{i}{t}}e^{\frac{i|x|^2}{4t}}, \ \ \|S(t)\|_{L^2}=\|Q\|_{L^2}.
\ee 
Merle obtains in \cite{Me1993} the classification in the energy space of minimal blowup elements: the only $H^1$ finite time blowup solution with mass $\|u\|_{L^2}=\|Q\|_{L^2}$ is given by \fref{defst} up to the symmetries of the flow.\\
The question of existence and possibly uniqueness of minimal blowup elements has since then been addressed in various settings. The existence of minimal elements can be obtained for (NLS) on a domain \cite{BuGeTz2003} through a brute force perturbative argument. Similar thresholds solutions have been derived for the energy critical problem \cite{DuMe2009} using the virial algebra and without the description of the associated blowup scenario. Then a robust dynamical approach for the proof of both existence and uniqueness has been developed by Rapha\"el, Szeftel \cite{RaSz2011}, for an inhomogeneous problem $$i \partial_t u +\Delta u+k(x)|u|^{\frac4N} u=0,$$ which is a canonical problem proposed by Merle \cite{Me1996} to break the pseudo-conformal symmetry, and which under suitable assumption on $k$ does not admit minimal blowup elements. The existence and uniqueness of minimal blowup elements in \cite{RaSz2011} is proved under sharp assumptions of $k$ which induce a dramatic influence on the bubble of concentration, and allows one to go beyond the perturbative case treated in \cite{BaReDu2011}. The argument involves a soft compactness argument using the reversibility of the flow as in \cite{Me1990}, \cite{Ma2005}, \cite{MaMe2006}, and a mixed Energy/Morawetz monotonicity formula available {\it at the minimal mass level} only to integrate the flow backwards from the singularity. The robustness of this approach and further developments led in \cite{MaMeRa2012b} to the construction of minimal elements for the mass critical gKdV problem $$\pa_tu+(u_{xx}+u^5)_x=0, \ \ (t,x)\in \R\times\R,$$ which was an open problem since the pioneering work \cite{MaMe2002}.


\subsection{Statement of the main results} 


We address in this paper the question of existence nondispersive dynamics, and we will describe two example of such dynamics: mass subcritical traveling solitary waves and minimal mass blowup solutions. In what follows, let $Q \in H^{1/2}(\R)$ be the unique ground state solution of \eqref{neoenoneo}.\\

A family of {\it mass subcritical} traveling solitary waves can be constructed using variational techniques and adapting the proof in \cite{FrJoLe2007}. Also, note that {\em no such elements exist} for the $L^2$-critical (NLS), since initial data with subcritical $L^2$-mass for \eqref{nls} always scatter to a free wave (see \cite{Do2011,KiTaVi2009,KiViZh2008}) and in particular no solitary waves with subcritical mass\footnote{because of the cancellation $\|Qe^{i\beta y}\|_{L^2}=\|Q\|_{L^2}$ for all $\beta \in \R^N$.} exist for (NLS). For the half-wave equation \eqref{eq:wave}, we have the following result.

\begin{thm}[Traveling solitary waves with arbitrarily small mass]
\label{thmone}
For all $|v|<1$, there exists a profile $Q_v \in H^{1/2}(\R)$ such that $$u(t,x)=e^{it} Q_v(x-vt)$$ is a traveling solitary wave solution to \fref{eq:wave}. Moreover, the mass $\| Q_v \|_{L^2}$ is strictly decreasing with respect to $|v|$ and for any $0< |v| < 1$, the profile $Q_v$ has strictly subcritical mass:
\be\label{nkonceno}
\|Q_v\|_{L^2}<\|Q\|_{L^2}.
\ee 
There also holds the limits: 
$$\left\{\begin{array}{ll} \| Q_v \|_{L^2} \to  \| Q \|_{L^2} & \ \ \mbox{as $|v| \to 0$},\\
\|Q_v\|_{L^2}\to 0 & \ \ \mbox{as $|v| \to 1$}.\end{array}\right.
$$
\end{thm}

A second example of nondispersive dynamics corresponds to a minimal mass singularity formation. The existence of blowup solutions for \fref{eq:wave} for which no simple obstruction to global existence like for \fref{nls} has been an open problem, and our claim is that we can adapt the strategy in \cite{RaSz2011} even though dispersion is nonlocal, and we can build through a dynamical argument minimal blowup elements with a complete description of the associated mass concentration scenario. The main result is the following.

\begin{thm}[Existence of minimal mass blowup elements]
\label{thmmain}
For all $(E_0,P_0)\in \R^*_+\times\R$, there exists $t^*<0$ and a minimal mass solution $u\in  C^0([t^*,0); H^{1/2}(\R))$ of equation \eqref{eq:wave} with $$\|u\|_{L^2}=\|Q\|_{L^2},  \ \ E(u)=E_0, \ \ P(u)=P_0,$$ which blows up at time $T=0$. More precisely, it holds that
\be
\label{neiovneone}
u(t,x)-\frac{1}{\l^{\frac12}(t)}Q\left(\frac{x-\alpha(t)}{\lambda(t)}\right)e^{i\gamma(t)}\to 0 \ \ \mbox{in} \ \ L^{2}(\R) \ \ \mbox{as} \ \ t\to 0^-,
\ee 
where 
$$
\lambda(t)=\l^* t^2 + \Or(t^3), \ \ \ \ \alpha(t) = \Or(t^3), \ \ \ \ \gamma(t)= \frac{1}{\l^*|t|} + \Or(t) , 
$$
with some constant $\l^* > 0$, and the blowup speed is given by:
$$
\| D^{\frac 1 2} u(t) \|_{L^2} \sim \frac{C(u_0)}{|t|} \ \ \mbox{as} \ \ t \to 0^-.
$$
\end{thm}

{\it Comments on the result.}\\

{\it 1. Extension}: Similar questions can be addressed for the generalized $L^2$-mass critical problem 
\be \label{eq:fNLS}
i \partial_t u = D^s u - |u|^{2s} u, \ \ (t,x)\in \R\times \R,
\ee
with fractional power $1 < s < 2$. Since nondegeneracy (and uniqueness) of ground states is also known in this case (see \cite{FrLe2012}), we claim that our construction of minimal blow up solution carries over verbatim (except for some technicalities when the nonlinearity $|u|^{2s} u$ fails to be smooth). However, the case $s=1$ treated here is critical with respect to many aspects of the problem; in particular, the absence of any smoothing properties for the propagator $e^{-it D}$ is a delicate issue. For equation \eqref{eq:fNLS}, we claim that the associated minimal elements would concentrate an $L^2$ bubble \fref{neiovneone} at the speed $$\l(t)=\l^*|t|^{\frac{2}{s}}.$$ The analysis could also in principle be extended to the higher dimensional case, provided that the ground state are known to be nondegenerate; see \cite{FrLeSi2012} for a recent result in $N \geq 2$ space dimensions.\\

{\it 2. On minimal elements}: Theorem \ref{thmone} shows that scattering does not occur below the ground state. This is maybe not so surprising for the half-wave which is a one dimensional like wave equation. However, the variational setting for the construction of traveling waves with strictly subcritical mass \fref{nkonceno} can be adapted to the case $1<s<2$. This shows a major difference with the mass critical (NLS) with local dispersion $s=2$, and in particular that the sharp threshold for global existence and the sharp threshold for scattering are not the same.\\

{\it 3. Role of the momentum}: The construction of minimal elements with nonzero linear momentum is a nontrivial task, since equation \eqref{eq:wave} neither has  Galilean boost symmetry (which is an essential feature of \eqref{nls}) nor does our problem exhibit Lorentz boost symmetry (which occurs for classical nonlinear wave equation). To overcome this lack of symmetries to generate solutions of uniform motion, we construct {\it boosted ground state profiles}  for equation \eqref{eq:wave} by a suitable ansatz that incorporates a velocity parameter $v$ of uniform motion. Let us stress that these boosted ground states have indeed a strictly subcritical $L^2$-mass. As a consequence, the key is to compute the motion of the generalized boost parameter $v$ and to realize that in the regime, we are working with, it asymptotically vanishes sufficiently fast and hence does not perturb the concentration dynamics. A similar issue occurred in \cite{RaSz2011}.\\

{\it 4. Structure of the ground state}: An important qualitative difference between the local problem \fref{nls} and the nonlocal problem \fref{eq:wave} is the structure of the ground state solitary wave $Q$ which decays exponentially for \fref{nls}, while for the half-wave equation \eqref{eq:wave} the ground state exhibits a slow algebraic decay:
$$
Q(x) \sim \langle x \rangle^{-2} \ \ \mbox{as}\ \ |x|\to +\infty.
$$ 
Also, the linearized operator close to $Q$ displays a nonlocal dispersion, which makes the use of spectral estimates as in \cite{KrSc2009} particularly delicate. Here we will use two important facts. In \cite{FrLe2012}, despite the nonlocal structure of the problem, the quadratic form associated to the linearized Hamiltonian is proved to be {\it nondegenerate}, and this is in fact an important step of the proof of uniqueness of the ground state. This nondegeneracy itself is then an essential ingredient to adapt the strategy in \cite{{RaSz2011}} for the construction of minimal elements, which does not require any further spectral information --- like virial-type coercivity as in \cite{MeRa2005,MaMeRa2012a}-.\\

{\it 5. Bourgain--Wang solutions}: In \cite{BoWa1997}, Bourgain and Wang show that the minimal blowup element $S(t)$ given by \fref{defst} for the local problem \fref{nls} can be used to construct mass super critical blowup solutions whose singular part is given to leading order by $S(t)$, see also \cite{KrSc2009}. These solutions are shown to be unstable by ``log-log" blowup and scattering in \cite{MeRaSz2010}. The extension of this result to the case of the $L^2$-critical half-wave equation (i.\,e.~the construction of similar threshold dynamics based on the minimal element) is a natural question in the continuation of this work.\\

In the present work, our aim is to present a robust and self-contained construction of minimal blowup elements in a setting of nonlocal dispersion. Moreover, we believe that the arguments developed here will be of broader interest in the further understanding of blowup phenomena of PDE with fractional powers of the Laplacian.\\

There are three major questions in the continuation of this work. First, the question of uniqueness (modulo symmetries) of minimal mass blowup elements is a delicate open problem for equation \eqref{eq:wave}, and for which we further hope to extend the strategy developed in \cite{RaSz2011, MaMeRa2012a} to the half-wave problem. Second, one can ask for the behavior of the minimal blowup element on the left in time, and one typically expects that the minimal mass blowup element is a connection between scattering at $-\infty$ and blowup in finite time on the right. Again, this is a non trivial claim in the absence of an explicit formula like \fref{defst}, and the solutions of Theorem \ref{thmmain} are constructed locally in time only around blowup. This question relates directly to the description of the phase portrait of the flow around the ground state $Q$, and the understanding of threshold dynamics, see \cite{MeRaSz2010,MaMeRa2012a,MaMeRa2012b}. Finally, the understanding of the flow below the ground state mass in the presence of arbitrarily small solitary waves is a very interesting problem.\\

{\bf Notation and Definitions}. We use $D^s$ to denote the fractional derivative of order $s \geq 0$, i.\,e., we set $$\widehat{(D^s f)}(\xi) = |\xi|^{s} \widehat{f}(\xi).$$  We employ standard notation for $L^p$-spaces and we use $$(f,g) = \int \bar{f} g$$ as the inner product on $L^2(\R)$. We shall use $X \lesssim Y$ to denote that $X \leq CY$ holds, where the constant $C >0$ may change from line to line, but $C$ is allowed to depend on universally fixed quantities only. Likewise, we use $X \sim Y$ to denote that both $X \lesssim Y$ and $Y \lesssim X$ hold. Furthermore, we use $X \lesssim_\alpha Y$ to denote that $X \leq C_\alpha Y$ where the constant $C_\alpha >0$ is also allowed to depend on some quantity $\alpha$.

For a sufficiently regular function $f : \R \to \C$, we define the generator of $L^2$ scaling given by
\begin{equation*}
\Lambda f := \frac 1 2 f + x f'.
\end{equation*}
Note that the operator $\Lambda$ is skew-adjoint on $L^2(\R)$, i.\,e., we have
\begin{equation*}
(\Lambda f, g) = - (f, \Lambda g).
\end{equation*} 
We write $\Lambda^k f$, with $k\in \N$, for the iterates of $\Lambda$ with the convention that $\Lambda^0 f \equiv f$. In the following, we sometimes use the multi-variable calculus notation
$$
\nabla f = f', \quad \Delta f = f'' 
$$
for functions $f : \R \to \R$ to improve the readability of certain formulae derived below.

In some parts of this paper, it will be convenient to identify any complex-valued function $f : \R \to \C$ with the function $\bm{f} : \R \to \R^2$ by setting
\begin{equation*}
\bm{f} = \left [ { f_1 \atop f_2 } \right  ] = \left [ {\Re f \atop \Im f} \right ] .
\end{equation*} 
Correspondingly, we will identify the multiplication by $i$ in $\C$ with the multiplication by the real $2 \times 2$-matrix defined as
\begin{equation*}
J = \left [ \begin{array}{cc} 0 & -1 \\ 1 & 0 \end{array} \right ] .
\end{equation*}
In what follows, regularity properties such as $f \in H^k(\R)$ (as a $\C$-valued function) are obviously equivalent to saying that $\bm{f} \in H^k(\R)$ (as a $\R^2$-valued function). Furthermore, the action of differential operators (such as $\nabla$, $\Lambda$ and $D^s$ etc.) on $\bm{f}$ is defined in a self-evident fashion. 

Throughout this paper, we denote the linearized operator (with respect to complex-valued functions) close to the ground state $Q$ by
$$
L = \left [ \begin{array}{cc} L_+ & 0 \\ 0 & L_- \end{array} \right ] ,  
$$ 
with the scalar self-adjoint operators 
$$
L_+ = D + 1 - 3Q^2, \quad L_-=D + 1 - Q^2,
$$
acting on $L^2(\R; \R)$. 

\subsection*{Acknowledgments}
J.\,K.~acknowledges support from the Swiss National Science Foundation (SNF). Part of this work was done while E.\,L.~was partially supported by a Steno research fellowship from the Danish Research Council. P.\,R. is supported by the French ERC/ANR project SWAP and the ERC advanced grant BLOWDISOL. Part of this work was done while P.\,R. was visiting the Department of Mathematics at MIT, which he would like to thank for its kind hospitality.

\section{Traveling Solitary Waves with Subcritical Mass}
\label{sec:Qboost}

In this section we prove Theorem \ref{thmone}, which establishes the existence and properties of traveling solitary waves for \eqref{eq:wave}. In particular, we will see that traveling solitary waves with arbitrarily small $L^2$-mass exist, which is in striking contrast to the $L^2$-critical NLS.

\subsection{Preliminaries}

Let $v \in \R$ with $|v| < 1$ be given. By making the ansatz $u(t,x) = e^{it} Q_v(x-vt)$ for \eqref{eq:wave}, we find that the profile $Q_v \in H^{1/2}(\R)$ has to satisfy
\be \label{eq:Qboost}
D Q_v + Q_v+ i (v \cdot \nabla) Q_v - |Q_v|^2 Q_v = 0. 
\ee
Following an idea in \cite{FrJoLe2007}, we obtain nontrivial solutions $Q_v \in H^{1/2}(\R)$ as optimizers for the interpolation inequality
\be \label{ineq:GN}
\int |u|^4 \leq C_v \left ( \int \overline{u} D u + \overline{u} (i v \cdot \nabla u)  \right ) \left ( \int |u|^2 \right ) .
\ee 
Note that $|v| < 1$ is needed to ensure that $\int \overline{u} D u + \overline{u} (iv \cdot \nabla u) > 0$ for $u \not \equiv 0$. Here $C_v >0$ denotes the optimal constant given by
\be \label{eq:wein}
\frac{1}{C_v} = \inf_{u \in H^{1/2}(\R) \setminus \{ 0 \}} \frac{ \left ( \int \overline{u} D u + \overline{u} (i v \cdot \nabla u)  \right ) \left ( \int |u|^2 \right )}{\int |u|^4} .
\ee
By Sobolev inequalities, we see that the infimum on the right is strictly positive (and hence $C_v < +\infty$ is finite). Furthermore, the fact this infimum is indeed attained can be deduced from concentration-compactness arguments, which in our case follow from a direct adaptation of the proof given in \cite[Appendix B]{FrJoLe2007}. In particular, optimizers $Q_v \in H^{1/2}(\R)$ for \eqref{ineq:GN} exist and after a suitable rescaling $Q_v(x) \mapsto a Q_v(bx)$ with $a, b > 0$ they are found to satisfy equation \eqref{eq:Qboost}. Following the terminology introduced in \cite{FrJoLe2007}, we refer to such optimizers $Q_v(x)$ that solve equation \eqref{eq:Qboost} as {\em boosted ground states} (with velocity $v$) in what follows. In particular, the unboosted ground states $Q_{v=0}(x)=Q(x)$ is the unique (modulo symmetries) ground state solving \eqref{neoenoneo} above. Finally, we observe that
\be \label{eq:CvQ}
\frac{2}{C_v} = \int |Q_v|^2 ,
\ee
which follows from the fact $Q_v$ optimizes \eqref{ineq:GN} and satisfies equation \eqref{eq:Qboost}; see more details on this relation for a similar problem treated in \cite{FrJoLe2007}. In particular, the relation \eqref{eq:CvQ} shows that two different boosted ground states $Q_v$ and $\widetilde{Q}_v$ with the same velocity $v$ must satisfy $\| Q_v \|_{L^2} = \| \widetilde{Q}_v \|_{L^2}$.\\
We may reformulate \fref{eq:CvQ} as follows. Let the energy functional 
$$
\cE_v(u) = \frac 1 2 \int \overline{u} D u + \frac{1}{2} \int \overline{u} (i v \cdot \nabla u) - \frac 1 4 \int |u|^4,
$$
then\footnote{as follows from a standard Pohozaev integration by parts on \fref{eq:Qboost}.} 
\be
\label{cnonenoe}\cE_v(Q_v)=0
\ee
and there holds the sharp Gagliardo-Nirenberg interpolation inequality:
\be 
\label{sahrpgagl}
\forall u\in H^{\frac 12},  \ \ \cE_v(u) \geq \frac{1}{2} \left ( \int \left \{ \overline{u} D u + \overline{u} (i v \cdot \nabla u) \right \} \right ) \left ( 1 - \frac{ \| u \|_{L^2}^2}{\| Q_v \|_{L^2}^2} \right ).
\ee

\subsection{Proof of Theorem \ref{thmone}}

Let $v \in \R$ with $|v| < 1$ be given. From the previous paragraph we know that boosted ground states $Q_v$ satisfying equation \eqref{eq:Qboost} exist. Due to the behavior of the problem under spatial reflections $x \mapsto -x$, we can assume without loss generality that all velocities are positive numbers, i.\,e.,
\be
0 \leq v < 1.
\ee

{\bf step1} Sign of the momentum. Let $0\leq v<1$. We claim: 
\be \label{eq:mom}
v \cdot \int \overline{Q_v} (i \nabla Q_v) \leq 0 .
\ee
Indeed, assume on the contrary that $v \cdot \int \overline{Q_v} (i \nabla Q_v) > 0$ holds. We define the reflected function $\widetilde{Q}_v(x) := Q_v(-x)$. Note that $\int |\widetilde{Q}_v|^2 = \int |Q_v|^2$ and $v \cdot \int \overline{\widetilde{Q}_v} (i \nabla \widetilde{Q}_v) < 0$. Since the remaining terms in $\cE_v(u)$ are invariant with respect to space reflections, we find that $\cE_v(\widetilde{Q}_v) < \cE_v(Q_v)=0$. But $\|\widetilde{Q}_v\|_{L^2}=\|Q\|_{L^2}$ implies $\cE_v(\widetilde{Q}_v)\geq 0$ from \fref{sahrpgagl}, contradiction. We conclude that \eqref{eq:mom} holds. In particular, we see that
\be \label{eq:mom2} 
\int \overline{Q_v} (i \nabla Q_v) \leq 0 \quad \mbox{for $0 < v < 1$}. 
\ee
For the case $v=0$, we recall the fact from \cite{FrLe2012} that (after translation and shift by a complex constant phase) the functions $Q_{v=0}=Q_{v=0}(|x|)$ is even. Hence, in this special case, we have
\be \label{eq:mom3}
\int \overline{Q_{v=0}} (i \nabla Q_{v=0}) = 0.
\ee

{\bf step 2} The mass is non increasing. We claim the monotoncity:
\be
\label{cneonoenvjojeo}
\| Q_{v_2} \|_{L^2} < \| Q_{v_1} \|_{L^2}\ \ \mbox{for}\ \ 0\leq v_1<v_2<1.
\ee Note that this implies in particular the subcritical mass property: $$\|Q_v\|_{L^2}<\|Q\|_{L^2}\ \ \mbox{for}\ \ 0<v<1.$$
Indeed, let $Q_{v_1}$ and $Q_{v_2}$ be two boosted ground states satisfying \eqref{eq:Qboost} with $v=v_1$ and $v=v_2$, respectively. Since $\cE_{v_1}(Q_{v_1}) = 0$ by \eqref{cnonenoe}, we find using \eqref{eq:mom2} if $v_1 > 0$ and \eqref{eq:mom3} if $v_1=0$ that
\begin{align*}
\cE_{v_2}(Q_{v_1}) & = \cE_{v_1}(Q_{v_1}) + (v_2-v_1) \cdot \int \overline{Q_{v_1}} (i \nabla Q_{v_1}) \leq 0,
\end{align*}
since $v_2-v_1 > 0$ by assumption, which together with \fref{sahrpgagl} implies $\|Q_{v_1}\|_{L^2}\geq \|Q_{v_2}\|_{L^2}$. In case of equality $\|Q_{v_1}\|_{L^2}= \|Q_{v_2}\|_{L^2}$, $Q_{v_1}$ attains the minimization problem \fref{eq:wein} with $v_2$. In particular, the function $Q_{v_1}$ satisfies the equation
\be
D Q_{v_1} + \lambda Q_{v_1} + v_2 \cdot \nabla Q_{v_1} - |Q_{v_1}|^2 Q_{v_1} = 0,
\ee
with some Lagrange multiplier $\lambda \in \R$. On the other hand, by assumption, the boosted ground state $Q_{v_1}$ also satisfies equation \eqref{eq:Qboost} with $v=v_1$. By subtracting the equations satisfied by $Q_{v_1}$, we obtain that
\be
(\lambda- 1) Q_{v_1} + (v_2-v_1) \cdot \nabla Q_{v_1} = 0.
\ee
Since $v_2 \neq v_1$ by assumption and $Q_{v_1}(x) \to 0$ as $|x| \to +\infty$, we deduce from this equation that $Q_{v_1} \equiv 0$ holds, which is absurd.\\

{\bf step 3} Limits. We claim:
$$\left\{\begin{array}{ll} \| Q_v \|_{L^2} \to  \| Q \|_{L^2} & \ \ \mbox{as $|v| \to 0$},\\
\|Q_v\|_{L^2}\to 0 & \ \  \mbox{as $|v| \to 1$}.\end{array}\right.
$$
To show that $\| Q_v \|_{L^2} \to \| Q \|_{L^2}$ as $v \to 0$, we argue as follows. From $|Ê\xi | - v \cdot \xi \geq (1-|v|) | \xi |$ for $\xi \in \R$ and Plancherel's identity, we deduce that $C_v \leq (1-|v|)^{-1} C_{v=0}$ for the optimal constants in \eqref{ineq:GN}. From this simple bound and recalling \eqref{eq:CvQ} and the monotonicity \fref{cneonoenvjojeo}, we deduce the bounds
$$
 \sqrt{1-|v|} \| Q \|_{L^2} \leq \| Q_v \|_{L^2} \leq \| Q \|_{L^2} .
$$
whence it follows that $\| Q_v \|_{L^2} \to \| Q \|_{L^2}$ as $v \to 0$.

It remains to show that $\|Q_v \|_{L^2} \to 0$ as $|v| \to 1$. It suffices to prove this claim for $v \to 1$, since $v \to -1$ can be treated in a verbatim way. Let $\varphi \in H^{1/2}(\R)$ with $\varphi \not \equiv 0$ have only positive Fourier components, i.\,e., we assume that $\mathrm{supp} \, \widehat{\varphi} \subset [0, +\infty)$ holds. For $v > 0$, this gives us
\be 
(D + i v \cdot \nabla) \varphi = (1-v) D \varphi. 
\ee
From \eqref{ineq:GN} we obtain that
\be
C_v \geq \left ( \frac{1}{1-v} \right ) \left ( \frac{\int |\varphi|^4 }{\left ( \int \overline{\varphi} D \varphi  \right ) \left ( \int |\varphi|^2 \right ) } \right ) .
\ee
Therefore $C_v \to +\infty$ as $v \to 1$. In view of \eqref{eq:CvQ}, this shows that $\| Q_v \|_{L^2} \to 0$ as $v \to 1$.\\
The proof of Theorem \ref{thmone} is now complete. \hfill $\blacksquare$

\begin{remark} \label{rem:QvCV}
By uniqueness of the ground state $Q$ and a concentration-compactness argument, one can show from standard arguments that if $v_n \to 0$ then (after possibly passing to a subsequence):
\be \label{eq:QvCVgood}
 \mbox{$e^{i \gamma_n} Q_{v_n}( \cdot + y_n) \to Q$ in $H^{1/2}(\R)$ as $n \to +\infty$},
\ee
for some sequences $\{ \gamma_n \}_{n=1}^\infty, \{ y_n \}_{n=1}^\infty$ in $\R$.
\end{remark}

For the reader's convenience, we sketch the arguments showing the convergence claim \eqref{eq:QvCVgood} above. For $|v| < 1$, we define the functional
\be
\cJ_v(u) = \frac{ \left ( \int \overline{u} D u + \overline{u} (i v \cdot \nabla u)  \right ) \left ( \int |u|^2 \right )}{\int |u|^4}   ,
\ee
for $u \in H^{1/2}(\R)$ with $u \not \equiv 0$.  Adapting the proof in \cite[Appendix B]{FrJoLe2007}, we see that every minimizing sequence for $\cJ_v(u)$ is relatively compact in $H^{1/2}(\R)$ up to translations and scalings. Moreover, as shown in \cite{FrLe2012}, the functional $\cJ_{v=0}(u)$ has a unique (modulo symmetries) minimizer $Q$, which is the unique ground state solution satisfying \eqref{neoenoneo}. Therefore if $\{ u_n \}_{n=1}^\infty \subset H^{1/2}(\R) \setminus \{ 0 \}$ is a minimizing sequence for $\cJ_{v=0}(u)$, then (after passing to a subsequence if necessary):
\be \label{eq:relcom}
\mbox{$a_n u_n( b_n( \cdot + y_n ) ) \to Q$ in $H^{1/2}(\R)$ as $n \to +\infty$},
\ee
for some sequences $\{ a_n \}_{n=1}^\infty \subset \C \setminus \{ 0 \}$, $\{ b_n \}_{n=1}^\infty \subset \R \setminus \{ 0 \}$ and $\{ y_n \}_{n=1}^\infty \subset \R$.

Now, we suppose that $v_n \to 0$ and let $\{ Q_{v_n} \}_{n=1}^\infty$ be a sequence of boosted ground states. Note that
$$
\cJ_{v=0}(Q_{v_n}) = \frac{ \int \overline{Q_{v_n}} D Q_{v_n}}{ \int \left (\overline{Q_{v_n}} D Q_{v_n} + \overline{Q_{v_n}} (i v \cdot \nabla Q_{v_n}) \right ) } \cJ_{v_n}(Q_{v_n})  \leq \frac{1}{1-|v_n|} \frac{2}{\| Q_v \|_{L^2}^2} ,
$$
using that $|\xi| - v\cdot \xi \geq (1-|v|) |\xi|$ and that $Q_v$ minimizes $\cJ_{v_n}(u)$ and \eqref{eq:CvQ}. On the other hand, we have the obvious lower bound $J_{v=0}(Q_{v_n}) \geq  J(Q) = 2/\| Q \|_{L^2}^2$. Since $\| Q_v \|_{L^2} \to \| Q \|_{L^2}$ as $v \to 0$, we conclude that
$$
\cJ_{v=0}(Q_{v_n}) \to \cJ_{v=0}(Q) \ \ \mbox{as} \ \ n \to +\infty. 
$$
Therefore $\{ Q_{v_n} \}_{n=1}^\infty$ furnishes a minimizing sequence for $\cJ_{v=0}(u)$. From \eqref{eq:relcom} and using the normalization constraints satisfied by $Q_{v_n}$ (to see that $|a_n| = |b_n| =1$), we deduce that \eqref{eq:QvCVgood} holds true.


\section{Sketch of the Proof of Theorem \ref{thmmain}}


Before we start our analysis, let us make some formal remarks. To construct minimal mass blowup solutions for problem \eqref{eq:wave}, we first renormalize the flow
$$u(t,x) = \frac{1}{\lambda^{\frac 1 2}(t)} v\left (t, \frac{x- \alpha(t)}{\lambda(t)} \right ) e^{i \gamma(t)} ,
$$
which leads the renormalized equation:
\be
\label{equation*}
i \partial_s v - D v - v + v |v|^2  = i \frac{\lambda_s}{\lambda} \Lambda v + i \frac{\alpha_s}{\lambda} \cdot \nabla v + \tgamma_s v.
\ee
Following the slow modulated ansatz strategy developed in \cite{MeRa2003,RaSz2011,KrMaRa2009}, we freeze the modulation equations $$-\lsl=b, \ \ \frac{\alpha_s}{\lambda}=v, $$and we look for an approximate solution of the form: $$v(s,y)=P_{\matchal P(s)}, \ \ \matchal P(s)=(b(s),v(s))$$ with an expansion:
 $$b_s=P_1(b,v), \ \ v_s=P_2(b,v), \ \ Q_\mathcal P=Q(y)+\Sigma_{|\alpha|+\beta\geq 1}v^{\alpha}b^{\beta}P_{\alpha,\beta}(y).$$ Each step requires inverting an elliptic system of the form of the form $Lu=f$, where $L=(L_+,L_-)$ is the matrix linearized operator close to $Q$ which displays a nontrivial kernel induced by the symmetry group. We adjust the modulation equation for $(b_s,v_s)$ to ensure the solvability of the obtained system, and a specific algebra leads to the laws to leading order: 
 $$b_s=-\frac12b^2, \ \ v_s=-b v.$$ This allows us to construct a high order approximation $Q_{\matchal P}$ solution to 
\begin{equation*}
- i \frac{b^2}{2}  \partial_b \Qp - i bv \partial_v \Qp  - D \Qp - \Qp + i b \Lambda \Qp - i v \cdot \nabla \Qp + |\Qp|^2 \Qp =  -\Psi_\Pa,
\end{equation*}
where $\Psi_\Pa = \Or(b^5 + v \Pa^2)$ is some small and well-localized error term. Furthermore, we have that the $\Qp$ has almost minimal mass in the sense that
\begin{equation*}
\int |\Qp|^2 = \int Q^2 + \Or(b^4 + v^2 + v^2 \Pa).
\end{equation*}
We now aim at constructing an exact solution of the form $$u(t,x) = \frac{1}{\lambda^{\frac 1 2}(t)} \left[Q_{\mathcal P(t)}+\e\right]v\left (t, \frac{x- \alpha(t)}{\lambda(t)} \right ) e^{i \gamma(t)} ,
$$
and this amounts propagating suitable dispersive estimates for $\e$. Here, a key ingredient will be a backwards monotonicity mixed energy/virial estimate which schematically yields the bound
\begin{equation*}
\frac{d}{dt} \left \{  \frac{1}{\l}\left[\| D^{\frac 1 2} \e \|_{L^2}^2 + \| \tilde{\e} \|^2_{L^2}  + b \Im \left ( \int_{ |y| \lesssim 1} y \cdot \nabla \tilde{\e} \overline{\tilde{\e}} \right )\right] \right  \} \geq 0 + \mbox{lower order terms},
\end{equation*}
where the monotonicity {\it in the critical mass regime} relies on the coercivity of the linearized energy only.
Using the above backwards monotonicity, we can bootstrap and apply a soft compactness argument to construct solutions of the form above such that
\begin{equation*}
\lambda \sim t^2, \quad b \sim t, \quad v \sim t^2, \quad \| \eps(t) \|_{H^{1/2}}^2 \sim t^2.
\end{equation*}
In particular, we deduce that the blowup solutions have minimal mass $\| u_0 \|_{L^2} = \| Q \|_{L^2}$, energy $E(u_0) = E_0$, momentum $P(u_0) = P_0$, and a blowup rate given by
\begin{equation*}
 \| D^{\frac 1 2} u(t) \|_{L^2} \sim \frac{C(u_0)}{|t|} \ \ \mbox{as} \ \ t \to 0^-.
\end{equation*}

In the following Sections \ref{sec:profile}--\ref{sec:existence_min}, we will implement the strategy sketched above. Finally, in Section \ref{sec:existence_min} below, we will state and prove Theorem \ref{thm:exist}, which in particular yields Theorem \ref{thmmain}.

\section{Approximate Blowup Profile}
\label{sec:profile}
This section is devoted to the construction of the approximate blowup profile $\Qp$ with parameters $\Pa = (b,v)$. In what follows, it will be convenient to identify a complex-valued function $f : \R \to \C$ with the function $\bm{f} : \R \to \R^2$ through $\bm{f} = [\Re f, \Im f]^\top$, as we have already mentioned above. 
Correspondingly, we will identify the multiplication by $i$ in $\C$ with the multiplication by the real $2 \times 2$-matrix
\begin{equation*}
J = \left [ \begin{array}{cc} 0 & -1 \\ 1 & 0 \end{array} \right ] .
\end{equation*}

Employing this notation, we have the following result about an approximate blowup profile $\Qbp$, parameterized by $\Pa=(b,v)$, around the ground state $\bm{Q} = [Q, 0]^\top$.

\begin{prop}[Approximate Blowup Profile] \label{prop:Qb_existence}
Let $\Pa=(b,v) \in \R \times \R$. There exists a smooth function $\Qbp= \Qbp(x)$ of the form
\begin{equation} \label{eq:Qbansatz}
\Qbp = \bm{Q} + b \bm{R}_{1,0} + v \bm{R}_{0,1} + bv \bm{R}_{1,1} + b^2 \bm{R}_{2,0} + v^2 \bm{R}_{0,2} + b^3 \bm{R}_{3,0} + b^2 v \bm{R}_{2,1} + b^4 \bm{R}_{4,0}
\end{equation}
that  satisfies the equation
\begin{equation} \label{eq:Qbdef}
- J \frac{1}{2} b^2 \partial_b \Qbp - J bv \partial_v \Qbp - D \Qbp - \Qbp + J b \Lambda \Qbp - J v \cdot \nabla \Qbp + |\Qbp|^2 \Qbp = - \bm{\Psi}_\Pa.
\end{equation}
Here, the functions $\{ \bm{R}_{k, \ell} \}_{0 \leq k \leq 3, 0 \leq \ell \leq 1}$ satisfy the following regularity and decay bounds:
\begin{equation} \label{ineq:Qbbound1}
\| \bm{R}_{k, \ell} \|_{H^m} + \| \Lambda \bm{R}_{k,\ell} \|_{H^m} + \| \Lambda^2 \bm{R}_{k, \ell} \|_{H^m} \lesssim_m 1,  \quad \mbox{for $m \in \N$},
\end{equation}
\begin{equation} \label{ineq:Qbbound2}
| \bm{R}_{k,\ell}(x) | + | \Lambda \bm{R}_{k, \ell}(x)| + | \Lambda^2 \bm{R}_{k,\ell}(x)| \lesssim \langle x \rangle^{-2}, \quad \mbox{for $x \in \R$}.
\end{equation}
Moreover, the term on the right-hand side in \eqref{eq:Qbdef} satisfies
\begin{equation} \label{ineq:Qberror}
\| \bm{\Psi}_\Pa \|_{H^m} \lesssim_m \Or \left (b^5 +  v^2 \Pa \right ), \quad |\nabla^k \bm{\Psi}_\Pa(x)| \lesssim \Or \left ( b^5 +  v^2 \Pa \right ) \langle x \rangle^{-2} ,
\end{equation}
for $m \in \N$ and $x \in \R$.
\end{prop}

\begin{remark}
The proof of Proposition \ref{prop:Qb_existence} will actually show that the functions $\{ \bm{R}_{k, \ell} \}$ have the following symmetry structure:
$$
\bm{R}_{1,0} = \left [ {0 \atop even} \right ], \quad \bm{R}_{0,1} = \left [ {0 \atop odd} \right ], \quad \bm{R}_{1,1} = \left [ {odd \atop 0} \right ],
$$
$$
\bm{R}_{2,0} = \left [ {even \atop 0} \right ], \quad \bm{R}_{0,2} = \left [ {even \atop 0} \right ] , \quad \bm{R}_{3,0} = \left [ {0 \atop even} \right ],
$$
$$
\bm{R}_{2,1} = \left [ {0 \atop odd} \right ],  \quad \bm{R}_{4,0} = \left [ {even \atop 0} \right ] . 
$$
These symmetry properties will be of essential use throughout the following.
\end{remark}

\begin{proof}
We recall the definition of the linear operator 
\begin{equation}
L = \left [ \begin{array}{cc} L_+ & 0 \\ 0 & L_- \end{array} \right ] 
\end{equation} 
acting on $L^2(\R; \R^2)$, where $L_+$ and $L_-$ denote the unbounded operators acting on $L^2(\R;Ê\R)$ given by
\begin{equation}
L_+ = D +1 - 3Q^2, \quad L_- = D+1 - Q^2 .
\end{equation}
From [FrLe] we have the key property that the kernel of $L$ is given by
\begin{equation}
\mathrm{ker} \, L = \mathrm{span} \, \left \{  \left [ { \nabla Q \atop 0 } \right ] , \left [ { 0 \atop Q } \right ] \right \}. 
\end{equation}
Note also that the bounded inverse $L^{-1}=\mathrm{diag}(L_+^{-1}, L_-^{-1})$ exists on the orthogonal complement $\{ \mathrm{ker} \, L \}^\perp = \{ \nabla Q \}^\perp \oplus \{ Q \}^\perp$.

Next, let $\Qbp$ be given by \eqref{eq:Qbansatz} with the functions $\{ \bm{R}_{k,\ell} \}$ to be determined such that
\begin{equation*}
\mbox{LHS of \eqref{eq:Qbdef}} = \Or \left (b^5 + v^2 \Pa \right ) .
\end{equation*}
We divide the rest of the proof of Proposition \ref{prop:Qb_existence} as follows.

\medskip
{\bf Step 1: Determining the functions $\{ \bm{R}_{k,\ell} \}$.} 

We discuss our ansatz for $\Qbp$ to solve \eqref{eq:Qbdef} order by order. The proof of the regularity and decay bounds for the functions $\{ \bm{R}_{k, \ell} \}$ will be given further below (which, in particular, will guarantee that the following calculations are rigorous).

\medskip
{\bf Order $\Or(1)$:} Clearly, we have that 
$$
D \bm{Q} + \bm{Q} - |\bm{Q}|^2 \bm{Q} = 0,
$$
since $\bm{Q} = [ Q, 0]^\top$ with $Q=Q(|x|)>0$ being the ground state solution.

\medskip
{\bf Order $\Or(b)$:} We obtain the equation
\begin{equation}
L \bm{R}_{1,0} = J \Lambda \bm{ Q }.
\end{equation}
Note that $J \Lambda \bm{Q} = [ 0, \Lambda Q]^\top$ satisfies $J \Lambda \bm{Q} \perp \mathrm{ker} \, L$ due to the fact that $(\Lambda Q, Q) =0$, which can be easily seen by using the $L^2$-criticality. Hence we can find a unique solution $\bm{R}_{1,0} \perp \mathrm{ker} \, L$ to the equation above. In what follows, we denote 
\begin{equation}
\bm{R}_{1,0} = L^{-1}  J \Lambda \bm{Q}= \left [ { 0 \atop L_-^{-1} \Lambda Q} \right ] . 
\end{equation}

\medskip
 {\bf Order $\Or(v)$:} Here we need to solve 
 \begin{equation}
 L \bm{R}_{0,1} = - J \nabla  \bm{Q} .
 \end{equation}
We observe the orthogonality $J \nabla \bm{Q} = [0, \nabla Q]^\top \perp \mathrm{ker} \, L$, since $(\nabla Q, Q) = 0$ holds. Thus there is a unique solution $\bm{R}_{0,1} \perp \mathrm{ker} \, L$, which we denote as
\begin{equation}
\bm{R}_{0,1} = -L^{-1} J \nabla \bm{Q} = \left [ { 0 \atop -L_-^{-1} \nabla Q } \right ] .
\end{equation} 

\medskip
{\bf Order $\Or(bv)$:} First, we note that $\bm{Q} \cdot \bm{R}_{1,0} = \bm{Q} \cdot \bm{R}_{0,1} = 0$. Using this, we find that $\bm{R}_{1,1}$ has to solve the equation
\begin{equation} \label{eq:R11}
L \bm{R}_{1,1} = - J \bm{R}_{0,1} + J \Lambda \bm{R}_{0,1}  - J \nabla \bm{R}_{1,0} + 2 (\bm{R}_{1,0} \cdot \bm{R}_{0,1}) \bm{Q} .
\end{equation}
Now, we claim that 
\begin{equation} \label{eq:R11ortho}
\mbox{RHS of \eqref{eq:R11}} \perp \mathrm{ker} \, L .
\end{equation}
Indeed, we note that 
\begin{equation} \label{eq:R10_sol}
\bm{R}_{1,0} = \left [ {0 \atop S_1} \right ], \quad \mbox{with $L_- S_1 = \Lambda Q$},
\end{equation}
\begin{equation} \label{eq:R11_sol}
 \bm{R}_{0,1} = \left [ {0 \atop G_1} \right ], \quad \mbox{with $L_- G_1 = -\nabla Q$}. 
\end{equation}
Therefore the condition \eqref{eq:R11ortho} is equivalent to
\begin{equation} \label{eq:R11_magic}
(\nabla Q, G_1) - (\nabla Q, \Lambda G_1)  + (\nabla Q, \nabla S_1) + 2(\nabla Q, S_1 G_1 Q) = 0.
\end{equation}
To see that this holds true, we argue as follows. Using the commutator formula $[\Lambda, \nabla] = -\nabla$ and integrating by parts, we obtain
\begin{align} 
- (\nabla Q, \Lambda G_1) & = (\Lambda \nabla Q, G_1 ) = ( \nabla  \Lambda Q, G_1) - (\nabla Q, G_1 ) \nonumber \\
& = (\nabla L_- S_1, G_1 ) - (\nabla Q, G_1 ), \label{eq:R11_magic1}
\end{align}
Next, since $L_-$ is self-adjoint, we observe that
\begin{align}
(\nabla L_- F_1, G_1) + (\nabla Q, \nabla F_1) & = - (L_- F_1, \nabla G_1) - (L_- G_1, \nabla F_1) = (F_1, [\nabla, L_-] G_1) \nonumber \\
& = -  (F_1, (\nabla Q^2) G_1) = -2 (\nabla Q, F_1 G_1 Q) . \label{eq:R11_magic2}
\end{align}
By combining \eqref{eq:R11_magic1} and \eqref{eq:R11_magic2}, we conclude that \eqref{eq:R11_magic} holds. This shows that \eqref{eq:R11ortho} holds, and hence there is a unique solution $\bm{R}_{1,1} \perp \mathrm{ker} \,  L$ of equation \eqref{eq:R11}. Moreover, since $Q$ and $F_1$ are even functions whereas $G_1$ is odd, we note that 
\begin{equation}
\bm{R}_{1,1} = \left [ {F_2 \atop 0} \right ],  \quad \mbox{with some odd function $F_2$}.
\end{equation}

\medskip
{\bf Order $\Or(b^2)$:} We find the equation
\begin{equation} \label{eq:R20}
L \bm{R}_{2,0} = - \frac{1}{2} J \bm{R}_{1,0} + J \Lambda \bm{R}_{1,0} + |\bm{R}_{1,0}|^2 \bm{Q}. 
\end{equation}
Since $\bm{R}_{1,0} = [0,S_1]^\top$ with $L_- S_1 = \Lambda Q$, the solvability condition for $\bm{R}_{2,0}$ reduces to
\begin{equation}
\frac{1}{2} (\nabla Q, S_1) - (\nabla Q, \Lambda S_1) + (\nabla Q, S_1^2 Q) = 0.
\end{equation}
However, this is obviously true, since $S_1$ and $Q$ are even functions. Thus there exists a unique solution $\bm{R}_{2,0} \perp \mathrm{ker} \, L$ of equation \eqref{eq:R20}, which is given by
\begin{equation}
\bm{R}_{2,0} = \left [ { L_+^{-1} ( \frac{1}{2} S_1 - \Lambda S_1 + S_1^2 Q)  \atop 0} \right ],
\end{equation}
with $L_- S_1 = \Lambda Q$.

\medskip
{\bf Order $\Or(v^2)$:} We obtain the equation
\begin{equation}
L \bm{R}_{0,2} = - J \nabla \bm{R}_{0,1} + |\bm{R}_{0,1}|^2 \bm{Q} .
\end{equation}
Since $\bm{R}_{0,1} = [0, G_1]^\top$ and $\bm{Q} = [Q,0]^\top$ the solvability condition reads
\begin{equation}
(\nabla Q, \nabla G_1) + (\nabla Q, G_1^2 Q ) = 0.
\end{equation}
Clearly, this holds true, since $G_1$ is an odd function,  whereas $Q$ is even. Hence there exists a unique solution $\bm{R}_{0,2} \perp \mathrm{ker} \, L$ and it is given by
\begin{equation}
\bm{R}_{0,2} = \left [ {L_+^{-1} ( \nabla G_1 + G_1^2 Q) \atop 0 } \right ] .
\end{equation}

\medskip
{\bf Order $\Or(b^3)$:} We notice that $\bm{R}_{1,0} \cdot \bm{R}_{2,0} = 0$ and we obtain the equation
\begin{equation} \label{eq:R30}
L \bm{R}_{3,0} = -J \bm{R}_{2,0} + J \Lambda \bm{R}_{2,0}  +2 (\bm{R}_{2,0} \cdot \bm{Q}) \bm{R}_{1,0}  + |\bm{R}_{1,0}|^2 \bm{R}_{1,0} .
\end{equation} 
Note that the right side is of the form $[0, f]^\top$ with some nontrivial $f$. Hence the solvability condition for $\bm{R}_{3,0}$ is equivalent to
\begin{equation} \label{eq:R30_magic}
-(Q, T_2) + (Q, \Lambda T_2) + 2 (Q, Q T_2 S_1) + (Q, S_1^2 S_1) = 0, 
\end{equation}
where the function $S_1$ and $T_2$ satisfy 
\begin{equation}
L_- S_1 = \Lambda Q, \quad L_+ T_2 = \frac{1}{2} S_1 - \Lambda S_1 + S_1^2 Q .
\end{equation}
To see that \eqref{eq:R30_magic} holds, we first note that
\begin{align*}
\mbox{RHS of \eqref{eq:R30_magic}} & = -(Q,T_2) - (\Lambda Q, T_2) + 2 (T_2, Q^2 S_1) + (Q, S_1^2 S_1) \\
& = -(Q,T_2) - (L_-S_1, T_2) + 2(T_2, Q^2 S_1) + (Q,S_1^2 S_1) \\
& = -(Q,T_2) - (L_+ S_1, T_2) + (Q, S_1^2 S_1) \\
& = -(Q,T_2) -  \frac{1}{2} (S_1, S_1) + (S_1, \Lambda S_1) - (S_1, S_1^2 Q) + (Q, S_1^2 S_1) \\
& = -(Q,T_2) - \frac{1}{2} (S_1,S_1), 
\end{align*}
where in the last step also used that $(S_1, \Lambda S_1) = 0$ since $\Lambda^* = -\Lambda$. Thus it remains to show that
\begin{equation} \label{eq:R30_magic2}
-(Q,T_2) = \frac{1}{2} (S_1, S_1) .
\end{equation}
Indeed, by using $L_+ \Lambda Q = -Q$ and the equations for $T_2$ and $S_1$ above, we deduce
\begin{align}
-(Q, T_2) & =(\Lambda Q, \frac{1}{2} S_1 - \Lambda S_1 + S_1^2 Q) \nonumber \\
& = \frac{1}{2} (L_- S_1, S_1) - (L_-S_1, \Lambda S_1) +  (\Lambda Q, S_1^2 Q) \nonumber \\
& = \frac{1}{2} (S_1, DS_1) + \frac{1}{2} (S_1, S_1) - \frac{1}{2} (S_1, Q^2 S_1) - (L_-S_1, \Lambda S_1) +  (\Lambda Q, S_1^2 Q) . \label{eq:R30_magic3}
\end{align}
Next, we apply the commutator formula $ (L_- f, \Lambda f) = \frac{1}{2} (f, [L_-,\Lambda] f)$, which shows that
\begin{align} \label{eq:R30_magic4}
(L_- S_1, \Lambda S_1) & = \frac{1}{2} (S_1, [L_-,\Lambda] S_1) = \frac{1}{2} (S_1, [D,\Lambda] S_1) -  \frac{1}{2} (S_1, [Q^2,\Lambda] S_1) \nonumber \\
& =  \frac{1}{2} (S_1, D S_1) + (S_1, (x \cdot \nabla Q) Q S_1),
\end{align}
using that $[D,\Lambda] = D$ holds. Furthermore, we have the pointwise identity
 \begin{equation} \label{eq:R30_magic5}
-(x \cdot \nabla Q) Q +Q \Lambda Q =  \frac{1}{2} Q^2.
\end{equation}
If we now insert \eqref{eq:R30_magic4} and \eqref{eq:R30_magic5} into \eqref{eq:R30_magic3}, we obtain the desired relation \eqref{eq:R30_magic2} and thus the solvability condition \eqref{eq:R30_magic} holds as well. Note also that $\bm{R}_{3,0} = [0, g]^\top$ with some even function $g$.

\medskip
{\bf Order $\Or(b^4)$:} We have to solve 
\begin{equation} \label{eq:R40}
L \bm{R}_{4,0} = - \frac{3}{2} J \bm{R}_{3,0} + J \Lambda \bm{R}_{3,0} + |\bm{R}_{2,0}|^2 \bm{Q} + 2 (\bm{R}_{1,0} \cdot \bm{R}_{3,0}) \bm{Q} + 2 (\bm{R}_{2,0} \cdot \bm{Q} ) \bm{R}_{2,0} ,
\end{equation}
where we have already used that $\bm{R}_{1,0} \cdot \bm{Q} = \bm{R}_{3,0} \cdot \bm{Q} = 0$. Moreover, we readily see that 
\begin{equation}
\mbox{RHS of \eqref{eq:R40}} = \left [ {even \atop 0} \right ] \perp \mathrm{ker} \, L ,
\end{equation}
since $(g, \nabla Q) =0$ for any even function $g \in L^2(\R)$. Hence there is a unique solution $\bm{R}_{4,0} \perp \mathrm{ker} \, L$ of equation \eqref{eq:R40}, and we have that $\bm{R}_{4,0} = [h, 0]^\top$ holds with some even function $h$.

\medskip
{\bf Order $\Or(b^2 v)$:} At this order, we obtain the equation
\begin{equation} \label{eq:R21}
L \bm{R}_{2,1} = -\frac{3}{2} J \bm{R}_{1,1} + J \Lambda \bm{R}_{1,1} -  J \nabla \bm{R}_{2,0}  + 2 (\bm{R}_{1,1} \cdot \bm{Q}) \bm{R}_{1,0} + 2 (\bm{R}_{1,0} \cdot \bm{R}_{0,1} ) \bm{R}_{1,0} + |\bm{R}_{1,0}|^2 \bm{R}_{0,1} .
\end{equation}
Note also that $\bm{R}_{1,0} \cdot \bm{Q} = \bm{R}_{1,1} \cdot \bm{R}_{1,0} = \bm{R}_{0,1} \cdot \bm{R}_{2,0} = 0$. Using the symmetries of the previously constructed functions, we readily check that
\begin{equation}
\mbox{RHS of \eqref{eq:R40}} = \left [ {0 \atop odd} \right ] \perp \mathrm{ker} \, L ,
\end{equation}
since $(g, Q)=0$ for any odd function $g \in L^2(\R)$. Thus there exists a unique solution $\bm{R}_{2,1} \perp \mathrm{ker} \, L$ of equation \eqref{eq:R21}, and we see that $\bm{R}_{2,1} = [0, g]^\top$ with some odd function $g$.

\medskip
{\bf Step 2: Regularity and decay bounds.}
Let $m \geq 0$ be given. First, we recall that $\| Q \|_{H^m} \lesssim_m 1$ and $| Q(x) | \lesssim \langle x \rangle^{-2}$ holds. Since moreover $L_- \Lambda Q = -Q$ and $(\Lambda Q, Q)=0$, we can apply Lemma \ref{lem:decay} to conclude that 
\begin{equation}
\| \Lambda Q \|_{H^m} \lesssim_m 1, \quad \left | \Lambda Q(x) \right | \lesssim \langle x \rangle^{-2}.
\end{equation}
Next, by applying $\Lambda$ to the equation $L_- \Lambda = -Q$ and using that $[L_-, \Lambda] = D+ 2 x Q'Q$, we deduce
\begin{equation}
L_- \{  \Lambda^2 Q + \Lambda Q + \alpha Q \} = -(2x Q' Q+Q^2) \Lambda Q,
\end{equation}
for any $\alpha \in \R$. (Recall here that $L_- Q =0$.) By choosing $\alpha=\frac{(\Lambda Q, \Lambda Q)}{(Q,Q)}$ and using the previous bounds for $Q$ and $\Lambda Q$ (and hence for $xQ'$ as well), we can apply Lemma \ref{lem:decay} again to obtain the bounds 
\begin{equation}
\| \Lambda^2 Q \|_{m} \lesssim_m 1, \quad \left | \Lambda^2 Q(x) \right | \lesssim \langle x \rangle^{-2} .
\end{equation}

Having these bounds for $\bm{Q}=[Q, 0]^\top, \Lambda \bm{Q} = [ \Lambda Q, 0]^\top$, and $\Lambda^2 \bm{Q} = [\Lambda^2 Q, 0]^\top$ at hand, we can now prove the claimed bounds \eqref{ineq:Qbbound1} and \eqref{ineq:Qbbound2} by iterating the equations satisfied by the functions $\{ \bm{R}_{k, \ell} \}_{0 \leq k \leq 3, 0 \leq \ell 1}$ above. For instance, recall that $\bm{R}_{1,0} = [0, S_1]^\top$ with $L_- S_1 = \Lambda Q$ and hence $\Lambda L_- S_1 = \Lambda^2 Q$. Then, by using the commutator $[L_-, \Lambda]$ and the previous estimates for $\{ Q, \Lambda Q, \Lambda^2 Q\}$, we derive that
\begin{equation}
\| \Lambda^k S_1 \|_{H^m} \lesssim_m 1, \quad \left | \Lambda^k S_1(x) \right | \lesssim \langle x \rangle^{-2}, \quad \mbox{for $k=0,1,2$ and $m \geq 0$}.
\end{equation}
Using this and proceeding in the same manner, we deduce that \eqref{ineq:Qbbound1} and \eqref{ineq:Qbbound2} hold.

Finally, we mention that the bounds \eqref{ineq:Qberror} for the error term $\bm{\Psi}_\Pa$ follows from expanding $|\bm{Q}_\Pa|^2 \bm{Q}_\Pa$ and using the regularity and decay bounds for the functions $\{\bm{R}_{k,\ell} \}$. We omit the straightforward details. The proof of Proposition \ref{prop:Qb_existence} is now complete. \end{proof}

We now turn to some key properties of the approximate blowup profile $\Qbp$ constructed in Propostion \ref{prop:Qb_existence} above.

\begin{lemma} \label{lem:Qb_properties}
The mass, the energy and the linear momentum of $\Qbp$ satisfy:
$$
\int |\Qbp|^2 = \int Q^2 + \Or(b^4 + v^2 + v \Pa^2),
$$
$$
E(\Qbp) = e_1 b^2 + \Or(b^4+v^2 + v \Pa^2), \quad P(\Qbp) = p_1 v + \Or(b^4 + v^2 + v \Pa^2) . 
$$
Here $e_1 >0$ and $p_1 > 0$ are the positive constants given by 
$$e_1 = \frac{1}{2} (L_- S_1, S_1), \quad p_1 = 2 (L_- G_1, G_1),
$$ 
where $S_1$ and $G_1$ satisfy $L_- S_1 = \Lambda Q$ and $L_- G_1 = -\nabla Q$, respectively. 
\end{lemma}

\begin{remark}
Note that $L_- > 0$ on $Q^\perp$ and we have $S_1 \perp Q$ and $G_1 \perp Q$. 
\end{remark}

\begin{remark}
As an aside, we mention that a calculation shows that $$\int |\Qbp|^2 = \int Q^2 - c v^2 + \Or(b^4 + v \Pa^2)$$ with some constant $c > 0$. Hence, the boosted blowup profiles have a strictly subcritical mass for $v \neq 0$ small.
\end{remark}

\begin{proof}
From the proof of Proposition \ref{prop:Qb_existence} we recall that the facts that $\bm{R}_{1,0} = [0, S_1]^\top$, $\bm{R}_{0,1} = [0, G_1]^\top$, and $\bm{R}_{1,1} = [f, 0]^\top$ with some odd function $f$. Hence we have $\int \bm{Q} \cdot \bm{R}_{0,1} = \int \bm{Q} \cdot \bm{R}_{1,0} = \int \bm{Q} \cdot \bm{R}_{1,1} = 0$. Next, we recall that $\bm{R}_{2,0} = [T_2,0]^\top$ satisfies $(S_1,S_1) + 2(Q,T_2) = 0$, as shown in \eqref{eq:R30_magic3} above. In summary, we thus see that
$$
\int |\Qbp|^2 = \int Q^2 + \Or(b^4 + v^2 + v \Pa^2 ). 
$$

To treat the expansion of the energy, we first recall that $E(\bm{Q}) = 0$ and $DQ + Q - Q^3=0$ and thus $E'(\bm{Q}) = - Q$. Since moreover we have $(Q,S_1) =0$ and $(Q,G_1)=0$, we obtain
\begin{equation*}
E(\Qbp) = b^2 \left \{ \frac{1}{2} (S_1, D S_1) + (T_2, DQ) - \frac{1}{2} (Q^2, S_1^2) - (Q^3, T_2) \right \}  + \Or(b^4 + v^2 + v \Pa^2).   
\end{equation*}
Note also that the term $\Or(bv)$ vanishes in the expansion for $E(\Qbp)$, since $G_1$ and $S_1$ are odd and even functions, respectively, and hence $(S_1, D G_1) =0$ etc.  Using that $DQ + Q - Q^3 = 0$ and \eqref{eq:R30_magic3} once again, we see that the expression $\{ \ldots \}$ above equals $e_1 = \frac 1 2 (L_- S_1, S_1)$, as claimed.

For the expansion for the linear momentum functional, we observe that $P(\bm{f}) = 2 \int f_1 \nabla f_2$ for $\bm{f} = [f_1, f_2]^\top$. Hence, 
\begin{align*}
P(\Qbp) & = 2 b \int Q \nabla S_1 + 2 v \int Q \nabla G_1 + b^2 \int S_1 \nabla S_1 + 2b^3 \int T_2 \nabla S_1 + \Or(b^4 + v^2 + v \Pa^2) \\
& =  2v(L_- G_1, G_1) + \Or(b^4 + v^2 + v \Pa^2),
\end{align*}
since $L_-G_1 = -\nabla Q$ and using that $\int Q \nabla S_1 = \int S_1 \nabla S_1 = \int T_2 \nabla S_1=0$ due to the fact that $Q, S_1, T_2$ are even functions. The proof of Lemma \ref{lem:Qb_properties} is now complete. \end{proof}

\section{Modulation Estimates}
\label{sec:modestimates}

We start with a general observation: If $u=u(t,x)$ solves \eqref{eq:wave}, then we define the function $v=v(s,y)$ by setting
\begin{equation} \label{def:v}
u(t,x) = \frac{1}{\lambda^{\frac 1 2}(t)} v \left ( s, \frac{x - \alpha(t)}{\lambda(t)} \right ) e^{i \gamma(t)}, \quad \frac{ds}{dt} = \frac{1}{\lambda(t)} .
\end{equation}
It is easy to check that $v=v(s,y)$ with $y= \lambda^{-1} (x - \alpha)$ satisfies
\begin{equation}
i \partial_s v - D v - v + v |v|^2  = i \frac{\lambda_s}{\lambda} \Lambda v + i \frac{\alpha_s}{\lambda} \cdot \nabla v + \tgamma_s v,
\end{equation}
where we set $\tgamma_s = \gamma_s -1$. Here, of course, the operators $D$ and $\nabla$ are understood as $D = D_y$ and $\nabla = \nabla_y$, respectively.

\subsection{Geometrical decomposition and modulation equations}
Let $u(t) \in H^{1/2}(\R)$ be a solution of \eqref{eq:wave} on some time interval $[t_0,t_1]$ with $t_1 < 0$. Assume that $u(t)$ admits a geometrical decomposition of the form
\begin{equation} \label{eq:decom_u}
u(t,x) = \frac{1}{\lambda^{\frac 1 2}(t)} \left  [ Q_{\Pa(t)} + \eps \right ] \left (t, \frac{x - \alpha(t)}{\lambda(t)} \right ) e^{i \gamma(t)}, 
\end{equation}
with $\Pa(t) = (b(t), v(t))$ and we impose the uniform smallness bound 
\begin{equation} \label{ineq:parsmall}
b^2(t) + |v(t)| + \| \eps(t) \|_{H^{1/2}}^2 \ll 1.
\end{equation}
Furthermore, we assume that $u(t)$ has almost critical mass in the sense that
\begin{equation} \label{ineq:almost_crit_mass}
\left | \int |u(t)|^2 - \int Q^2 \right | \lesssim \lambda^2(t) , \quad \forall t \in [t_0,t_1].
\end{equation}

To fix the modulation parameters $\{ b(t), v(t), \lambda(t), \alpha(t), \gamma(t) \}$ uniquely, we impose the following orthogonality conditions on $\eps = \eps_1 + i \eps_2$ as follows:
\begin{align} \label{eq:ortho1}
(\eps_1, \Lambda \Theta_\Pa) - (\eps_2, \Lambda \Sigma_\Pa) = 0, \\ \label{eq:ortho2}
 (\eps_1, \partial_b \Theta_\Pa) - (\eps_2, \partial_b \Sigma_\Pa) = 0, \\ \label{eq:ortho3}
 (\eps_1, \rho_2) - (\eps_2, \rho_1) = 0, \\ \label{eq:ortho4}
 (\eps_1, \nabla \Theta_\Pa ) - (\eps_2, \nabla \Sigma_\Pa) = 0 , \\ \label{eq:ortho5}
 (\eps_1, \partial_v \Theta_\Pa) - (\eps_2, \partial_v \Sigma_\Pa ) = 0 .
\end{align}
Here and in what follows, we use the notation
\begin{equation}
Q_\Pa = \Sigma_\Pa + i \Theta_\Pa,
\end{equation}
which (in terms of the vector notation used in Section \ref{sec:profile} means that 
\begin{equation}
\Qbp = \left [ {\Sigma_\Pa \atop \Theta_\Pa} \right ] .
\end{equation}
In condition \eqref{eq:ortho4}, the function $\rho = \rho_1 + i \rho_2$ is defined by 
\begin{equation} \label{def:rho}
L_+ \rho_1 = S_1, \quad L_- \rho_2 =  2 b Q S_1 \rho_1 + b \Lambda \rho_1 - 2b T_2 + 2 v Q G_1 \rho_1 +  v \cdot \nabla \rho_1 + v F_2,
\end{equation}
where $S_1$, $T_2$ and $F_2$ are the functions introduced in the proof of Proposition \ref{prop:Qb_existence}. Note that $L_+^{-1}$ exists on $L^2_{\mathrm{even}}(\R)$ and thus $\rho_1$ is well-defined. Moreover, it is easy to see that the right-hand side in the equation for $\rho_2$ is perpendicular to $Q$. Indeed,
\begin{align*}
 (Q, 2QS_1 \rho_1 + \Lambda \rho_1 - 2 T_2) &= 2 (Q^2 S_1, \rho_1) - (\Lambda Q, \rho_1) - 2 (Q, T_2) \\
 & = 2 (Q^2 S_1, \rho_1) - (S_1, L_- \rho_1) + (S_1,S_1) \\
 & = - (S_1, L_+ \rho_1) +(S_1, S_1) = 0, 
\end{align*}
using that $(S_1,S_1) = -2(T_2, Q)$, see \eqref{eq:R30_magic3}, and the definition of $\rho_1$. Moreover, we clearly see that $2 Q G_1 \rho_1 + v \cdot \nabla \rho_1 + F_2 \perp Q$, since $G_1$ and $F_2$ are odd function, whereas $\rho_1$ and $Q$ are even. Hence $\rho_2$ is well-defined too.

We refer to Appendix \ref{sec:modeqn} for some standard arguments, which show that the orthogonality conditions \eqref{eq:ortho1}-\eqref{eq:ortho5}  imply that the modulation parameters $\{ b(t), v(t), \lambda(t), \gamma(t), \alpha(t) \}$ are uniquely determined, provided that $\eps=\eps_1 + i \eps_2$ is sufficiently small in $H^{1/2}(\R)$. Moreover, it follows from standard arguments that $\{ b(t), v(t), \lambda(t), \gamma(t), \alpha(t) \}$ are $C^1$-functions. See Appendix \ref{sec:modeqn} for more details.

In the following, we shall  often use the short-hand notation $\Sigma = \Sigma_\Pa$ and $\Theta = \Theta_\Pa$. If we insert the decomposition \eqref{eq:decom_u} into \eqref{eq:wave}, we obtain the following system
\begin{align}  \label{eq:eps1}
& \left ( b_s + \frac{1}{2} b^2 \right ) \partial_b \Sigma_\Pa + \left ( v_s + bv \right ) \partial_v \Sigma_\Pa + \partial_s \eps_1 - M_-(\eps) + b \Lambda \eps_1 - v \cdot \nabla \eps_1 \\
& = \left ( \frac{\lambda_s}{\lambda} + b \right ) \left ( \Lambda \Sigma_\Pa + \Lambda \eps_1 \right ) + \left ( \frac{\alpha_s}{\lambda} - v \right ) \cdot \left ( \nabla \Sigma_\Pa + \nabla \eps_1 \right ) + \tgamma_s \left ( \Theta_\Pa + \eps_2 \right ) \nonumber \\
& \quad + \Im (\Psi_\Pa ) - R_2(\eps),    \nonumber
\end{align}
\begin{align} \label{eq:eps2}
& \left ( b_s + \frac{1}{2} b^2 \right ) \partial_b \Theta_\Pa + \left ( v_s + bv \right ) \partial_v \Theta_\Pa + \partial_s \eps_2 + M_+(\eps) + b \Lambda \eps_2 - v \cdot \nabla \eps_2 \\
& = \left ( \frac{\lambda_s}{\lambda} + b \right ) \left ( \Lambda \Theta_\Pa + \Lambda \eps_2 \right ) + \left ( \frac{\alpha_s}{\lambda} - v \right ) \cdot \left ( \nabla \Theta_\Pa + \nabla \eps_2 \right )  - \tgamma_s \left ( \Sigma_\Pa + \eps_1 \right ) \nonumber \\
& \quad  - \Re (\Psi_\Pa ) + R_1(\eps).   \nonumber
\end{align}
Here $\Psi_\Pa$ denotes the error term from Proposition \ref{prop:Qb_existence}, and $M=(M_+,M_-)$ are small deformations of the linearized operator $L=(L_+,L_-)$ given by
\begin{align}
M_+( \eps ) & = D \eps_1 + \eps_1 - |Q_\Pa|^2 \eps_1 - 2 \Sigma_\Pa^2 \eps_1 - 2 \Sigma_\Pa \Theta_\Pa \eps_2, \\
M_- (\eps) & = D \eps_2 + \eps_2 - |Q_\Pa|^2 \eps_2 - 2 \Theta_\Pa^2 \eps_2 - 2 \Sigma_\Pa \Theta_\Pa \eps_1 .
\end{align}
The higher order terms $R_1(\eps)$ and $R_2(\eps)$ are found to be
\begin{align}
R_1(\eps) & =  3 \Sigma_\Pa \eps_1^2 + 2 \Theta_\Pa \eps_1 \eps_2 + \Sigma_\Pa \eps_2^2 + |\eps|^2 \eps_1, \\
R_2(\eps) & = 3 \Theta_\Pa \eps_2^2 + 2 \Sigma_\Pa \eps_1 \eps_2 + \Theta_\Pa \eps_1^2 + |\eps|^2 \eps_2 .
\end{align}

We have the following energy type bound.
\begin{lemma} \label{lem:energy_control}
For $t \in [t_0,t_1]$, it holds that
\begin{equation*} 
b^2 + |v| + \| \eps \|_{H^{1/2}}^2 \lesssim \lambda (|E_0| + |P_0|) + \Or(\l^2 + b^4 + v^2 + v \Pa^2) .
\end{equation*}
Here $E_0 = E(u_0)$ and $P_0 = P(u_0)$ denote the conserved energy and linear momentum of $u=u(t,x)$, respectively.
\end{lemma}

\begin{proof}
By conversation of $L^2$-mass and Lemma \ref{lem:Qb_properties}, we find that $\int |u|^2 = \int |Q_\Pa + \eps |^2 = \int |Q|^2 + 2 \Re ( \eps, Q_\Pa) + \int |\eps|^2 + \Or(b^4+ v^2 + v\Pa^2)$.
By assumption \eqref{ineq:almost_crit_mass}, this implies
\begin{equation} \label{eq:control_mass}
2 \Re ( \eps, Q_\Pa ) +  \int |\eps|^2 = \Or(\l^2 + b^4 + v^2 + v \Pa^2 ) .
\end{equation}  
Next, we recall that $v = Q_\Pa + \eps$ thanks to \eqref{def:v} and the assumed form of $u=u(t,x)$. Hence, by energy conservation and scaling, we obtain
\begin{equation} \label{eq:control_scaling}
E(v) = \lambda E(u_0).
\end{equation}
On the other hand, from Lemma \ref{lem:Qb_properties} and by expanding the energy functional,
\begin{align} 
 E(v) & = E(Q_\Pa) + \Re \left (\eps, D Q_\Pa - |Q_\Pa|^2 Q_\Pa \right ) \label{eq:control_energy}  \\
 & \quad +  \frac{1}{2} \int | D^{\frac 1 2} \eps |^2 - \frac{1}{2} \int \left \{  |Q_\Pa|^2 (\eps_1^2 + \eps_2^2) - 2 \Sigma_\Pa^2 \eps_1 - 4 \Sigma_\Pa \Theta_\Pa \eps_1 \eps_2 - 2 \Theta_\Pa \eps_2^2 \right \} \nonumber \\
 & \quad + \Or \left ( \| \eps \|_{H^{1/2}}^3 + \Pa^2 \| \eps \|_{H^{1/2}}^2 \right ) \nonumber \\
  & = e_1 b^2 + \Re \left ( \eps, D Q_\Pa - |Q_\Pa|^2 Q_\Pa \right )  \nonumber \\
& \quad  + \frac{1}{2} \int | D^{\frac 1 2} \eps |^2 - \frac{1}{2} \int \left \{  |Q_\Pa|^2 (\eps_1^2 + \eps_2^2) - 2 \Sigma_\Pa^2 \eps_1 - 4 \Sigma_\Pa \Theta_\Pa \eps_1 \eps_2 - 2 \Theta_\Pa \eps_2^2 \right \}  \nonumber \\
& \quad + \Or \left ( \| \eps \|_{H^{1/2}}^3 + \| \eps \|_{H^{1/2}}^2 \Pa^2 + b^4 + v^2 +  v\Pa^2 \right ) , \nonumber
\end{align}
where $e_1 = \frac{1}{2}(L_- S_1, S_1) > 0$. Combining \eqref{eq:control_mass}, \eqref{eq:control_scaling} and \eqref{eq:control_energy}, we find that
\begin{align*}
 \lambda E_0 & = b^2 e_1 + \Re ( \eps, D Q_\Pa + Q_\Pa - |Q_\Pa|^2 Q_\Pa)+  \frac{1}{2} \left \{ M_+(\eps) + M_- (\eps) \right \}  \\
&   \quad + \Or \left  ( \| \eps \|_{H^{1/2}}^3 + \| \eps \|_{H^{1/2}}^2 \Pa^2 + b^4 + v^2 +  v \Pa^2 \right ) .
\end{align*}
In the previous equation, we note that the term linear in $\eps=\eps_1 + i \eps_2$ satisfies
\begin{align*}
\Re \left ( \eps, DQ_\Pa + Q_\Pa - |Q_\Pa|^2 Q_\Pa \right  ) & = \Im \left ( \eps, \frac{b^2}{2} \partial_b Q_\Pa + bv \partial_v Q_\Pa - b \Lambda Q_\Pa  + v \cdot \nabla Q_\Pa \right ) \\
& \quad + \Or( \eps  ( b^4 + v^2 + v \Pa^2) ) \\
& = \Or(b^4 + v^2 + v \Pa^2),
\end{align*}
thanks to the orthogonality conditions \eqref{eq:ortho1}, \eqref{eq:ortho2}, \eqref{eq:ortho4} and \eqref{eq:ortho5}. Next, we observe that quadratic form $M=(M_+, M_-)$ is a small deformation of the quadratic form given by the linearization $L=(L_+,L_-)$ around $Q$. Hence, we deduce
\begin{multline} \label{eq:controlfinal}
b^2 e_1 + \frac{1}{2} \left \{ (L_+ \eps_1,  \eps_1) + (L_- \eps_2,  \eps_2) \right \}  \\ = \lambda E_0 + \Or( \| \eps \|_{H^{1/2}}^3 + b^4 + v^2 + v \Pa^2) + o( \| \eps \|_{H^{1/2}}^2 ) .
\end{multline} 
 Next, we recall from Lemma \ref{lem:coerc} the coercivity estimate
\begin{align}
(L_+ \eps_1, \eps_1) + (L_- \eps_2, \eps_2) & \geq c_0 \| \eps \|_{H^{1/2}}^2  \nonumber \\
& \quad - \frac{1}{c_0} \left \{ (\eps_1, Q)^2 + (\eps_1, S_1)^2 + (\eps_1, G_1)^2 + (\eps_2, \rho_1)^2 \right \}
 \label{ineq:coercgood}
\end{align}
with some universal constant $c_0 > 0$. (Here recall  that $L_- S_1 = \Lambda Q$ and $L_- G_1 = -\nabla Q$.) Note that the orthogonality conditions \eqref{eq:ortho2}, \eqref{eq:ortho3} and \eqref{eq:ortho5} imply that
$$
(\eps_1,S_1)^2 = \Or \left ( \Pa^2  \| \eps \|_{L^2}^2 \right ), \quad (\eps_1, G_1)^2 = \Or( \Pa^2 \| \eps \|_{L^2}^2), \quad (\eps_2, \rho_1)^2 = \Or ( \Pa^2 \| \eps \|_{L^2}^2) . 
$$
Furthermore, from the relation \eqref{eq:control_mass} we deduce that
$$
\left | (\eps_1, Q ) \right |^2 = o( \| \eps \|_{L^2}^2 ) + \Or(\l^2 + b^4 + v^2 +  v \Pa^2) .
$$
Combining these bounds with \eqref{ineq:coercgood} and the universal smallness assumption for $\Pa$ and $\| \eps \|_{H^{1/2}}$, we obtain that
$$
(L_+ \eps_1, \eps_1) + (L_- \eps_2, \eps_2) \geq \frac{c_0}{2} \| \eps \|_{H^{1/2}}^2 + \Or(b^4 + v^2 + v \Pa^2) .
$$
Inserting this bound into \eqref{eq:controlfinal} and recalling that $e_1 = \frac 1 2 (L_- S_1, S_1) > 0$ holds, we derive that 
\begin{equation} \label{ineq:energy_control}
b^2 + \| \eps \|_{H^{1/2}}^2 \lesssim \lambda E_0 + \Or (\l^2 + b^4 +v^2 + v \Pa^2 ).
\end{equation}

As our final step, we derive the bound for the boost parameter $v$. Here we observe that
$$
P(v) = \lambda P(u_0),
$$
by scaling and using the conversation of the linear momentum $P(u(t)) = P(u_0)$. Hence, by expansion and Lemma \ref{lem:Qb_properties} and using the orthogonality \eqref{eq:ortho4}, we obtain
\begin{align*}
\lambda P_0 & = P(v) = P(Q_\Pa) + 2 \Re ( \eps, -i \nabla (\Sigma_\Pa + i \Theta_\Pa) ) + \Re ( \eps, -i \nabla \eps ) \\
 & = p_1 v +  \Or( b^4 + v^2 + v \Pa^2 + \| \eps \|_{H^{1/2}}^2 ),
\end{align*}
with the universal constant $p_1 = 2 (L_- G_1, G_1) > 0$. Recalling that \eqref{ineq:energy_control} holds, we complete the proof of Lemma \ref{lem:energy_control}. \end{proof}

We continue with the estimating the modulation parameters. To this end, we define the vector-valued function
\begin{equation}
\Mod(t) := \left ( b_s + \frac{1}{2} b^2, \tgamma_s, \frac{\lambda_s}{\lambda} + b, \frac{\alpha_s}{\lambda}  - v, v_s + bv \right ) .
\end{equation}
We have the following result.

\begin{lemma} \label{lem:Mod_bound} 
For $t \in [t_0,t_1]$, we have the bound
$$
\left | \Mod(t) \right | \lesssim  \lambda^2 + b^4 + v^2 + v \Pa^2 + \Pa^2 \| \eps \|_{L^2} + \| \eps \|_{L^2}^2 + \| \eps \|_{H^{1/2}}^3  .
$$				
Furthermore, we have the improved bound
$$
\left | \frac{\lambda_s}{\lambda} + b \right | \lesssim b^5 + v^2 \Pa + \Pa^2 \| \eps \|_{L^2} + \| \eps \|_{L^2}^2 + \| \eps \|_{H^{1/2}}^3 .
$$
\end{lemma}

\begin{proof}

We divide the proof into the following steps, where we also make use of the estimates \eqref{eq:inner1}--\eqref{eq:inner5}, which are shown in Lemma \ref{lem:modequations} in the Appendix below.

\medskip
{\bf Step 1: Law for $b$.} 
We project the equations \eqref{eq:eps1} and \eqref{eq:eps2} onto $-\Lambda \Theta_b$ and $\Lambda \Sigma_b$, respectively. Adding this and using \eqref{eq:inner1} yields after some calculation (using also the condition \eqref{eq:ortho1}):
\begin{align*}
& - \left (b_s + \frac{1}{2} b^2 \right ) \left ( (L_-S_1, S_1) + \Or (\Pa^2) \right )  + \left ( \frac{\alpha_s}{\lambda} - v \right ) \Or(\Pa) \nonumber \\
& = \Re( \eps, Q_\Pa) + (R_2(\eps), \Lambda \Sigma_\Pa) + (R_1(\eps), \Lambda \Theta_\Pa)  \\
& \quad - (\Im (\Psi_\Pa), \Lambda \Theta_\Pa) + (\Re(\Psi_\Pa), \Lambda \Sigma_\Pa) \\ & \quad + \Or \left ((\Pa^2 + \left | \Mod(t) \right |) ( \| \eps \|_{L^2} + \Pa^2 )  \right ) .
\end{align*}
Here we also used that $(\partial_v \Sigma, \Lambda \Theta) - (\partial_v \Theta, \Lambda \Sigma ) = (G_1, \Lambda Q) + \Or(\Pa^2) = \Or(\Pa^2)$, since $G_1 = -L_-^{-1}\nabla Q$ is odd and $\Lambda Q$ is even, and hence $(G_1, \Lambda Q)=0$. 
Next, we recall from Proposition \ref{prop:Qb_existence} the universal constants
$$
e_1 = \frac{1}{2} (L_-S_1, S_1) > 0, \quad p_1 = 2 (L_- G_1, G_1) > 0.
$$
Now by using that 
$$
2\Re( \eps, Q_\Pa) = -  \int | \eps |^2 + \left ( \int |u|^2 - \int Q^2 \right ) + \Or (b^4+v^2+v \Pa^2) ,
$$
we deduce that
\begin{align}
& - \left (b_s + \frac{1}{2} b^2 \right ) \left ( 2 e_1 + \Or (\Pa^2) \right )  + \left ( \frac{\alpha_s}{\lambda} - v \right ) \left ( \frac{1}{2} p_1 + \Or(\Pa^2) \right )  \nonumber \\
& = - \frac{1}{2} \int |\eps|^2 + (R_2(\eps), \Lambda \Sigma_b) + (R_1(\eps), \Lambda \Theta_b) \nonumber \\
& \quad \Or \left ((\Pa^2 + \left | \Mod(t) \right |) ( \| \eps \|_{L^2}  + \Pa^2) + \left | \| u \|_{L^2}^2 - \| Q \|_{L^2}^2 \right | + b^4  + v^2 + v \Pa^2 \right )  \nonumber.
\end{align}

\medskip
{\bf Step 2: Law for $\lambda$.}
By projecting \eqref{eq:eps1} and \eqref{eq:eps2} onto $-\partial_b \Theta_\Pa$ and $\partial_b \Sigma_\Pa$ respectively, we obtain from adding and using \eqref{eq:ortho2} that
\begin{align*}
& \left ( \frac{\lambda_s}{\lambda} + b \right ) \left ( 2 e_1 + \Or (\Pa^2) \right )  + (v_s + bv) \Or(\Pa) =   + (R_2(\eps), \partial_b \Sigma_b) + (R_1(\eps), \partial_b \Theta_b) \\
& \quad + \Or \left ( (\Pa^2 + \left | \Mod(t) \right |) ( \| \eps \|_{L^2} + \Pa^2) + b^5 +  v^2 \Pa \right ) . 
\end{align*}
Here we also used that $(\Theta_b, \partial_b \Theta_b)+ (\Sigma_b, \partial_b \Sigma_b) = b (S_1, S_1) + 2b (Q, T_2) + v (F_2, Q) + \Or(\Pa^2) = \Or(\Pa^2)$, since $(S_1,S_1)+ 2(T_2,Q) =0$ and $(F_2, Q)=0$ because is $F_2$ is odd. Note also here that
$$
-(\nabla \Sigma, \partial_b \Theta) + (\nabla \Theta, \partial_b \Sigma) = -(\nabla Q, S_1) + \Or(\Pa^2) = \Or(\Pa^2),
$$
because $\nabla Q$ is odd and $S_1$ is even.

\medskip
{\bf Step 3: Law for $\tgamma$.} 
Now, we project \eqref{eq:eps1} and \eqref{eq:eps2} onto $-\rho_2$ and $\rho_1$, respectively. Adding this gives us
\begin{align*}
\tgamma_s ( (Q, \rho_1) + \Or(\Pa^2) ) & = - \left ( b_s + \frac{1}{2} b^2 \right ) \left ( (S_1, \rho_1) + \Or (\Pa^2) \right ) + \left ( \frac{\lambda_s}{\lambda} +b \right ) \Or (\Pa) \\
& \quad (R_2(\eps), \partial_b \Sigma_\Pa) + (R_1(\eps), \partial_b \Theta_\Pa) \\
& \quad + \Or \left ( (\Pa^2 + \left | \Mod(t) \right |) \| \eps \|_{L^2} + b^5 + v^2 \Pa \right ). 
\end{align*}
Note here also that $(\partial_v \Theta, \rho_1) = (G_1, \rho_1) = 0$ since $G_1$ is odd and $\rho_1 = L_+^{-1} S_1$ is even. Note also that $(Q, \rho_1) = (L_-S_1, S_1)=2 e_1$, which follows from $L_+ \Lambda Q = -Q$ and the definition of $\rho_1$.

\medskip
{\bf Step 4: Law for $v$.}
We project \eqref{eq:eps1} and \eqref{eq:eps2} onto $-\nabla \Theta_\Pa$ and $\nabla \Sigma_\Pa$, respectively. This gives us
\begin{align*}
\left ( v_s + bv \right ) \left ( -p_1 + \Or(\Pa^2) \right ) + \left ( b_s + \frac{1}{2} b^2 \right ) \Or(\Pa) & = (R_2(\eps), \nabla \Sigma_\Pa) + (R_1(\eps), \nabla \Theta_\Pa) \\
& \quad + \Or \left ( (\Pa^2 + \left | \Mod(t) \right |) \| \eps \|_{L^2} + b^5 + v^2 \Pa \right ). 
\end{align*}

\medskip
{\bf Step 5: Law for $\alpha$.}
Finally, if we project \eqref{eq:eps1} and \eqref{eq:eps2} onto $-\partial_v \Theta_\Pa$ and $\partial_v \Sigma_\Pa$, respectively. This yields
\begin{align*}
\left ( b_s + \frac{1}{2} b^2 \right ) \Or(\Pa) + \left ( \frac{\alpha_s}{\lambda} - v \right ) \left ( p_1 + \Or (\Pa^2) \right ) & =  (R_2(\eps), \partial_v \Sigma_\Pa) + (R_1(\eps), \partial_v \Theta_\Pa) \\
& \quad + \Or \left ( (\Pa^2 + \left | \Mod(t) \right |) \| \eps \|_{L^2} + b^4 + v^2 + v \Pa^2 \right ). 
\end{align*}
Note here $(-\Lambda \Sigma, \partial_v \Theta) + (\Lambda \Theta, \partial_v \Sigma) = (\Lambda Q, G_1) + \Or(\Pa^2) = \Or (\Pa^2)$ holds, since $\Lambda Q$ is even and $G_1$ is odd.

\medskip
{\bf Step 6: Conclusion.}
We collect the previous equations and estimate the nonlinear terms in $\eps$ by Sobolev inequalities. This gives us
\begin{align*}
(A + B) \Mod(t) & = \Or \left ( (\Pa^2 + \left | \Mod(t) \right |) \| \eps \|_{L^2} + \| \eps \|_{L^2}^2  + \| \eps \|_{H^{1/2}}^3 + \right . \\
 & \qquad \left . + \left | \| u \|_{L^2}^2 - \| Q \|_{L^2}^2 \right | + b^4 + v^2 + v \Pa^2 \right ).
\end{align*}
Here $A=O(1)$ is in invertible $5 \times 5$-matrix, whereas $B= \Or (\Pa)$ is some $5 \times 5$-matrix that is polynomial in $\Pa=(b,v)$. For $|\Pa|  \ll 1$, we can thus invert $A+B$ by Taylor expansion and derive the estimate for $ \Mod(t)$ stated in Lemma \ref{lem:Mod_bound}.  (Note also that we assumed the bound \eqref{ineq:almost_crit_mass}.)

Finally, we deduce the improved bound for $\left | \frac{\lambda_s}{\lambda} +b \right |$, by recalling the estimate derived in {\bf Step 2} above. \end{proof}

\section{Refined Energy Bounds} \label{sec:refenergy}

In this section, we establish a refined energy-viral type estimate, which will be a key ingredient in the compactness argument to construct minimal mass blowup solutions. 

Let $u=u(t,x)$ be a solution to \eqref{eq:wave} on the time interval $[t_0, 0)$ and suppose that $w$ is an approximate solution to \eqref{eq:wave} such that
\begin{equation}
i \partial_t w - D w + |w|^2 w = \psi,
\end{equation}
with the a-priori bounds
\begin{equation} \label{apriori:w}
\| w \|_{L^2} \lesssim 1, \quad \| D^{\frac 1 2} w \|_{L^2} \lesssim \lambda^{-\frac 1 2}, \quad \| D w \|_{L^2} \lesssim \lambda^{-1} .
\end{equation}
We decompose $u = w + \tilde{u}$ and hence $\tilde{u}$ satisfies
\begin{equation} \label{eq:w}
i \partial_t \tilde{u} - D \tilde{u} + (|u|^2 u - |w|^2 w) = - \psi,
\end{equation}
where we assume the a-priori bounds  
\begin{equation} \label{apriori:u}
\| D^{\frac{1}{2}+\eps} \tilde{u} \|_{L^2} \lesssim 1, \quad \| D^{\frac 1 2} \tilde{u} \|_{L^2} \lesssim \lambda^{\frac 1 2}, \quad \| \tilde{u} \|_{L^2} \lesssim \lambda,
\end{equation}
with some fixed $\eps\in (0, \frac{1}{4})$, as well as 
\begin{equation} \label{apriori:mod}
\left | \lambda_t + b \right | \lesssim \lambda^2, \quad b \sim \lambda^{\frac 1 2}, \quad \left | b_t \right | \lesssim 1, \quad \left | \alpha_t \right | \lesssim \lambda.
\end{equation}
Next, let $\phi : \R \to \R$ be a smooth and even function with the following properties\footnote{Since $\phi(x)$ is even, it clearly suffices to consider non-negative $x \geq 0$.} 
\begin{equation}
\phi'(x) = \left \{ \begin{array}{ll} x & \quad \mbox{for $0 \leq x \leq 1$,} \\
3-e^{-|x|} & \quad \mbox{for $x \geq 2$}, \end{array} \right . 
\end{equation}
and the convexity condition
\begin{equation}
\phi''(x) \geq 0 \quad \mbox{for $x \geq 0$}.
\end{equation}
Furthermore, we denote
\begin{equation*}
F(u) = \frac{1}{4} |u |^4, \quad f(u) = |u|^2 u, \quad F'(u) \cdot h = \Re ( f(u) \overline{h} ) .
\end{equation*}
Let $A > 0$ be a large constant (to be chosen later) and define the quantity
\begin{align} \label{def:Ifunc}
\In_A(u) & := \frac{1}{2} \int |D^{\frac 1 2} \tilde{u}|^2 + \frac{1}{2} \int \frac{|\tilde{u}|^2}{\lambda} - \int  \left [ F(w+ \tilde{u}) - F(w) - F'(w) \cdot \tilde{u} \right ] \\
& \quad + \frac{b}{2}  \Im \left ( \int A \nabla \phi \left ( \frac{x-\alpha}{A \lambda} \right) \cdot \nabla \tilde{u} \overline{\tilde{u}} \right ) . \nonumber
\end{align}

Our strategy will be to use the preceding functional to bootstrap control over $\|\tilde{u}\|_{H^{\frac{1}{2}}}$, see Lemma~\ref{lem:back}, and then to invoke a separate argument to improve control over $\|D^{\frac{1}{2}+\eps} \tilde{u} \|_{L^2}$. In the following lemma, control over the latter norm will help us bound certain error terms.

\begin{lemma}[Localized energy/virial estimate] \label{lem:morawetz}
Let $\In_A$ be as above. Then we have
\begin{align*}
\frac{d \In_A}{dt} &= - \frac{1}{\lambda} \Im \left ( \int w^2 \overline{\tilde{u}^2} \right ) - \Re \left ( \int w_t \overline{ ( 2 | \tilde{u}^2 | w + \tilde{u}^2 \overline{w})} \right ) \\
& \quad + \frac{b}{2\lambda}   \int \frac{|\tilde{u}|^2}{\lambda} + \frac{b}{2\lambda}  \int_{s=0}^{+\infty}  \sqrt{s} \int_{\R} \Delta \phi \left ( \frac{x-\alpha}{A \lambda} \right ) | \nabla \tilde{u}_s|^2 \, dx \,  d s   \\
& \quad - \frac{1}{8} \frac{b}{A^2 \lambda^3} \int_{s=0}^{+\infty} \sqrt{s} \int_{\R} \Delta^2 \phi \left ( \frac{x-\alpha}{A \lambda} \right ) | \tilde{u}_s |^2 \, dx \, ds  \\
& \quad + b \Re \left ( \int A \nabla \phi \left ( \frac{x-\alpha}{A \lambda} \right ) ( 2 |\tilde{u}|^2 w + \tilde{u}^2 \overline{w} ) \cdot \overline{\nabla w} \right ) \\
& \quad + \Im \left ( \int \left [ -D \psi - \frac{\psi}{\lambda} + (2 |w|^2 \psi- w^2 \overline{\psi}) + i b A \nabla \phi \left ( \frac{x-\alpha}{A \lambda} \right ) \cdot \nabla \psi  \right .  \right . \\
& \quad \left . \left . + i \frac{b}{2 \lambda} \Delta \phi \left ( \frac{x-\alpha}{A \lambda} \right )  \psi \right ] \overline{\tilde{u}} \right )  \\
& \quad + \Or \left ( \lambda \| \psi \|_{L^2}^2 + \lambda^{-1} \| \tilde{u}\|_{L^2}^2 + \loger{\| \tilde{u} \|_{H^{1/2}}} \| \tilde{u} \|_{H^{1/2}}^2 \right )  .
\end{align*}
Here we denote $\tilde{u}_s := \sqrt{\frac{2}{\pi}} \frac{1}{-\Delta+s} \tilde{u}$ with $s > 0$.
\end{lemma}


\begin{proof}[Proof of Lemma \ref{lem:morawetz}]
We divide the proof into two main steps as follows.

\medskip
{\bf Step 1: Estimating the energy part.}
Using \eqref{eq:w}, a computation shows that
\begin{align} \label{eq:mora_energy1}
& \frac{d}{dt} \left \{ \frac{1}{2} \int | D^{\frac 1 2} \tilde{u}|^2 + \frac 1 2 \int \frac{ | \tilde{u} |^2}{\lambda} - \int \left [ F(w+ \tilde{u}) - F(w) - F'(w) \cdot \tilde{u} \right ]  \right \} \\
& =   \Re \left ( \partial_t \tilde{u}, \overline{D \tilde{u} + \frac{1}{\lambda} \tilde{u} - (f(u) - f(w))} \right ) - \frac{\lambda_t}{2 \lambda^2} \int |\tilde{u}|^2 \nonumber  \\
& \quad  - \Re \left ( \partial_t w, \overline{(f(\tilde{u}+w)-f(w)- f'(w) \cdot \tilde{u} )} \right ) \nonumber \\
& = - \Im \left ( \psi, \overline{(D \tilde{u} + \frac{1}{\lambda} \tilde{u} - (f(u) - f(\tilde{u})))} \right ) - \frac{1}{\lambda} \Im \left ( f(u) - f(\tilde{u}), \overline{\tilde{u}} \right ) \nonumber \\
& \quad - \frac{\lambda_t}{2 \lambda^2} \int |\tilde{u}|^2 - \Re \left ( \partial_t w, \overline{f(\tilde{u}+w)-f(w) - f'(w) \cdot \tilde{u})} \right ) \nonumber  \\
& = - \Im \left ( \psi, \overline{D \tilde{u} + \frac{1}{\lambda} \tilde{u} - (2 |w^2| \tilde{u} - \overline{\tilde{u}} w^2)} \right ) - \frac{1}{\lambda} \int \overline{\tilde{u}^2} w^2 \nonumber \\
& \quad - \frac{\lambda_t}{2 \lambda^2} \int |\tilde{u}|^2 - \Re \left ( \partial_t w, \overline { ( \overline{w} \tilde{u}^2 +2 w |\tilde{u}|^2) } \right ) \nonumber \\
& \quad - \Im \left ( \psi- \frac{1}{\lambda} \tilde{u}, \overline{ (f(w+\tilde{u}) - f(w) - f'(w) \cdot \tilde{u}) } \right ) - \Re \left ( \partial_t w, \overline{\tilde{u} |\tilde{u}|^2} \right ),  \nonumber
\end{align}
where we denote $f'(w) \cdot \tilde{u} = 2 |w|^2 \tilde{u} + w^2 \overline{\tilde{u}}$. From \eqref{apriori:mod} we obtain that
\begin{equation} \label{eq:mora_energy2}
- \frac{\lambda_t}{2 \lambda^2} \int |\tilde{u}|^2 = \frac{b}{2 \lambda} \int \frac{ |\tilde{u}|^2}{\lambda} - \frac{1}{2 \lambda^2}  \left ( \lambda_t + b \right ) \| \tilde{u} \|_{L^2}^2 =   \frac{b}{2 \lambda} \int \frac{ |\tilde{u}|^2}{\lambda} + \Or \left ( \| \tilde{u} \|_{H^{1/2}}^2 \right ) .
\end{equation}
Next, we estimate
\begin{align} \label{ineq:mora_energy1}
& \quad \left | \Im \left ( \psi - \frac{1}{\lambda} \tilde{u}, \overline{(f(w+\tilde{u}) - f(w) - f'(w) \cdot \tilde{u})} \right ) \right | \\
& = \left | \Im \left ( \psi - \frac{1}{\lambda} \tilde{u}, \overline{ ( \tilde{u}^2 \overline{w} + 2 | \tilde{u} |^2 w + |\tilde{u}|^2 \tilde{u} ) } \right ) \right | \nonumber \\
& \lesssim (\| \psi \|_{L^2} + \lambda^{-1} \| \tilde{u} \|_{L^2} ) \| \tilde{u} \|_{L^6}^2 ( \| w \|_{L^6} + \| \tilde{u} \|_{L^6} )  \nonumber \\
& \lesssim( \| \psi \|_{L^2} + \lambda^{-1} \| \tilde{u} \|_{L^2} ) \|\tilde{u}\|_{\dot{H}^{1/2}}^{\frac{4}{3}}\|\tilde{u}\|_{L^2}^{\frac{2}{3}} ( \lambda^{-\frac 1 3} + \lambda^{\frac 2 3} ) \nonumber \\
& \lesssim \lambda \| \psi \|_{L^2}^2 + \lambda^{-1} \| \tilde{u} \|_{L^2}^2 + \| \tilde{u} \|_{H^{1/2}}^2. \nonumber
\end{align}
using the interpolation estimate $\| f \|_{L^6} \lesssim \| f \|_{\dot{H}^{1/2}}^{2/3} \| f \|_{L^2}^{1/3}$ in $\R$ together with the assumed a-priori bounds \eqref{apriori:w} and \eqref{apriori:u}. For the cubic terms hitting $\partial_t w$, we use the equation for $w$ and the bounds \eqref{apriori:w} and \eqref{apriori:u}. This leads us to
\begin{align} \label{ineq:mora_energy2}
\left | \int \partial_t w \overline{ |\tilde{u}|^2 \tilde{u} } \right | & \lesssim \| w \|_{\dot{H}^{3/4}} \| |\tilde{u}|^2 \tilde{u} \|_{\dot{H}^{1/4}} + \| w \|_{L^6}^3 \| \tilde{u} \|_{L^6}^3 + \| \psi \|_{L^2} \| \tilde{u} \|_{L^6}^3 \\
& \lesssim \frac{1}{\lambda^{3/4}} \| \tilde{u} \|_{L^2}^{1/2} \| \tilde{u} \|_{\dot{H}^{1/2}}^{5/2} + \frac{1}{\lambda} \| \tilde{u} \|_{L^2} \| \tilde{u} \|_{\dot{H}^{1/2}}^2 + \| \psi \|_{L^2} \| \tilde{u} \|_{\dot{H}^{1/2}}^2 \| \tilde{u} \|_{L^2} \nonumber \\
& \lesssim \| \tilde{u} \|_{H^{1/2}}^2 + \lambda \| \psi \|_{L^2}^2 . \nonumber
\end{align}
Here we used the bound 
$$
\|  |f|^2 f \|_{\dot{H}^{1/4}} \lesssim \| |f|^2 \|_{L^4} \| D^{\frac 1 4} f \|_{L^4} \lesssim \| f \|_{L^8}^2 \| D^{\frac 1 2} f \|_{L^2} \lesssim \| f \|_{L^2}^{1/2} \| f \|_{\dot{H}^{1/2}}^{5/2},
$$
which follows from Sobolev embedding, the interpolation estimate $\| f \|_{L^8} \lesssim \| f \|_{L^2}^{1/4} \| f \|_{\dot{H}^{1/4}}^{3/4}$ in $\R$, and the fractional chain rule $\| D^s F(u) \|_{p} \lesssim \| F'(u) \|_{p_1} \| D^s u \|_{p_2}$ for any $F \in C^1(\C)$ with $0 < s \leq 1$ and $1 < p, p_1, p_2 < \infty$ such that $\frac{1}{p} = \frac{1}{p_1} + \frac{1}{p_2}$. 

We now insert \eqref{eq:mora_energy2}, \eqref{ineq:mora_energy1} and \eqref{ineq:mora_energy2} into \eqref{eq:mora_energy1}. Combined with the assumed a priori bounds on $\tilde{u}$, we conclude  
\begin{align*}
& \frac{d}{dt} \left \{ \frac{1}{2} \int | D^{\frac 1 2} \tilde{u}|^2 + \frac 1 2 \int \frac{ | \tilde{u} |^2}{\lambda} - \int \left [ F(w+ \tilde{u}) - F(w) - F'(w) \cdot \tilde{u} \right ]  \right \} \\
& = - \frac{1}{\lambda} \Im \left ( \int w^2 \overline{\tilde{u}^2} \right ) - \Re \left ( \int w_t \overline{ ( 2 | \tilde{u}^2 | w + \tilde{u}^2 \overline{w})} \right ) + \frac{b}{2 \lambda} \int \frac{|\tilde{u}|^2}{\lambda} \\
& \quad + \Im \left ( \int \left [ -D \psi - \frac{\psi}{\lambda} + (2|w|^2 \psi- w^2 \overline{\psi} )  \right ] \overline{\tilde{u}} \right ) \\
& \quad + \Or ( \lambda \| \psi \|_{L^2}^2 + \lambda^{-1} \| \tilde{u}\|_{L^2}^2 + \| \tilde{u} \|_{H^{1/2}}^2 ).
\end{align*}

\medskip
{\bf Step 2: Estimating the localized virial part.}
We set
\begin{equation}
\nabla \tilde{\phi}(t,x) := b A \nabla \phi \left ( \frac{x - \alpha}{A \lambda} \right ) .
\end{equation}
Then we obtain
\begin{align} \label{eq:mora_evo}
& \frac{1}{2} \frac{d}{dt} \left ( b \Im \left ( \int A \nabla \phi \left ( \frac{x}{A \lambda} \right ) \cdot \nabla \tilde{u} \overline{\tilde{u}} \right ) \right ) \\
& = \frac{1}{2} \Im \left ( \int (\partial_t \nabla \tilde{\phi} ) \cdot \nabla \tilde{u} \overline{\tilde{u}} \right ) + \frac{1}{2} \Im \left (  \int \nabla \tilde{\phi} \cdot \left ( ( \nabla \partial_t \tilde{u}) \overline{\tilde{u}} + \nabla \tilde{u} \partial_t \overline{\tilde{u}} \right ) \right ) . \nonumber
\end{align}
Using the bounds \eqref{apriori:mod}, we estimate
\begin{equation}
\left | \partial_t \nabla \tilde{\phi} \right | \lesssim |b_t| + b \left | \frac{\lambda_t}{\lambda} \right |  + \frac{b}{\lambda} \left| \alpha_t \right | \lesssim 1+ \frac{b}{\lambda} ( |b| + \left | \lambda_t + b\right | + \left | \alpha_t \right | ) \lesssim 1,
\end{equation}
\begin{equation}
 \left |\partial_t \Delta \tilde{\phi} \right | \lesssim \lambda^{-1} .
\end{equation}
Hence, by Lemma \ref{lem:leibniz}, we deduce that
\begin{equation}
\left | \Im \left ( \int (\partial_t \nabla \tilde{\phi}) \cdot \nabla \tilde{u} \overline{\tilde{u}} \right ) \right | \lesssim \| \tilde{u} \|_{\dot{H}^{1/2}}^2 + \lambda^{-1} \| \tilde{u}\|_{L^2}^2 .
\end{equation}

Now, we turn to the second term in \eqref{eq:mora_evo} containing the time derivative of $\tilde{u}$. To handle this term, it is expedient to write this using commutators $[A,B] \equiv AB-BA$. Moreover, it is convenient to adapt the notation
\begin{equation}
p = - i \nabla_x
\end{equation}
in the following and hence $D = |p|$. Using \eqref{eq:w} and that $D=D^*$ is self-adjoint, a calculation yields that

\begin{align} \label{eq:mora_manip}
 & \frac{1}{2} \Im \left (  \int \nabla \tilde{\phi} \cdot  \left ( ( \nabla \partial_t \tilde{u} ) \overline{\tilde{u}} + \nabla \tilde{u} \partial_t \overline{\tilde{u}} \right ) \right )  = - \frac{1}{4} \Re \left ( \int \overline{\tilde{u}} \left [-i |p|, \nabla \tilde{\phi} \cdot p + p  \cdot \nabla \tilde{\phi} \right ] \tilde{u} \right )  \\
 & \quad - b \Re \left ( \int  (|u|^2 u - |w|^2 w) A\nabla \phi \left ( \frac{x-\alpha}{A \lambda} \right )  \cdot \overline{\nabla \tilde{u}} \right ) \nonumber \\ &\quad - \frac{1}{2} \frac{b}{\lambda} \Re \left ( \int (|u|^2 u- |w|^2 w) \Delta \phi \left ( \frac{x-\alpha}{A \lambda} \right ) |\tilde{u}|^2 \right ) \nonumber \\ & \quad- b \Re \left ( \int \psi \nabla \phi \left ( \frac{x-\alpha}{A \lambda} \right )  \cdot \overline{\nabla \tilde{u}} \right ) \nonumber - \frac{1}{2} \frac{b}{\lambda} \Re \left ( \int \psi \Delta \phi \left ( \frac{x-\alpha}{A \lambda} \right ) \overline{\tilde{u}} \right ). \nonumber
\end{align}  
Next, we rewrite the commutator by using some identities from functional calculus. Here, we recall the known formula
\begin{equation}
x^{\beta} = \frac{\sin ( \pi \beta)}{\pi} \int_0^\infty s^{\beta-1} \frac{x}{x + s}  \, ds,
\end{equation}
for $x > 0$ and $0 < \beta < 1$. Using this formula and the spectral theorem applied to the self-adjoint operator $p^2$, we readily obtain the commutator formula
\begin{equation}
[|p|^{\alpha}, B] = \frac{\sin (\pi \alpha /2) }{\pi} \int_0^{+\infty} s^{\frac \alpha 2} \frac{1}{p^2+s} [p^2, B] \frac{1}{p^2+s} \, ds, 
\end{equation}
for any $0 < \alpha < 2$ and any (possibly unbounded) self-adjoint operator $B$ whose domain contains $\mathcal{S}(\R)$. In particular, we deduce that
\begin{equation}
\left [|p|, \nabla \tilde{\phi} \cdot p + p \cdot \nabla \tilde{\phi} \right ] = \frac{1}{\pi} \int_0^\infty \sqrt{s} \frac{1}{p^2+s} \left [p^2, \nabla \tilde{\phi} \cdot p + p \cdot \nabla \tilde{\phi} \right ] \frac{1}{p^2 + s}   \, d s ,
\end{equation}   
 Next, we recall the known formula
\begin{equation}
\left [p^2, \nabla \tilde{\phi} \cdot p + p \cdot \nabla \tilde{\phi} \right ] = -4i p \cdot \Delta \tilde{\phi} p + i \Delta^2 \tilde{\phi} ,
\end{equation}
for any smooth function $\tilde{\phi}$ on $\R$. We now define the auxiliary function
\begin{equation} \label{def:Us}
\tilde{u}_s(t,x) := \sqrt{\frac{2}{\pi}} \frac{1}{-\Delta + s} \tilde{u}(t,x), \quad \mbox{for $s > 0$}.
\end{equation}
Hence, by construction, we have that $\tilde{u}_s$ solves the elliptic equation
\begin{equation}
-\Delta \tilde{u}_s + s \tilde{u}_s = \sqrt{\frac{2}{\pi}} \tilde{u} .
\end{equation}
Note that the integral kernel for the resolvent $(-\Delta+s)^{-1}$ in $d=1$ dimension is explicitly given by $\frac{1}{2 \sqrt{s}} e^{-\sqrt{s} |x-y|}$. Hence, as an aside, we remark that we have the convolution formula
\begin{equation}
\tilde{u}_s(t,x) = \frac{1}{\sqrt{2 \pi s}} \int e^{- \sqrt{s} |x-y|} \tilde{u}(t,y) \, dy  .
\end{equation}
Recalling that $\nabla \tilde{\phi}(t,x) = b A \nabla \phi \left ( \frac{x-\alpha}{A \lambda} \right )$ and using that $(p^2+s)^{-1}$ is self-adjoint and the definition of $\tilde{u}_s$ above as well as Fubini's theorem, we conclude that
\begin{align} \label{eq:mora_magic}
 - \frac{1}{4} \Re \left ( \int \overline{\tilde{u}} \left [-i |p|, \nabla \tilde{\phi} \cdot p + p  \cdot \nabla \tilde{\phi} \right ] \tilde{u} \right ) & = \frac{b}{2\lambda}  \int_{s=0}^{+\infty} \sqrt{s} \int_{\R} \Delta \phi \left ( \frac{x-\alpha}{A \lambda} \right ) | \nabla \tilde{u}_s|^2 \, dx \, d s \\ 
 & \quad - \frac{1}{8} \frac{b}{A^2 \lambda^3} \int_{s=0}^{+\infty} \sqrt{s} \int_{\R} \Delta^2 \phi \left ( \frac{x-\alpha}{A \lambda} \right ) | \tilde{u}_s |^2 \, dx \, d s   . \nonumber
\end{align}
Next, we estimate the terms in \eqref{eq:mora_manip} that are cubic and higher order in $\tilde{u}$. Using the fractional Leibniz rule as well as the bounds \eqref{apriori:u}, \eqref{apriori:mod}, \eqref{apriori:w}, we find that
\begin{align} \label{ineq:loc_int}
& \left | b \Re \left ( A \nabla \phi \left ( \frac{x-\alpha}{A \lambda} \right ) ( 2 |\tilde{u}|^2 w + \tilde{u}^2 \overline{w} + |\tilde{u}|^2 \tilde{u}) \cdot \overline{\nabla \tilde{u}} \right )  \right . \\
& \quad \left .  - \frac{1}{2} \frac{b}{\lambda} \Re \left ( \int \Delta \phi \left ( \frac{x-\alpha}{A \lambda} \right ) (2 |\tilde{u}|^2 w + \tilde{u}^2 \overline{w} + |\tilde{u}|^2 \tilde{u} ) \tilde{u} \right ) \right | \nonumber \\
&\lesssim \|\tilde{u}\|_{\dot{H}^{\frac{1}{2}}}^2\|\nabla\tilde{\phi}\|_{L^\infty}\|\tilde{u}\|_{L^\infty}(\|\tilde{u}\|_{L^\infty} + \|w\|_{L^{\infty}})+ \|\tilde{u}\|_{\dot{H}^{\frac{1}{2}}} \| \nabla \tilde{\phi} \|_{L^\infty} \|\tilde{u}\|_{L^\infty}^2 \|w\|_{\dot{H}^{\frac{1}{2}}} \nonumber \\
& \quad +\|\tilde{u}\|_{\dot{H}^{\frac{1}{2}}}\|\nabla\tilde{\phi}\|_{\dot{H}^{\frac{1}{2}}}\|\tilde{u}\|_{L^{\infty}}^2(\|\tilde{u}\|_{L^\infty} + \|w\|_{L^\infty})
+\lambda^{-\frac{1}{2}}( \| \tilde{u} \|_{L^4}^3 \| w \|_{L^4}+ \|\tilde{u}\|_{L^4}^4 )\nonumber \\
& \lesssim \Or ( \loger{ \| \tilde{u} \|_{H^{1/2}} } \| \tilde{u} \|_{H^{1/2}}^2 ) +  \lambda^{-\frac 1 2} ( \lambda^{-\frac 1 4} \| \tilde{u} \|_{\dot{H}^{1/2}}^{\frac 3 2} \| \tilde{u} \|_{L^2}^{\frac 3 2}+ \| \tilde{u} \|_{\dot{H}^{1/2}}^2 \| \tilde{u} \|_{L^2}^2  ) \nonumber \\
& \lesssim \Or ( \loger{ \| \tilde{u} \|_{H^{1/2}} } \| \tilde{u} \|_{H^{1/2}}^2 ) + \lambda^{\frac 1 4} \| \tilde{u} \|_{H^{1/2}}^2 + \lambda^{\frac 3 2} \| \tilde{u} \|_{H^{1/2}}^2 \nonumber \\
& \lesssim \Or  ( \loger{ \| \tilde{u} \|_{H^{1/2}} } \| \tilde{u} \|_{H^{1/2}}^2 ) , \nonumber
&\end{align} 
where we have again exploited Lemma~\ref{lem:brezis} as well as the assumed a priori bounds on $\|\tilde{u}\|_{H^{\frac{1}{2}+}}$. Moreover, we also used the fact that, by Lemma \ref{lem:brezis} and by the bounds \eqref{apriori:w} and \eqref{apriori:mod}, we have $\| \tilde{u} \|_{\dot{H}^{1/2}} \| \nabla \tilde{\phi} \|_{L^\infty} \| w \|_{L^\infty} \lesssim \lambda^{\frac 1 2} \cdot  b  \cdot\lambda^{-\frac 1 2} \left |\log \lambda \right | \lesssim \lambda^{\frac 1 2} \left | \log \lambda \right  | \lesssim 1$. Furthermore note that $\| \nabla \tilde{\phi} \|_{\dot{H}^{1/2}} \lesssim 1$ holds, which can be easily checked by calculation.

Next, we consider the terms in \eqref{eq:mora_manip} that are quadratic in $\tilde{u}$. Integrating by parts, we obtain
\begin{align} \label{eq:loc_int1}
& -b \Re \left ( \int \psi A \nabla \phi \left ( \frac{x-\alpha}{A \lambda} \right )  \cdot \overline{\nabla \tilde{u}} \right ) - \frac{1}{2} \frac{b}{\lambda} \Re \left ( \int \psi \Delta \phi \left ( \frac{x-\alpha}{A \lambda} \right ) \overline{\tilde{u}} \right ) \\
& = \Im \left ( \int \left [ i b A \nabla \phi \left ( \frac{x-\alpha}{A \lambda} \right ) \cdot \nabla \psi + i \frac{b}{2 \lambda} \Delta \phi \left ( \frac{x-\alpha}{A \lambda} \right )  \psi \right ] \overline{\tilde{u}} \right ) .
\end{align}
Moreover, an integration by parts yields that
\begin{align} \label{eq:loc_int2}
& -b \Re \left ( \int A \nabla \phi \left ( \frac{x-\alpha}{A \lambda} \right ) (2 |w|^2 \tilde{u} + w^2 \overline{\tilde{u}} ) \cdot \overline{\nabla \tilde{u}} \right ) \\
& \quad -\frac{1}{2} \frac{b}{\lambda} \Re \left ( \int \Delta \phi \left ( \frac{x-\alpha}{A \lambda} \right ) (2 |w|^2  \tilde{u}  + w^2 \overline{\tilde{u}} ) \overline{\tilde{u}} \right ) \nonumber \\
& = b \Re \left ( \int A \nabla \phi \left ( \frac{x-\alpha}{A \lambda} \right ) ( 2 |\tilde{u}|^2 w + \tilde{u}^2 \overline{w} ) \cdot \overline{\nabla w} \right ) . \nonumber 
 \end{align}
Note that $\Delta \phi$ is not present on right-hand side of the previous equation and that the quadratic term is different from those appearing on the left-hand side. 

Finally, we insert \eqref{ineq:loc_int}, \eqref{eq:loc_int1} and \eqref{eq:loc_int2} into \eqref{eq:mora_manip}. This yields together with \eqref{eq:mora_magic} and another integration by parts the following equation
\begin{align*}
&  \frac{1}{2} \Im \left (  \int \nabla \tilde{\phi} \cdot \left ( ( \nabla \partial_t  \tilde{u}) \overline{\tilde{u}} + \nabla \tilde{u} \partial_t \overline{\tilde{u}} \right ) \right ) \\
& = \frac{b}{2\lambda}   \int_{s=0}^{+\infty} \sqrt{s} \int_{\R} \Delta \phi \left ( \frac{x}{A \lambda} \right ) | \nabla \tilde{u}_s|^2 \, dx \, d s  - \frac{1}{8} \frac{b}{A^2 \lambda^3} \int_{s=0}^{+\infty} \sqrt{s} \int_{\R} \Delta^2 \phi \left ( \frac{x-\alpha}{A \lambda} \right ) | \tilde{u}_s |^2 \, dx \, ds   \\
& \quad + b \Re \left ( \int A \nabla \phi \left ( \frac{x-\alpha}{A \lambda} \right ) ( 2 |\tilde{u}|^2 w + \tilde{u}^2 \overline{w} ) \cdot \overline{\nabla w} \right ) \\
& \quad + \Im \left ( \int \left [ i b A \nabla \phi \left ( \frac{x-\alpha}{A \lambda} \right ) \cdot \nabla \psi + i \frac{b}{2 \lambda} \Delta \phi \left ( \frac{x-\alpha}{A \lambda} \right )  \psi \right ] \overline{\tilde{u}} \right ) \\
& \quad + \Or \left (  \loger{\| \tilde{u} \|_{H^{1/2}}} \| \tilde{u} \|_{H^{1/2}}^2 \right ) ,
\end{align*}
where $\tilde{u}_s$ is defined in \eqref{def:Us}. This completes the proof of Lemma \ref{lem:morawetz}. \end{proof}

\section{Backwards Propagation of Smallness}

We now apply the energy estimate of the previous section in order to establish a bootstrap argument, which will be needed in the construction of minimal mass blowup solutions. Let $u=u(t,x)$ be an even solution to \eqref{eq:wave} defined in $[\tilde{t}_0, 0)$. Assume that $\tilde{t}_0 < t_1 < 0$ and suppose that $u$ admits on $[\tilde{t}_0, t_1]$ a geometrical decomposition of the form
\begin{equation}
u(t,x) = \frac{1}{\lambda^{\frac 1 2}(t)} \left [  Q_{\Pa(t)} + \eps \right ]  \left (t, \frac{x-\alpha(t)}{\lambda(t)} \right ) e^{i \gamma(t)} ,
\end{equation}
where $\eps = \eps_1 + i \eps_2$ satisfies the orthogonality conditions \eqref{eq:ortho1}--\eqref{eq:ortho5} and $b^2 + |v| + \| \eps \|_{H^{1/2}}^2 \ll 1$ holds. We set
\begin{equation}
\tilde{u}(t,x) = \frac{1}{\lambda^{\frac 1 2}(t)} \eps \left ( t, \frac{x-\alpha(t)}{\lambda(t)} \right ) e^{i \gamma(t)} .
\end{equation}
Suppose that the energy satisfies $E_0=E(u) > 0$ and define the constant 
\begin{equation} \label{def:C0}
C_0 = \sqrt{ \frac{ e_1 }{E_0} },
\end{equation}
with the universal constant $e_1 = \frac{1}{2} (L_- S_1, S_1) > 0$. Moreover, let $P_0 = P(u_0)$ be the linear momentum and define the constant
\begin{equation} \label{def:D0}
D_0 = \frac{P_0}{p_1} ,
\end{equation}
with the universal constant $p_1 = 2 (L_- G_1, G_1) > 0$.

We claim that the following backwards propagation estimate holds. 

\begin{lemma}[Backwards propagation of smallness] \label{lem:back}
Assume that, for some $t_1 < 0$ sufficiently close to 0 and some $\eps\in (0,\frac{1}{4})$ fixed, we have the bounds
$$
\left | \| u \|_{L^2}^2 - \| Q \|_{L^2}^2 \right | \lesssim \lambda^2(t_1),
$$
$$
\| D^{\frac 1 2} \tilde{u}(t_1) \|_{L^2}^2 + \frac{ \| \tilde{u}(t_1) \|_{L^2}^2 }{\lambda(t_1)} \lesssim \lambda(t_1) ,\,\,\| D^{\frac12 + \eps}\tilde{u}(t_1)\|_{L^2}^2 \lesssim \lambda^{\frac{1}{2}-2\eps}(t_1)
$$
$$
\left | \lambda(t_1) - \frac{t_1^2}{4C^2_0} \right | \lesssim \lambda^{\frac 3 2}(t_1), \quad \left | \frac{b(t_1)}{\lambda^{\frac 1 2}(t_1)} - \frac{1}{C_0} \right | \lesssim \lambda(t_1) , \quad
\left | \frac{v(t_1)}{\lambda(t_1)} - D_0 \right | \lesssim \lambda(t_1) .
$$
Then there exists a time $t_0 < t_1$ depending only on $C_0$ and $D_0$ such that $\forall t \in [t_0, t_1]$ it holds
$$
 \| D^{\frac 1 2} \tilde{u}(t) \|_{L^2}^2 + \frac{ \| \tilde{u}(t) \|_{L^2}^2}{\lambda(t)} \lesssim \| D^{\frac 1 2} \tilde{u}(t_1) \|_{L^2}^2 + \frac{\| \tilde{u}(t_1) \|_{L^2}^2}{\lambda(t_1)} + \lambda^3(t),
$$
$$
\| D^{\frac12 + \eps}\tilde{u}(t)\|_{L^2}^2 \lesssim \lambda^{\frac{1}{2}-2\eps}(t),
$$
$$
\left | \lambda(t) - \frac{t^2}{4C^2_0} \right | \lesssim \lambda^{\frac 3 2}(t), \quad \left | \frac{b(t)}{\lambda^{\frac 1 2}(t)} - \frac{1}{C_0} \right | \lesssim \lambda(t) , \quad
\left | \frac{v(t)}{\lambda(t)} - D_0 \right | \lesssim \lambda(t) .
$$
\end{lemma}

\begin{proof}
By assumption, we have $u \in C^0([t_0, t_1]; H^{1/2+\eps}(\R))$. Hence, by this continuity and the continuity of the functions $\{ \lambda(t), b(t), \alpha(t), v(t)  \}$, there exists a time $t_0 < t_1$ such that $\forall t \in [t_0,t_1]$ we have the bounds
\begin{equation} \label{ineq:cont_bound1}
\| \tilde{u} \|_{L^2} \leq K \lambda(t), \quad \| \tilde{u}(t) \|_{H^{1/2}} \leq K \lambda^{\frac 1 2}(t),
\end{equation}
\begin{equation}\label{ineq:cont_bound1.5}
 \|\tilde{u}(t)\|_{H^{\frac{1}{2}+\eps}}\leq K\lambda^{\frac{1}{4}-\eps}(t),
\end{equation}
\begin{equation} \label{ineq:cont_bound2}
\left | \lambda(t) - \frac{t^2}{4C^2_0} \right | \leq K \lambda^{\frac 3 2}(t), \quad \left | \frac{b(t)}{\lambda^{\frac 1 2}(t)} - \frac{1}{C_0} \right | \leq K \lambda(t),
\end{equation}
\begin{equation} \label{ineq:cont_bound3}
\left | \frac{v(t)}{\lambda(t)} - D_0 \right | \leq K \lambda(t) ,
\end{equation}
with some constant $K > 0$. We now claim that the bounds stated in Lemma \ref{lem:back} hold on $[t_0,t_1]$ and hence improving \eqref{ineq:cont_bound1} -- \eqref{ineq:cont_bound3} on $[t_0,t_1]$ for $t_0=t_0(C_0) < t_1$ small enough but independent of $t_1$. Here we first improve the bounds \eqref{ineq:cont_bound1}, \eqref{ineq:cont_bound2} and \eqref{ineq:cont_bound3}, and we defer the improvement of the technical bound \eqref{ineq:cont_bound1.5} to the appendix. 
 We divide the proof into the following steps.

\medskip
{\bf Step 1: Bounds on energy and $L^2$-norm.}
We set
\begin{equation} \label{def:wcoerc}
w(t,x) = \tilde{Q}(t,x) = \frac{1}{\lambda^{\frac 1 2}(t)} Q_{\Pa(t)} \left ( \frac{x-\alpha(t)}{\lambda(t)} \right ) e^{i \gamma(t)} .
\end{equation} 
Let $\In_A$ be given by \eqref{def:Ifunc}. Applying Lemma \ref{lem:morawetz}, we claim that we obtain the following coercivity estimate
\begin{equation} \label{ineq:coerc1}
\frac{d \In_A}{dt} \geq \frac{b}{\lambda^2} \int | \tilde{u} |^2 + \Or \left ( \loger{\| \tilde{u} \|_{H^{1/2}}} \| \tilde{u} \|_{H^{1/2}}^2 + K^4 \lambda^{\frac 5 2} \right ) . 
\end{equation}
For the moment, let us assume that we have already proven that \eqref{ineq:coerc1} holds. By Sobolev embedding and the smallness of $\eps$, we deduce the upper bound 
\begin{equation} \label{ineq:coerc2}
\left | \In_A \right | \lesssim \| D^{\frac 1 2} \tilde{u} \|_{L^2}^2 + \frac{1}{\l} \| \tilde{u} \|_{L^2}^2 .
\end{equation}
Note here that,  by Lemma \ref{lem:leibniz}, we have the bound
\begin{equation}
\left |  \Im \left ( \int A \nabla \phi \left ( \frac{x-\alpha}{A \lambda} \right ) \cdot \nabla \tilde{u} \overline{ \tilde{u} } \right ) \right | \lesssim \| D^{\frac 1 2} \tilde{u} \|_{L^2}^2 + \frac{1}{\lambda} \| \tilde{u} \|_{L^2}^2, 
\end{equation}
Furthermore, due to the proximity of $Q_\Pa$ to $Q$, we conclude the lower bound
\begin{align} \label{ineq:coerc3}
\In_A & = \frac{1}{2} \int |D^{\frac 1 2} \tilde{u}|^2 + \frac{1}{2} \int \frac{ |\tilde{u}|^2}{\lambda} - \int (F(w+\tilde{u}) - F(w) - F'(w) \cdot \tilde{u} )  \\
& \quad + \frac{b}{2}  \Im \left ( \int A \nabla \phi \left ( \frac{x-\alpha}{A \lambda} \right ) \cdot \nabla \tilde{u} \overline{\tilde{u}} \right ) \nonumber \\
& = \frac{1}{2 \lambda} \left [ (L_+ \eps_1, \eps_1) + (L_- \eps_2, \eps_2) + o( \| \eps \|_{H^{1/2}}^2 ) \right ] \geq \frac{c_0}{\lambda} \left [ \| \eps \|_{H^{1/2}}^2 - (\eps_1, Q)^2 \right ] ,\nonumber
\end{align} 
using the orthogonality conditions satisfied by $\eps$ and the coercivity estimate for the linearized operator $L=(L_+,L_-)$. On the other hand, using the conservation of the $L^2$-mass and applying Lemma \ref{lem:energy_control} (and in particular \eqref{eq:control_mass}) we combine the assumed bounds to conclude that
\begin{equation}  \label{ineq:coerc3.5}
\left | \Re (\eps, Q_b) \right | \lesssim \| \eps \|_{L^2}^2 + \lambda^2(t) + \left | \int |u|^2 - \int |Q|^2 \right | \lesssim \| \eps \|_{L^2}^2 + K^2  \lambda^2(t).
\end{equation}
This implies 
\begin{equation} \label{ineq:coerc4}
(\eps_1, Q)^2 \lesssim o (\| \eps \|_{L^2}^2) + K^4 \lambda^4(t).
\end{equation}
Next, we define
\begin{equation}
X(t) := \| D^{\frac 1 2} \tilde{u}(t) \|_{L^2}^2 + \frac{\| \tilde{u}(t) \|_{L^2}^2}{\lambda(t)} .
\end{equation}
By integrating \eqref{ineq:coerc1} in time and using \eqref{ineq:coerc2}, \eqref{ineq:coerc3} and \eqref{ineq:coerc4}, we find
\begin{align*}
 X(t)  & \lesssim X(t_1) + K^4 \lambda^3(t) + \int_t^{t_1} \left ( \loger{ \| \tilde{u} \|_{H^{1/2}}} \| \tilde{u}(\tau) \|_{H^{1/2}}^2 + K^4 \lambda^{5/2}(\tau) \right ) \, d\tau \\
& \lesssim  X(t_1) + K^4 \lambda^3(t) + \int_t^{t_1} \log^{\frac 1 2}\left ( 2 + X(\tau)^{-\frac 1 2} \right ) X(\tau) \, d\tau ,
\end{align*} 
for $t \in [t_0,t_1]$ with some $t_0 = t_0(C_0) < t_1$ close enough to $t_1  <0$. By Gronwall's inequality, we deduce the desired bound for $X(t)$. In particular, we obtain
\begin{equation} \label{ineq:X_bound}
X(t) = \| D^{\frac{1}{2}} \tilde{u}(t) \|_{L^2}^2 + \frac{ \| \tilde{u}(t) \|_{L^2}^2}{\lambda(t)} \lesssim \lambda(t) , \quad \forall t \in [t_0,t_1],
\end{equation}
which closes the bootstrap for \eqref{ineq:cont_bound1}. 

\medskip
{\bf  Step 2: Controlling the law for the parameters.}
From Lemma \ref{lem:Mod_bound} and using \eqref{ineq:cont_bound2} and \eqref{ineq:X_bound}, we deduce
\begin{equation} \label{ineq:ode_bound1}
\left | b_s + \frac 1 2 b^2 \right | + \left | \frac{\lambda_s}{\lambda} + b \right | \lesssim \lambda^2.
\end{equation}
As a  direct consequence of this bound, we observe that
\begin{equation}
\left ( \frac{b}{\lambda^{\frac 1 2}} \right )_s = \frac{b_s + \frac{1}{2} b^2}{\lambda^{\frac 1 2}} - \frac{b}{2 \lambda^{\frac 1 2}} \left ( \frac{\lambda_s}{\lambda} + b \right ) \lesssim \lambda^{\frac 3 2} .
\end{equation}
Hence, for any $s < s_1$, we have
\begin{equation} \label{ineq:ode_bound2}
\frac{1}{C_0} - \frac{b}{\lambda^{\frac 1 2}}(s) \lesssim \frac{1}{C_0} - \frac{b}{\lambda^{\frac 1 2}}(s_1) + \int_{s}^{s_1} \lambda^{\frac 3 2}(s') \, ds' \lesssim \lambda(s).
\end{equation}
Note that we used here that $\lambda(t) \sim t^2$ by \eqref{ineq:cont_bound2} and the relation $dt = \lambda^{-1} ds$, as well as the assumed initial bound for $\left | b/\lambda^{\frac 1 2}(t) - 1/C_0 \right |$ at time $t=t_1$. Next, by following the calculations in the proof of Lemma \ref{lem:energy_control} and recalling that $b^2 + |v| \sim \lambda$ thanks to \eqref{ineq:cont_bound2} and \eqref{ineq:cont_bound3} and $\| \eps \|_{H^{1/2}}^2 \lesssim \lambda^2$ by \eqref{ineq:X_bound} and scaling, we deduce
\begin{equation}
b^2 e_1 = \lambda E_0 + \left ( \int |u|^2 - \int Q^2 \right ) + \Or( \lambda^2 ),
\end{equation}
where $e_1 = \frac{1}{2} (L_- S_1, S_1) > 0$ is a universal constant. Since $\int |u|^2 - \int Q^2 = \Or(\lambda^2)$ and recalling the definition of $C_0> 0$ above, we deduce that
\begin{equation}
\frac{b^2}{\lambda} - \frac{1}{C_0^2} = \left ( \frac{b}{\lambda^{\frac 1 2}} - \frac{1}{C_0} \right ) \left ( \frac{b}{\lambda^{\frac 1 2}} + \frac{1}{C_0} \right ) = \Or(\lambda) .
\end{equation}
Furthermore, from \eqref{ineq:ode_bound2} we see that $\frac{b}{\lambda^{1/2}}  \gtrsim 1$. Hence, we obtain the desired bound
\begin{equation} \label{ineq:ode_bound3}
\left | \frac{b}{\lambda^{\frac 1 2}} - \frac{1}{C_0} \right | \lesssim \lambda.
\end{equation}
We conclude using \fref{ineq:cont_bound2}, \fref{ineq:ode_bound1}:
$$-\lambda_t=b+O(\l^2)=\frac{\l^{\frac12}}{C_0}+O(\l^{\frac 32}+t^4)=\frac{\l^{\frac12}}{C_0}+O(t^3).$$ Dividing by $\l^{\frac 12}\sim |t|$, integrating in $[t,t_1]$  and using the boundary value at $t_1$ ensures: 
$$\left|\l^{\frac12}(t)-\frac{t}{2C_0}\right|\lesssim \left|\l^{\frac12}(t_1)-\frac{t_1}{2C_0}\right|+O(t^3)\lesssim t^2$$ and the desired bound for $\l$ follows.

Next, we improve the bound \eqref{ineq:cont_bound3}. In fact, by following the calculations in the proof of Lemma \ref{lem:energy_control} for the linear momentum $P(u_0)$ and recalling that $b^2 + |v| \sim \lambda$ thanks to \eqref{ineq:cont_bound2} and \eqref{ineq:cont_bound3}, we deduce
\begin{equation}
 v p_1 = \lambda P_0 + \Or (\lambda^2) ,
\end{equation}
with the universal constant $p_1 = 2 (L_- G_1, G_1) > 0$. Here we  also used that $\| \eps \|_{H^{1/2}}^2 \lesssim \lambda^2$ by \eqref{ineq:X_bound} and by scaling. Recalling the definition of $D_0 = P_0/p_1$, we thus obtain
\begin{equation}
\left | \frac{v(t)}{\lambda(t)} - D_0 \right | \lesssim \lambda(t) .
\end{equation}
This completes the proof of Step 2, assuming that the coercivity estimate \eqref{ineq:coerc1} holds true.

\medskip
{\bf Step 3: Proof of the coercivity estimate \eqref{ineq:coerc1}.}
Recall that $w = \tilde{Q}$ is given in \eqref{def:wcoerc}. Let $\mathcal{K}_A(\tilde{u})$ denote the terms quadratic in $\tilde{u}$ on the right-hand side in the equation in Lemma \ref{lem:morawetz}, i.\,e., we put
\begin{align*}
\mathcal{K}_A(\tilde{u}) &:= - \frac{1}{\lambda} \Im \left ( \int w^2 \overline{\tilde{u}^2} \right ) - \Re \left ( \int w_t \overline{ ( 2 | \tilde{u}^2 | w + \tilde{u}^2 \overline{w})} \right ) \\
& \quad + \frac{b}{\lambda}   \int \frac{|\tilde{u}|^2}{\lambda} + \frac{b}{\lambda}  \int_{s=0}^{+\infty} \sqrt{s} \int_{\R} \Delta \phi \left ( \frac{x-\alpha}{A \lambda} \right ) | \nabla \tilde{u}_s|^2 \, dx \, d s   \\
& \quad - \frac{1}{4} \frac{b}{A^2 \lambda^3}   \int_{s =0}^{+\infty} \sqrt{s} \int_{\R} \Delta^2 \phi \left ( \frac{x-\alpha}{A \lambda} \right ) | \tilde{u}_s |^2 \, dx \, d s  \\
& \quad + b \Re \left ( \int A \nabla \phi \left ( \frac{x-\alpha}{A \lambda} \right ) ( 2 |\tilde{u}|^2 w + \tilde{u}^2 \overline{w} ) \cdot \overline{\nabla w} \right ) .
\end{align*}
Recall that the function $\tilde{u}_s=\tilde{u}_s(t,x)$ with the parameter $s >0$ was defined in Lemma \ref{lem:morawetz} to be $\tilde{u}_s = \sqrt{ \frac{2}{\pi}} \frac{1}{-\Delta +s} \tilde{u}$. Recalling that $\tilde{u}(t,x) = \lambda^{-1/2} \eps(t, \lambda^{-1} x)$, we now claim that the following estimate holds:
\begin{equation} \label{ineq:KA1}
\mathcal{K}_A(\tilde{u}) \geq \frac{c}{\lambda^{3/2}} \int |\eps|^2 + \Or (K^4 \lambda^{5/2})
\end{equation}
with some universal constant $c> 0$. 

Indeed, from Lemma \ref{lem:Mod_bound} and estimate \eqref{ineq:cont_bound1} we obtain that
\begin{equation} \label{ineq:modbound}
\left | \Mod(t) \right | \lesssim K^2 \lambda^2(t) .
\end{equation}
Using this estimate, we find that $w = \tilde{Q}$ satisfies
\begin{align*}
\partial_t \tilde{Q} & = e^{i \gamma(t)} \frac{1}{\lambda^{1/2}} \left [ - \frac{\lambda_t}{\lambda} \Lambda Q_\Pa+ i \gamma_t Q_b + b_t \frac{\partial Q_\Pa}{\partial b} + v_t \frac{\pa Q_\Pa}{\pa v} - \frac{\alpha_t}{\lambda} \cdot \nabla Q_\Pa \right ] \left ( \frac{x-\alpha}{\lambda} \right )  \\ 
& = \left ( \frac{i}{\lambda} + \frac{b}{2\lambda} \right ) \tilde{Q} + b \left ( \frac{x-\alpha}{\lambda} \right ) \cdot \nabla \tilde{Q} + \Or ( K \lambda^{-1/2} ), 
\end{align*}
recalling also that $\tilde{\gamma}_s = \gamma_s -1$ and $\frac{ds}{dt} = \lambda^{-1}$. Note that we also used the uniform bounds  $\|  \partial_b Q_\Pa \|_{L^\infty} \lesssim 1, \| \pa _v Q_\Pa \|_{L^\infty} \lesssim 1$ and the facts that $|b_t| \lesssim K, |v_t| \lesssim K$ , which can be seen from \eqref{ineq:modbound}, \eqref{ineq:cont_bound2} and \eqref{ineq:cont_bound3}. Hence,
\begin{align*}
- \Re \left ( \int \partial_t \tilde{Q} \overline{ ( 2 |\tilde{u}|^2 \tilde{Q} + \tilde{u}^2 \overline{\tilde{Q}} )} \right ) & =  \frac{1}{\lambda} \Im \left ( \int \tilde{Q} \overline{ (2 |\tilde{u}|^2 \tilde{Q} + \tilde{u}^2 \overline{\tilde{Q}} ) } \right ) - \frac{b}{2 \lambda} \Re \left ( \int (2 |\tilde{u}^2 \tilde{Q} + \tilde{u}^2 \overline{\tilde{Q}}) \overline{\tilde{Q}} \right ) \\
& - b \Re \left ( \int \left ( \frac{x-\alpha}{\lambda} \right ) ( 2 |\tilde{u}|^2 \tilde{Q} + \tilde{u}^2 \overline{\tilde{Q}} ) \cdot \overline{\nabla \tilde{Q} } \right )  \\
& + \Or( K \lambda^{-1} \| \eps \|_{L^2}^2 ).
\end{align*}
Note here that, in order to deduce the bound on the error term, we used that
$$
\left | \int \Or ( K \lambda^{-1/2} ) |\tilde{u}|^2 \overline{\tilde{Q}} \right |  \lesssim \frac{K}{\lambda} \| \eps \|_{L^2}^2 = \Or (K \lambda^{-1} \| \eps \|_{L^2}^2 ) ,
$$
thanks to the bound $\| \tilde{Q} \|_{L^\infty} \lesssim \lambda^{-1/2}$ and the scaling relation $\tilde{u}(t,x) = \lambda^{-1/2} \eps(t,\lambda^{-1} (x-\alpha) )$. Going back to the definition of $\mathcal{K}_A(\tilde{u})$ and expressing everything in terms of $\eps(t,x) = \lambda^{1/2} \tilde{u} (t, \lambda x+\alpha)$, we  conclude that
\begin{align*}
\mathcal{K}_A(\tilde{u}) & = \frac{b}{2\lambda^{2}} \left \{ \int_{s=0}^{+\infty} \sqrt{s} \int \Delta \phi \left ( \frac{x}{A} \right )  |\nabla \eps_s|^2 \, dx \, ds  + \int | \eps|^2 \right .  \\
& \quad - \int ((|Q_\Pa|^2 + 2 \Sigma^2) \eps_1^2 + 4 \Sigma \Theta \eps_1 \eps_2 + (|Q_\Pa|^2 + 2\Theta^2) \eps_2^2) \\
& \quad  - \frac{1}{4 A^2} \int_{s=0}^{+\infty} \sqrt{s} \int \Delta^2 \phi \left ( \frac{x}{A} \right) |\eps_s|^2 \, dx \, ds \\
& \quad \left . +  \; 2 \Re \left ( \int \left ( A \nabla \phi \left( \frac{x}{A} \right ) - x \right ) (2 |\eps|^2 Q_\Pa + \eps^2 \overline{Q_\Pa} ) \cdot \overline{\nabla Q_\Pa} \right ) \right \} \\
& \quad + \Or (K \lambda^{-1} \| \eps \|_{L^2}^2 ) .
\end{align*}
Next we note that $A \nabla \phi(x/A) - x \equiv 0$ for $|x| \leq A$ and we estimate
\begin{align*}
& \left | \int \left ( A \nabla \phi \left( \frac{x}{A} \right ) - x \right ) (2 |\eps|^2 Q_\Pa + \eps^2 \overline{Q_\Pa} ) \cdot  \overline{\nabla Q_\Pa} \right | \\ & \lesssim \| (A+|x|) Q_\Pa \|_{L^\infty ( \{ |x| \geq AÊ\})} \| \nabla Q_\Pa \|_{L^\infty} \| \eps \|_{L^2}^2  \lesssim \left \| \frac{A+|x|}{1+|x|^2} \right \|_{L^\infty( \{ |x| \geq A \} )} \| \eps \|_{L^2}^2  \lesssim \frac{1}{A} \| \eps \|_{L^2}^2 ,
\end{align*}
where we used the uniform decay estimate $|Q_\Pa(x)| \lesssim \langle x \rangle^{-2}$ and the bound $\| \nabla Q_\Pa \|_{L^\infty} \lesssim \| Q_\Pa \|_{H^2} \lesssim 1$. Furthermore, thanks to Lemma \ref{lem:phi4}, we have 
\begin{equation}
\left | \frac{1}{A^2} \int_{s=0}^{+\infty} \sqrt{s} \int \Delta^2 \phi \left ( \frac{x}{A} \right) |\eps_s|^2 \, dx \, ds \right | \lesssim \frac{1}{A} \| \eps \|_{L^2}^2.
\end{equation} 
Recalling the definitions of $L_{+,A}$ and $L_{-,A}$ in \eqref{def:LplusA} and \eqref{def:LminusA}, we deduce that
\begin{align*}
\mathcal{K}_A(\tilde{u}) & = \frac{b}{2 \lambda^2} \left \{ ( L_{+,A} \eps_1, \eps_1) + (L_{-,A} \eps_2, \eps_2)  +\Or\left(\frac{1}{A} \int |\eps|^2\right)  \right \} \\
& \quad +  \frac{1}{\lambda^{3/2}} \Or(K \lambda^{1/2} \| \eps \|_{L^2}^2) 
\end{align*}
Next, we recall that $b \sim \lambda^{\frac 1 2}$ due to \eqref{ineq:cont_bound2}. Hence, by Proposition \ref{prop:LA} and choosing $A > 0$ sufficiently large, we deduce from previous estimates that
\begin{equation}
\mathcal{K}_A(\tilde{u}) \gtrsim \frac{1}{\lambda^{3/2}} \left \{ \int |\eps|^2 - (\eps_1, Q)^2 \right \} \gtrsim \frac{1}{\lambda^{3/2}} \int |\eps|^2 + \Or(K \lambda^{5/2}),
\end{equation}
where the last step follows from \eqref{ineq:coerc4}. This completes the proof of \eqref{ineq:KA1} and  Step 3.

\medskip
{\bf Step 4: Controlling the remainder terms in $\frac{d}{dt} \In_A$.}
We now control the terms that appear in Lemma \ref{lem:morawetz} and contain $\psi$. Here we recall that $w = \tilde{Q}$ and \eqref{eq:w}, which yields
\begin{align*}
\psi = \frac{1}{\lambda^{\frac  3 2}} & \left [ i \left ( b_s + \frac 1 2 b^2 \right ) \partial_b Q_\Pa - i \left ( \frac{\lambda_s}{\lambda} +  b \right ) \Lambda Q_\Pa + i \left  ( v_s+ bv \right ) \partial_v Q_\Pa  \right . \\
&   \left .  - i \left ( \frac{\alpha_s}{\lambda} - v \right ) \cdot \nabla Q_\Pa + \tilde{\gamma}_s Q_\Pa + \Psi_\Pa \right ] \left ( \frac{x-\alpha}{\lambda} \right )e^{i \gamma} . 
\end{align*}
Here $\Psi_\Pa$ is the error term given in Proposition \ref{prop:Qb_existence}. In fact, by the estimates for $Q_\Pa$ and $\Psi_\Pa$ from Proposition \ref{prop:Qb_existence} and recalling \eqref{ineq:modbound}, we deduce the rough pointwise bounds
\begin{equation} \label{ineq:rough}
\left | \nabla^k \psi(x) \right | \lesssim \frac{1}{\lambda^{\frac 3 2 + k}} \left \langle \frac{x-\alpha}{\lambda} \right \rangle^{-2} K^2 \lambda^2 , \quad \mbox{for $k=0,1$}.
\end{equation}
 Hence,
\begin{equation}
\| \nabla^k \psi \|_{L^2} \lesssim K^2 \lambda^{1-k}, \quad \mbox{for $k=0,1$}.
\end{equation}
In particular, we obtain the following bounds
\begin{equation}
\lambda \| \psi \|_{L^2}^2 \lesssim K^4 \lambda^3, 
\end{equation}
\begin{align}
& \left | \Im \left ( \int  \left [ i b A \nabla \phi \left ( \frac{x-\alpha}{A \lambda} \right ) \cdot \nabla \psi + i \frac{b}{2 \lambda} \Delta \phi \left ( \frac{x-\alpha}{A \lambda} \right )  \psi \right ] \overline{\tilde{u}} \right ) \right | \\
& \lesssim \lambda^{\frac 1 2} \| \nabla \psi \|_{L^2} \| \tilde{u} \|_{L^2} + \lambda^{-\frac 1 2} \| \psi \|_{L^2} \| \tilde{u} \|_{L^2} \nonumber \\
& \lesssim K^2 \lambda^{\frac 1 2} \| \eps \|_{L^2} \lesssim o \left ( \frac{\| \eps \|_{L^2}^2}{\lambda^{\frac 3 2}} \right ) + K^4 \lambda^{\frac 5 2} . \nonumber
\end{align}
Similar as in [RSz], the rough bound \eqref{ineq:rough} is not sufficient to control the remaining terms with $\psi$ in Lemma \ref{lem:morawetz}. In fact, we have to exploit a further cancellation as follows. Write $\psi = \psi_1 + \psi_2$ with $\psi_2= \Or(\Pa|\Mod|+b^5)=\Or(\l^{\frac 52})$, i.\,e., we denote
\begin{align*}
\psi_1  = \frac{1}{\lambda^{\frac 3 2}} & \left [ - \left ( b_s + \frac 1 2 b^2 \right ) S_1 - i \left ( \frac{\lambda_s}{\lambda} +  b \right ) \Lambda Q  - \left ( v_s + bv \right ) G_1 \right . \\
& \quad \left . - i \left ( \frac{\alpha_s}{\lambda} - v \right ) \cdot \nabla Q +  \tilde{\gamma}_s Q \right ] \left ( \frac{x-\alpha}{\lambda} \right )e^{i \gamma}  .
\end{align*}
Let us first deal with the estimating the contributions coming from $\psi_2$. Indeed, since $|b|^2 + |v| \sim \lambda$ we note that $\psi_2 = \Or(\l^{\frac 52})$ satisfies the pointwise bound
\begin{equation}
\left | \nabla^k \psi_2(x) \right | \lesssim \frac{1}{\lambda^{\frac 3 2 -k}}  \left \langle \frac{x-\alpha}{\lambda} \right \rangle^{-2} K^2 \lambda^{\frac 5 2} , \quad \mbox{for $k=0,1$}.
\end{equation} 
Hence,
\begin{equation}  
\| \nabla^k \psi_2 \|_{L^2} \lesssim K^2 \lambda^{\frac 3 2 - k} , \quad \mbox{for $k=0,1$}.
\end{equation}
Therefore, we obtain similarly as above
\begin{align*}
& \left | \Im \left ( \int \left [ -D \psi_2 - \frac{\psi_2}{\lambda} + (2 |w|^2 \psi_2 -w^2 \overline{\psi_2} ) \right ] \overline{\tilde{u}} \right )  \right | \\
& \lesssim \left ( \| \nabla \psi_2 \|_{L^2} + \lambda^{-1} \| \psi_2 \|_{L^2} + \| \psi_2 \|_{L^\infty} \| w \|_{L^4}^2 \right ) \| \eps \|_{L^2} \nonumber \\
& \lesssim K^2 \lambda^{\frac 1 2}  \|\eps \|_{L^2} \lesssim o \left ( \frac{\| \eps \|_{L^2}^2}{\lambda^{\frac 3 2}} \right ) + K^4 \lambda^{\frac 5 2} ,
\end{align*}
which is acceptable. We finally use the fact that $\psi_1$ belongs to the generalized null space of $L=(L_+, L_-)$ and hence an extra factor of $\Or(\Pa)$ is gained using the orthogonality conditions obeyed by $\eps=\eps_1 + i \eps_2$. Indeed, we find the following bound
\begin{align*}
& \left | \Im  \left ( \int \left [ -D \psi_1 - \frac{\psi_1}{\lambda} + (2 |w|^2 \psi_1 -w^2 \overline{\psi}_1 ) \right ] \overline{\tilde{u}} \right )  \right | \\
& \lesssim \frac{ \left | \Mod(t) \right |}{\lambda^2} \left [ \left |(\eps_2, L_- S_1)\right | + \left | (\eps_2, L_- G_1 ) \right | + \left | (\eps_2, L_-Q) \right | + \Or ( \Pa \| \eps \|_{L^2}) \right ]  \\
& \quad  + \frac{1}{\lambda^2} \left | \frac{\lambda_s}{\lambda} + b \right |  \left | (\eps_1, L_+ \Lambda Q) \right | + \frac{1}{\lambda^2} \left | \frac{\alpha_s}{\lambda}-v \right | \left | (\eps_1, L_+ \nabla Q) \right |  \\
& \lesssim K^2 \lambda^{\frac 1 2} \| \eps \|_{L^2} + \frac{K^2 \lambda \| \eps \|_{L^2} + \lambda^{\frac 5 2}}{\lambda^2} \left (K \lambda^{\frac 1 2} \| \eps \|_{L^2} + K^2 \lambda^2 \right ) \\
& \lesssim o \left ( \frac{\| \eps \|_{L^2}^2}{\lambda^{\frac 3 2}} \right ) + K^4 \lambda^{\frac 5 2} ,
\end{align*}
which is an acceptable bound.  Here we used \eqref{ineq:modbound} once again and $|\Pa| \lesssim \lambda^{\frac 1 2}$, as well as $(\eps_2, L_- S_1) = (\eps_2, \Lambda Q) = \Or (\Pa \|\eps \|_{L^2})$ and $(\eps_2, L_- G_1) = - (\eps_2, \nabla Q) = \Or(\Pa \| \eps \|_{L^2})$, thanks to the orthogonality conditions for $\eps$. Moreover, we used that $L_+ \nabla Q =0$ and $L_+ \Lambda Q = -Q$ together with improved bound in Lemma \ref{lem:Mod_bound}, combined with the fact that $|(\eps_1, Q )| \lesssim \lambda^{\frac 1 2} \| \eps \|_{L^2} + K^2 \lambda^2$, which follows from $\| \eps \|_{L^2} \lesssim \lambda$ and the conversation of $L^2$-mass leading to bound \eqref{ineq:coerc4} above. 

Finally, we recall  \eqref{ineq:KA1} and we  insert all the derived estimates for the terms involving $\psi$ in Lemma \ref{lem:morawetz} and we conclude that the coercivity property \eqref{ineq:coerc1} holds.

\medskip
{\bf Step 5: Bounds on $\| D^{\frac 1 2 + \eps} \tilde{u}(t)\|_{L^2}$.} This step is detailed in Appendix \ref{sec:higherHs}. 

\medskip
\noindent
The proof of Lemma \ref{lem:back} is now complete. \end{proof}

\section{Existence of Minimal Mass Blowup Solutions}
\label{sec:existence_min}

With the results of the previous sections as hand, we are now ready to prove the following main result, which in particular yields Theorem \ref{thmmain}.

\begin{thm} \label{thm:exist}
Let $\gamma_0,x_0, P_0 \in \R$ and $E_0 > 0$ be given. Then there exist a time $t_0 < 0$ and a solution $u_c \in C^0([t_0,0); H^{\frac 1 2+\eps}(\R))$ of \eqref{eq:wave} with some $0 < \eps < \frac 1 4$ such that $u_c$ blows up at time $T=0$ with
$$
E(u_c) = E_0, \quad P(u_0) = P_0, \quad \mbox{and} \quad \| u_c \|_{L^2} = \| Q \|_{L^2}.
$$
Furthermore, we have $\| D^{\frac 1 2} u_c(t) \|_{L^2} \sim |t|^{-1}$ as $t \to 0^-$, and $u_c$ is of the form
$$
u_c(t,x) = \frac{1}{\lambda_c^{\frac 1 2}(t)} \left [ Q_{\Pa_c(t)} + \eps_c \right ] \left ( t, \frac{x-\alpha_c(t)}{\lambda_c(t)} \right ) e^{i \gamma_c(t)} = \tilde{Q}_c + \tilde{u}_c,
$$
with $\Pa_c(t) = (b_c(t), v_c(t))$, and $\eps_c$ satisfies the orthogonality conditions \eqref{eq:ortho1}--\eqref{eq:ortho3}. Finally, the following estimates hold:
$$
\| \tilde{u}_c \|_{L^2} \lesssim \lambda_c, \quad \| \tilde{u}_c \|_{H^{1/2}} \lesssim \lambda^{\frac 1 2}_c, 
$$
$$
\lambda_c(t) - \frac{t^2}{4 C_0^2} = \Or (\lambda^{\frac 3 2}_c), \quad \frac{b_c}{\lambda_c^{\frac 1 2}}(t) - \frac{1}{C_0} = \Or(\lambda_c),  \quad \frac{v_c}{\lambda_c}(t) - D_0 = \Or(\lambda_c) ,
$$
$$
 \gamma_c(t) = -\frac{4 C_0^2}{t} + \gamma_0 + \Or(\lambda^{\frac 1 2}_c) , \quad \alpha_c(t) = x_0 + \Or(\lambda_c^{\frac 3 2}) .
$$
Here $C_0 > 0$ and $D_0 \in \R$ are the constants defined in \eqref{def:C0} and \eqref{def:D0}, respectively.
\end{thm}

\begin{proof}
We use a compactness argument; see also \cite{Me1990,Ma2005,RaSz2011} for such compactness techniques.

Let $t_n \to 0^-$ be a sequence of negative times and let $u_n$ be the solution to \eqref{eq:wave} with initial data at $t=t_n$ given by
\begin{equation} \label{def:un_initial}
u_n(t_n,x) = \frac{1}{\lambda_n^{\frac 1 2}(t_n)} Q_{\Pa_n(t_n)} \left ( \frac{x-\alpha_n(t_n)}{\lambda_n(t_n)} \right ) e^{i \gamma_n(t_n)},
\end{equation}
where the sequences $\Pa_n(t_n) = (b_n(t_n), v_n(t_n))$ and $\{ \gamma_n(t_n), \alpha_n(t_n) \}$ are given by
\begin{equation} \label{def:un_initial2}
b_n(t_n) = - \frac{t_n}{2 C_0^2}, \quad \lambda_n(t_n) = \frac{t^2_n}{4 C_0^2}, \quad \gamma_n(t_n) = \gamma_0 - \frac{4 C_0^2}{t_n},
\end{equation}
\begin{equation} \label{def:un_initial2.5}
v_n(t_n) = \frac{D_0 t_n^2}{2 C_0}, \quad \alpha_n(t_n) = x_0 .
\end{equation} 
By Lemma \ref{lem:Qb_properties}, we have
\begin{equation}
\int | u_n(t_n) |^2 = \int Q^2 + \Or(t_n^4) ,
\end{equation}
and $\tilde{u}(t_n) = 0$ by construction. Thus $u_n$ satisfies the assumptions of Lemma \ref{lem:back}. Hence we can find a backwards time $t_0$ independent of $n$ such that for all $t \in [t_0, t_n)$ we have the geometric decomposition
\begin{equation}
u_n(t,x) = \frac{1}{\lambda_n^{\frac 1 2}(t)} Q_{\Pa_n(t)} \left (t, \frac{x-\alpha_n(t)}{\lambda_n(t)} \right ) + \tilde{u}_n(t,x),
\end{equation}
with the uniform bounds (with some fixed $\eps \in (0,\frac 1 4)$) given by
\begin{equation}
\| D^{\frac 1 2} \tilde{u}_n \|_{L^2}^2 + \frac{\| \tilde{u}_n \|_{L^2}^2}{\lambda_n(t)} \lesssim \lambda_n^3(t),
\end{equation}
\begin{equation}
\| D^{\frac 1 2 + \eps} \tilde{u}_n \|_{L^2}^2 \lesssim \lambda^{\frac 1 2 - 2 \eps}_n(t),
\end{equation}
\begin{equation}
\left | \frac{b_n(t)}{\lambda^{\frac 1 2}_n(t)} - \frac{1}{C_0} \right | \lesssim \lambda_n(t), \quad \left | \lambda_n(t) - \frac{t^2}{4 C_0^2} \right | \lesssim \lambda^{\frac 3 2}_n(t), \quad \left | \frac{v_n(t)}{\lambda_n(t)} - D_0 \right | \lesssim \lambda_n(t) .
\end{equation}
Next, we  conclude that $\{ u_n(t_0) \}_{n=1}^\infty$ converges strongly in $H^{1/2}(\R)$ (after passing to a subsequence if necessary). Indeed, from the uniform bound $\| \tilde{u}(t_0) \|_{H^{1/2+\eps}} \lesssim 1$ we can assume (after passing to a subsequence if necessary) that $u_n(t_0) \weakto u_c$ weakly in $H^s(\R)$ for any $s \in [0,\frac{1}{2} + \eps]$. Moreover, we note the uniform bound
\begin{equation}
\left | \frac{d}{dt} \int \chi_R |u_n|^2 \right | \lesssim \left | \int u_n [\chi_R, i D] \overline{u_n} \right | \lesssim \| \nabla \chi_R \|_{L^\infty}  \| u_n \|_{L^2}^2 \lesssim \frac{1}{R},
\end{equation}
with a smooth cutoff function $\chi_R(x)= \chi(x/R)$ where $\chi(x) \equiv 0$ for $|x| \leq 1$ and $\chi(x) \equiv 1$ for $|x| \geq 2$. Note also that we used the commutator estimate $\| [\chi_R, D] \|_{L^2 \to L^2} \lesssim \| \nabla \chi_R \|_{L^\infty}$; see, e.\,g., \cite{Ca1965,St1993}. By integrating the previous bound from $t_1$ to $t_0$ and using \eqref{def:un_initial}--\eqref{def:un_initial2}, we derive that the sequence $\{ u_n(t_0) \}_{n=1}^\infty$ is tight in $L^2(\R)$. That is, for every $\delta > 0$ there is a radius $R > 0$ such that $\int_{|x| \geq R} |u_n(t_0)|^2 \leq \delta$ for all $n \geq 1$. Combining this fact with the weak convergence of $\{ u_n(t_0) \}_{n=1}^\infty$ in $H^s(\R)$, we deduce that
\begin{equation}
\mbox{$u_n(t_0) \to u_c(t_0)$ strongly in $H^{s}(\R)$ for every $s \in [0,\frac{1}{2} + \eps)$}.
\end{equation}
Thus, by the local wellposedness (see Appendix \ref{sec:cauchy}), we obtain that
\begin{equation}
\mbox{$u_n(t) \to u_c(t)$ strongly in $H^{1/2}(\R)$ for $t \in [t_0, T_c)$},
\end{equation}
where $T_c>t_0$ is the life time of $u_c$ on the right. Moreover, $u_c$ admits for $t<\min\{T_c,0\}$ a geometrical decomposition of the form stated in Theorem \ref{thm:exist} with
\begin{equation}
b_n(t) \to b_c(t),  \quad v_n(t) \to v_c(t), \quad \lambda_n(t) \to \lambda_c(t), \quad \gamma_n(t) \to \gamma_c(t) , \quad \alpha_n(t) \to \alpha_c(t),
\end{equation}
and $\{ b_c(t), v_c(t), \lambda_c(t)\}$ satisfy the bounds stated in Theorem \ref{thm:exist}.  Moreover, we derive the bounds for $\| \tilde{u}_n \|_{L^2} \lesssim \lambda_c$ and $\| \tilde{u}_c \|_{H^{1/2}} \lesssim \lambda_c^{\frac 1 2}$.  In particular, this implies that $u_c(t)$ blows up at time $T_c=0$ such that $\| D^{\frac 1 2} u_c(t) \|_{L^2}^2 \sim \lambda^{-1}(t) \sim |t|^{-2}$ as $t \to 0^-$. In addition, we deduce from $L^2$-mass conservation and the strong convergence that
$$
\| u_c \|_{L^2} = \lim_{n \to +\infty} \| u_n(t_n) \|_{L^2} = \| Q \|_{L^2}.
$$ 
As for the energy, we notice that
$$
E(u_c(t)) = \frac{b_c^2}{\lambda_c} e_1 + o(1) \to E_0 \quad \mbox{as} \quad t \to 0^- ,
$$
by the choice of $C_0$ and $b_n(t_n)$ and $\lambda_n(t_n)$. By energy conservation, this implies that
$$
E(u_c) = E_0 .
$$ 
Also, we observe that
$$
P(u_c(t)) = \frac{v_c}{\lambda_c} p_1 + o(1) \to P_0 \quad \mbox{as} \quad t \to 0^-, 
$$
by our choice of $D_0$ and $v_n(t_n)$ and $\lambda_n(t_n)$. By momentum conservation, this shows that 
$$
P(u_c) = P_0.
$$
Next, we recall the rough bound
$$
\left | (\tilde{\gamma}_n)_s \right | \lesssim \lambda_n.
$$
Therefore, using that $ds/dt=\lambda^{-1}$ and the estimates for $\lambda_n$,
$$
\left | \frac{d}{dt} \left ( \gamma_n + \frac{4 C_0^2}{t} \right ) \right | = \frac{1}{\lambda_n} \left | (\gamma_n)_s - \frac{4 C_0^2 \lambda_n}{t^2} \right | = \frac{1}{\lambda_n} \left | (\tilde{\gamma}_n)_s - \left ( \frac{4 C_0^2 \lambda_n}{t^2} - 1 \right ) \right | \lesssim 1.
$$
Integrating this bound and using \eqref{def:un_initial2} and $\lambda_c \sim t^2$, we find
$$
\gamma_n(t) + \frac{4 C_0^2}{t} = \gamma_0 + \Or(\lambda_c^{\frac 1 2}),
$$
whence the claim for $\gamma_c$ follows, since we have $\lambda_c \sim t^2$. Finally, we recall the rough bound $\left | \frac{(\alpha_n)_s}{\lambda_n} + v_n \right | \lesssim \lambda_n$. Integrating this and using the bounds for $v_n$ and $\lambda_n$, we deduce that
$$
\left | \frac{d}{dt} \left ( \alpha_n - x_0 \right ) \right | =  \left | \frac{(\alpha_n)_s}{\lambda_n} \right | \lesssim \lambda_n + \left | v_n \right |  \lesssim \lambda_n.
$$
Integrating this bound and using \eqref{def:un_initial2.5}, we find that
$$
\alpha_n(t) = x_0 + \Or(\lambda_c^{\frac 3 2}).
$$
which shows that the claim for $\alpha_c(t)$ holds.
 
The proof of Theorem \ref{thm:exist} is now complete. \end{proof}

\begin{appendix}

\section{Decay and Smoothing Estimates for $L_+$ and $L_-$}

In this section, we collect some regularity and decay estimates concerning the linearized operators $L_-$ and $L_+$.

\begin{lemma} \label{lem:decay}
Let $f, g \in H^k(\R)$ for some $k \geq 0$ and suppose $f \perp Q$ and $g \perp Q'$. Then we have the regularity bounds
$$
\| L^{-1}_- f \|_{H^{k+1}} \lesssim_k \| f \|_{H^k}, \quad \| L_+^{-1} g \|_{H^{k+1}} \lesssim_k \| g \|_{H^k},
$$
and the decay estimates
$$
\| \langle x \rangle^2 L^{-1}_- f \|_{L^\infty} \lesssim  \| \langle x \rangle^2 f \|_{L^\infty},   \quad \| \langle x \rangle^2 L_+^{-1} g \|_{L^\infty} \lesssim \| \langle x \rangle^2 f \|_{L^\infty}.
$$
\end{lemma}

\begin{proof}
It suffices to prove the lemma for $L_-^{-1} f$, since the estimates for $L_+^{-1} g$ follow in the same fashion.

To show the regularity bound, we can (by interpolation) assume that $k \in \mathbb{N}$ is an integer. Let $g=L_-^{-1} f$ and thus
$$
D g + g = Q^2 g + f .
$$ 
Note that $Q \in W^{k, \infty}(\R)$ for any $k \in \mathbb{N}$ by Sobolev embeddings and the fact that $Q \in H^s(\R)$ for all $s \geq 0$. Applying $\nabla^k + 1$ to the equation above and using Leibniz rule and H\"older, we find that
\begin{equation} \label{ineq:Lregbound}
\| g \|_{H^{k+1}_x} \sim \| (\nabla^k +1) (D g +g) \|_{L^2} \lesssim_k \| Q \|_{W^{k,\infty}}^2 \| g \|_{H^k} + \| f \|_{H^k} . 
\end{equation}
Note in particular that $\| g \|_{L^2} = \| L^{-1}_- f \|_{L^2} \lesssim \| f \|_{L^2}$ holds, since $L_-$ has a bounded inverse on $Q^\perp$. Hence \eqref{ineq:Lregbound} shows that the desired regularity estimates is true for $k=0$. By induction, we obtain the desired estimate $\| L^{-1}_- f \|_{H^{k+1}} \lesssim_k \| f \|_{H^k}$ for any integer $k \in \mathbb{N}$.

To show the decay estimate, we argue as follows. Assume that $\| \langle x \rangle^2 f \|_{L^\infty_x} < +\infty$, because otherwise there is nothing to prove. As above, let $g = L^{-1}_- f$ and rewrite the equation satisfied by $g$ in resolvent form:
$$
g = \frac{1}{D+1} Q^2 g + \frac{1}{D+1} f.
$$
Let $R(x-y) = \mathcal{F}^{-1} ( \frac{1}{|\xi| + 1})(x-y)$ denote the associated kernel of the resolvent $(D +1)^{-1}$. From \cite{FrLe2012} we recall the standard fact that $R \in L^p(\R)$ for any $1 < p < \infty$. Since $f \in L^2(\R)$, this implies that $(R \ast f)(x)$ is continuous and vanishes as $|x| \to \infty$.  Moreover (see, e.\,g., \cite{FrLe2012} again) we have the pointwise bound
$$
0 < R(z) \lesssim \frac{1}{|x|^2}, \quad \mbox{for $|x| \geq 1$}.
$$ 
Using this bound and our decay assumption on $f(x)$, it is elementary to check that
$$
| (R \ast f)(x)| \lesssim \min \{1, |x|^{-2} \}  .
$$
Using this bound, we can bootstrap the equation for $g$ and using that $Q^2(x)$ is continuous and vanishes at infinity; we refer to \cite{FrJoLe2007b} for details on a similar decay estimate. This shows that $|g(x)| \lesssim \langle x \rangle^{-2}$ as desired.

\end{proof}

\section{Coercivity Estimates for the Localized Energy}

In the following, we assume that $A > 0$ is a sufficiently large constant. Let $\phi : \R \to \R$ be the smooth cutoff function introduced in Section \ref{sec:refenergy}. For $\eps = \eps_1 + i \eps_2 \in H^{1/2}(\R)$, we consider the quadratic forms
\begin{align} \label{def:LplusA}
L_{+,A}(\eps_1) & :=  \int_{s=0}^\infty \sqrt{s} \int \phi''_A |\nabla \eps_{1s} |^2 \, dx \, ds  + \int |\eps_1|^2 - 3 \int Q^{2} |\eps_1|^2, \\
L_{-,A}(\eps_2) & :=  \int_{s=0}^\infty \sqrt{s} \int \phi''_A |\nabla \eps_{2s} |^2 \, dx \, ds  +\int |\eps_2|^2 - \int Q^2 |\eps_2|^2, \label{def:LminusA}
\end{align}
where $\phi''_A(x) = \phi''(x/A)$. As in Lemma \ref{lem:morawetz}, we denote
\begin{equation} \label{def:f_mu_re}
u_{s} = \sqrt{ \frac{2}{\pi}} \frac{1}{-\Delta + s} u, \quad \mbox{for $s > 0$}.
\end{equation}
We start with the following simple identity.

\begin{lemma} \label{lem:sobo_ident}
For $u \in H^{1/2}(\R)$, we have
\begin{equation} \label{eq:sobo_ident1}
 \int_{s=0}^\infty \sqrt{s} \int |\nabla u_s|^2 \, dx \, ds = \| D^{\frac{1}{2}} u \|_{L^2}^2 .
\end{equation}
\end{lemma}

\begin{proof}
By applying Fubini's theorem and using the Fourier transform, we find that
$$
 \int_{s=0}^{+\infty} \sqrt{s} \int |\nabla u_s|^2 \, dx \, ds = \frac{2}{\pi} \int \int_{s=0}^{+\infty} \frac{\sqrt{s} \, ds}{(\xi^2 + s)^2} |\xi|^2 | \hat{u}(\xi) |^2 \, d \xi = \int |\xi|  | \hat{u}(\xi) |^2 \, d \xi = \| D^{\frac 1 2} f \|_{L^2}^2,
$$ 
which shows the claim. \end{proof}

\begin{remark}
Clearly, the proof of Lemma \ref{lem:sobo_ident}  shows that
\begin{equation} \label{eq:sobo_ident2}
\frac{2}{\pi} \int_{s=0}^{+\infty}  \sqrt{s} \int | D^{\alpha} u_s|^2 \, dx \, ds = \| D^{\alpha-\frac{1}{2}} u \|_{L^2}^2, \quad \mbox{for $u \in \mathcal{S}(\R)$},
\end{equation}
with any exponent $\alpha \in \R$, provided that for $\alpha \leq 0$ we also impose that $\hat{u}(\xi)$ vanishes identically in a neighborhood around $\xi =0$.
\end{remark}

Next, we establish a technical result, which shows that, when taking the limit $A \to +\infty$, the quadratic form $\int_{s \geq0} \sqrt{s} \int \phi''_A |\nabla u_s|^2 \,dx \, ds + \| u \|_{L^2}^2$ defines a weak topology that serves as a useful substitute for the weak convergence in $H^{1/2}(\R)$. The precise statement reads as follows. 

\begin{lemma} \label{lem:lower_TA}
Let $A_n \to +\infty$ and suppose that $\{ u_n \}_{n=1}^\infty$ is a sequence in $H^{1/2}(\R)$ such that
$$
\int_{s=0}^{+\infty} \sqrt{s} \int \phi''_{A_n} |\nabla (u_n)_s|^2 \, dx \, ds +  \| u_n \|_{L^2}^2 \leq C,
$$ 
for some constant $C> 0$ independent of $n \geq 1$. Then, after possibly passing to a subsequence of $\{u_{n} \}_{n=1}^\infty$, we have that
$$
\mbox{$u_n \weakto u$ weakly in $L^2(\R)$ and $u_n \to u$ strongly in $L^2_{\mathrm{loc}}(\R)$,}
$$ 
and $u$ belongs to $H^{1/2}(\R)$. Moreover, we have the bound
$$
\| D^{\frac 1 2} u \|_{L^2}^2 \leq \liminf_{n \to +\infty} \int_{s=0}^{+\infty} \sqrt{s} \int \phi''_{A_n} |\nabla( u_n)_s|^2 \, dx \, ds.
$$
\end{lemma}

\begin{proof}
Let $\zeta\in \mathcal S(\Bbb R)$ be a smooth cutoff function in Fourier space that satisfies
\begin{equation*}
\hat{\zeta}(\xi)=\left\{\begin{array}{ll} 1 & \quad \mbox{for $|\xi|\leq 1$}, \\ 0 & \quad \mbox{for $|\xi|\geq 2$} .\end{array}\right.
\end{equation*}
For any $u \in H^{1/2}(\R)$, we write $u= u^l + u^h$ with 
\begin{equation*}
 \hat{u}^l=\hat{\zeta}\hat{u}, \quad \hat{u}^h=(1-\hat{\zeta})\hat{u}.
\end{equation*}
Recall the definition \eqref{def:f_mu_re}, we readily notice the relations
\begin{equation*}
(u^l)_s=(u_s)^l, \quad  (u^h)_s=(u_s)^h.
\end{equation*}
Hence, we can use the notation $u_s^l = (u^l)_s$ and $u_s^h = (u^l)_s$ in the following.

\subsubsection*{Step 1: Control of $u^h$}

Let $\chi\in \mathcal C^{\infty}_0(\Bbb R)$ be a smooth cutoff function such that 
$$\chi(x)=\left\{\begin{array}{ll} 1 & \quad \mbox{for $|x|\leq 1$}, \\ 0 & \quad \mbox{for $|x|\geq 2$}.\end{array}\right.$$ 
For any $R>0$ given, we set
 $$\chi_R(x)=\chi\left(\frac xR\right).$$ 
We now claim the following control: For any $R>0$, there exist constants $C_R >0$ and $A_0=A_0(R)>0$ such that $\forall A \geq A_0$ and $\forall u\in H^{\frac 12}$, we have
\begin{equation}
\label{controllhigh}
\int |D^{\frac 12}(\chi_Ru^h)|^2\leq C_R \left [ \int_{s=0}^\infty \sqrt{s} \int \phi''_A | \nabla u^h_s|^2 \, dx \, ds +  \| u \|_{L^2}^2 \right ] . 
\end{equation}
Indeed, from definition \eqref{def:f_mu_re} we see that
 $$-\Delta(\chi_Ru^h)_s+s(\chi_Ru^h)_s= \sqrt{ \frac{2}{\pi}} \chi_Ru^h. $$ 
On the other hand, an elementary calculation shows that
\bee
-\Delta (\chi_Ru^h_s)+s\chi_Ru^h_s& = & \chi_R(-\Delta u^h_s+su^h_s)-2\nabla \chi_R\cdot\nabla u^h_s-u^h_s\Delta \chi_R\\
& = & \sqrt{\frac{2}{\pi}} \chi_Ru^h -2\nabla \chi_R\cdot\nabla u^h_s-u^h_s\Delta \chi_R .
\eee
Therefore, the function
\be\label{cnoeeon}
w^s :=\sqrt{\frac{\pi}{2}} \left \{ (\chi_Ru^h)_s-\chi_Ru^h_s \right \}
\ee 
satisfies the equation
$$-\Delta w^s+sw^s=\sqrt{\frac{\pi}{2}} \left \{ 2\nabla\chi_R\cdot u^h_s + u^h_s\Delta \chi_R \right \} .$$ 
Hence, we deduce the bound
$$
\int|\nabla w^s|^2+s\int|w^s|^2 \lesssim   \int\left \{ |\nabla \chi_R||\nabla u^h_s|+|u^h_s||\Delta \chi_R|\right \} |w^s|
$$ 
and, by using the Cauchy--Schwarz inequality, we conclude that
\be
\label{boundtwo}
\int|\nabla w^s|^2+s\int|w^s|^2\lesssim C_R\left\{\int_{|x|\leq 2R}|\nabla u_s^h|^2+\int|u^h_s|^2\right\}, \quad \mbox{for $s \geq 1$},
\ee
\be
\label{boundone}
\int|\nabla w^s|^2+s\int|w_s|^2\lesssim\frac{C_R}{s}\left\{\int|\nabla u_s^h|^2+\int|u^h_s|^2\right\}, \quad \mbox{for $0<s\leq 1$}.
\ee
Next, we apply identity \eqref{eq:sobo_ident2} while noting that $\widehat{u}^h(\xi) =0$ for $|\xi| \leq 1$. For some sufficiently large $A>A_0(R)$, we thus obtain
\bee
\int_{s=1}^{+\infty} \sqrt{s}\int|\nabla w^s|^2 \, dx \, ds&\lesssim& C_R\int_{s=0}^{+\infty} \sqrt{s}\left\{\int_{|x|\leq 2R}|\nabla u_s^h|^2dx+\int|u^h_s|^2 dx\right\}ds\\
& \lesssim & C_R\left[\int_{s=0}^{+\infty}\sqrt{s}\int\phi_{A}''|\nabla u^h_s|^2 \, dx \, ds+\|D^{-\frac 12}u^h\|_{L^2}^2\right]\\
& \lesssim & C_R\left[\int_{s=0}^{+\infty}\sqrt{s}\int\phi_{A}''|\nabla u^h_s|^2 \, dx \, ds+\|u\|_{L^2}^2\right],
\eee
\bee
\int_{s=0}^{1}\sqrt{s}\int|\nabla w^s|^2dxds&\lesssim &C_R \int_{s=0}^1\frac{\sqrt{s}}{s}\int\frac{(1+|\xi|^2)|\hat{u}^h|^2}{(s+|\xi|^2)^2} \, d\xi \, ds\\
& \leq & C_R\int_{s=0}^1\frac{ds}{\sqrt{s}}\int\frac{1+|\xi|^2}{|\xi|^4}|\hat{u}^h|^2 \, d\xi \, ds\lesssim C_R\|u\|_{L^2}^2.
\eee
Using \eqref{eq:sobo_ident1} and the previous bounds, we find that
\bee
\|D^{\frac12}(\chi_Ru^h)\|_{L^2}^2& = & \int_{s=0}^{+\infty}\sqrt{s}\int|\nabla (\chi_Ru^h)_s|^2 \, dx \, ds\\
& \lesssim & \int_{s=0}^{+\infty}\sqrt{s}\int|\nabla w^s|^2 \, dx \, ds+\int_{s=0}^{+\infty}\sqrt{s}\int|\nabla (\chi_R(u^h)_s)|^2 \, dx \, ds\\
& \lesssim &  C_R\left[\int_{s=0}^{+\infty}\sqrt{s}\int\phi_{A}''|\nabla u^h_s|^2 \, dx \, ds+\|u\|_{L^2}^2\right]+\int_{s=0}^{+\infty}\sqrt{s}\int|u^h|^2 \, dx \, ds\\
& \lesssim & C_R\left[\int_{s=0}^{+\infty}\sqrt{s}\int\phi_{A}''|\nabla u^h_s|^2 \, dx \, ds+\|u\|_{L^2}^2\right].
\eee
This completes the proof of estimate \eqref{controllhigh}.

\subsubsection*{Step 2: Conclusion}

Let $\{ u_n \}_{n=1}^\infty$ satisfy the assumptions in Lemma \ref{lem:lower_TA}. By \eqref{eq:sobo_ident1}, we have for all $A>0$ that
\bee
\int_{s=0}^{+\infty}\sqrt{s}\int \phi_{A}''|\nabla (u^l_n)_s|^2 \, dx \, ds & \leq&  \int_{s=0}^{+\infty}\sqrt{s}\int|\nabla (u^l_n)_s|^2 \, dx \, ds= \|D^{\frac 12}u_n^l\|_{L^2}^2\\
& \leq & C\|u_n\|_{L^2}^2\leq C.
\eee
Here we used the frequency localization of $u_n^l$ in the last step. Thus the assumed bound in Lemma \ref{lem:lower_TA} ensures that
\be
\label{assumptionhigh}
\int_{s=0}^{+\infty}\sqrt{s}\int \phi_{{A_n}}''|\nabla (u^h_n)_s|^2 \, dx \, ds\leq C.
\ee
We therefore conclude from \eqref{controllhigh} that, for all $R >0$, the $\{ u_n \}_{n=1}^\infty$ is a bounded sequence in $H^{1/2}(B_R)$ and $L^2(\R)$. Hence, by a simple diagonal extraction argument, we can find $u\in L^2(\Bbb R)$ and we can assume by passing to a subsequence if necessary that
$$
\mbox{$u_{n}\weakto u$ in $L^2(\R)$ and $u_n \weakto  u$ in $H^{1/2}(B_R)$ for all $R >0$}.
$$
By the compactness of the Sobolev embedding $H^{1/2}(\R) \hookrightarrow L^2_{\mathrm{loc}}(\R)$, we also have that
 $$ 
 \mbox{$u_n \to u$ in $L^2_{\mathrm{loc}}(\R)$}.
 $$ 
 
 It remains to show the ``weak lower semicontinuity property'' given by
 \begin{equation} \label{ineq:lwsc}
 \| D^{\frac 12}u\|_{L^2}^2=\int_{s=0}^{+\infty}\sqrt{s}\int|\nabla u_s|^2 \, dx \, ds\leq \liminf_{n\to +\infty}\int_{s=0}^{+\infty}\sqrt{s}\int\phi_{A_{n}}''|\nabla (u_n)_s|^2 \, dx \, ds
 \end{equation}
Indeed, we first we note that
 $$
 \nabla (u_n)_s(x) = \frac{1}{\sqrt{2 \pi}} \int e^{-\sqrt{s} |x-y|} \frac{x-y}{|x-y|} u_n(y) dy.
 $$
 Since $u_n \weakto u$ weakly in $L^2(\R)$ and $e^{-\sqrt{s} |x-y|} \frac{x-y}{|x-y|} \in L^2_y(\R)$ for any $x \in \R$, we thus obtain
 $$
\mbox{$\nabla ( u_n)_s(x) \to \nabla u_s(x)$ pointwise on $\R$ for any $s >0$}.
 $$
 Next,  by the Cauchy--Schwarz inequality, we derive the uniform pointwise bound
 $$
 |\nabla (u_n)_s(x)|\lesssim \| e^{-\sqrt{s} |\cdot|} \|_{L^2} \| u_n \|_{L^2} \lesssim \frac{C}{s^{1/4}} ,
 $$
 using that $\| u_n \|_{L^2} \leq C$ by assumption. Let $0 < \eps < 1$ and $B >0$ now be given. By the dominated convergence theorem, we deduce that
\bee
\int_{s=\eps}^{1/\eps}\sqrt{s}\int_{|x|\leq B}|\nabla u_s|^2 \, dx \, ds & = & \lim_{n\to +\infty}\int_{s=\eps}^{1/\eps}\sqrt{s}\int_{|x|\leq B}|\nabla (u_{n})_s|^2 \, dx \, ds\\
& \leq &  \liminf_{n\to +\infty}\int_{s=0}^{+\infty}\sqrt{s}\int\phi_{A_{n}}''|\nabla (u_{n})_s|^2 \, dx \, ds,
\eee
where in the last step we used Fatou's lemma and the fact that $\phi''_{A_n}(x) \geq 0$ satisfies $\lim_{n \to +\infty} \phi_{A_n}''(x) = 1$ for all $x \in \R$. Since the previous bound holds for arbitrary $0 < \eps < 1$ and $B >0$, we conclude that
$$\| D^{\frac 12}u\|_{L^2}^2=\int_{s=0}^{+\infty}\sqrt{s}\int|\nabla u_s|^2 \, dx \, ds\leq \liminf_{n\to +\infty}\int_{s=0}^{+\infty}\sqrt{s}\int\phi_{A_{n}}''|\nabla (u_n)_s|^2 \, dx \, ds.$$ 
The proof of Lemma \ref{lem:lower_TA} is now complete. \end{proof}

\begin{prop} \label{prop:LA}
Let $L_{+,A}(\eps_1)$ and $L_{-,A}(\eps_2)$ be the quadratic forms defined above. Then there exist universal constants $c_0>0$ and $A_0 > 0$ such that for all $A \geq A_0$ and all $\eps = \eps_1 + i \eps_2 \in H^{1/2}(\R)$ we have the coercivity estimate
$$
(L_{+,A} \eps_1, \eps_1) + ( L_{-,A} \eps_2, \eps_2) \geq c_0 \int |\eps|^2 - \frac{1}{c_0} \left \{ (\eps_1, Q)^2 + (\eps_1, S_1)^2 + (\eps_1, G_1)^2 + (\eps_2, \rho_1)^2 \right \} 
$$ 
Here $S_1$ and $G_1$ are the unique functions such that $L_- S_1 = \Lambda Q$ with $S_1 \perp Q$ and $L_- G_1= -\nabla Q$ with $G_1 \perp Q$, respectively, and the function $\rho_1$ is defined in \eqref{def:rho}.
\end{prop}

\begin{proof}
It suffices to prove the coercivity bound
\begin{equation} \label{ineq:LA_coerc}
(L_{-,A} \eps_2, \eps_2) \geq c_0 \int |\eps|^2 - \frac{1}{c_0} (\eps_2, \rho_1)^2,
\end{equation}
since the corresponding estimate for $L_{+,A}$ follows by the same strategy. 

To prove \eqref{ineq:LA_coerc}, we argue by contradiction as follows. Suppose that there exist a sequence of functions $\{ u_n \}_{n=1}^\infty$ in $H^{1/2}(\R)$ with 
\begin{equation}
\int |u_n|^2 =1, \quad (u_n, \rho_1) = 0,
\end{equation}
as well as a sequence $A_n \to +\infty$ such that
\begin{equation} \label{ineq:contra}
\int_{s=0}^{+\infty} \sqrt{s} \int \phi''_{A_n} |u_n|^2 \, dx \, ds + \int |u_n|^2 - \int Q^2 |u_n|^2 \leq o(1) \int |u_n|^2,
\end{equation}
where $o(1) \to 0$ as $n \to \infty$. By applying Lemma \ref{lem:lower_TA}, we find (after passing to subsequence if necessary) that
\begin{equation}
\mbox{$u_n \weakto u$ weakly in $L^2(\R)$ and $u_n \to u$ strongly in $L^2_{\mathrm{loc}}(\R)$. }
\end{equation}
But since $Q^2(x) \to 0$ as $|x| \to \infty$, we easily check that $\int Q^2 |u_n|^2 \to \int Q^2 |u|^2$. Moreover, from \eqref{ineq:contra} and $\int |u_n|^2 = 1$ we deduce that $\int Q^2 |u|^2 \geq 1$ must hold. In particular, the weak limit $u \not \equiv 0$ is nontrivial. However, by the weak lower semicontinuity inequality in Lemma \ref{lem:lower_TA} and the fact that $\liminf_{n \to \infty} \int |u_n|^2 \geq \int |u|^2$, we deduce that
\begin{equation}
(L_-u,u) = \int |D^{\frac 1 2} u|^2 + \int |u|^2 - \int Q^2 |u|^2 \leq 0, \quad \mbox{where $(u, \rho_1) =0$}.
\end{equation}
Since $u \not \equiv 0$, this bound contradicts the coercivity estimate for $L_-$ stated in Lemma \ref{lem:coerc} below. \end{proof}

We conclude this section with a bound for the error term in localized virial estimate needed in Section \ref{sec:refenergy}.
 
\begin{lemma} \label{lem:phi4}
For any $u \in L^2(\R)$, we have the bound
$$
 \left | \int_{s=0}^{+\infty} \sqrt{s} \int \phi^{(4)}_A  |u_s|^2 \, dx \, d s \right | \lesssim \frac{1}{A} \| u \|_{L^2}^2.
$$
\end{lemma}

\begin{remark} {\em 
Note that a naive application of \eqref{eq:sobo_ident2} would formally yield the bound $\left | \int_{s=0}^\infty \sqrt{s} \int \phi^{(4)} |u_s|^2 \right | \lesssim A^{-2} \| D^{-1/2} u \|_{L^2}^2$.
However, we have that $\| D^{-1/2} u \|_{L^2}=+\infty$ holds in $d=1$, unless $\hat{u}(\xi)$ vanishes appropriately in $\xi = 0$. In fact, the proof of Lemma \ref{lem:phi4} below involves some more careful analysis.}
\end{remark}

\begin{proof}
First, recall that $\phi''_A(x) = \phi'' \left ( \frac{x}{A} \right )$ and hence $\phi^{(4)}(x) = \frac{1}{A^2} \phi^{(4)} \left ( \frac{x}{A} \right )$. Now, we split the $s$-integral as follows
\begin{equation} \label{eq:phi4}
\frac{1}{A^2} \int_{s=0}^{+\infty} \sqrt{s} \int  \phi^{(4)} \left ( \frac{x}{A} \right ) | u_s |^2 \, dx \, d s =: I_{\leq \Lambda} + I_{\geq \Lambda},
\end{equation}
where $\Lambda > 0$ is some given number and we consider
\begin{equation*}
I_{\leq \Lambda} = \frac{1}{A^2} \int_{s=0}^{\Lambda} \sqrt{s} \int  \phi^{(4)} \left ( \frac{x}{A} \right ) | u_s |^2 \, dx \, d s, \quad I_{\geq \Lambda} = \frac{1}{A^2} \int_{s=\Lambda}^{+\infty} \sqrt{s} \int  \phi^{(4)} \left ( \frac{x}{A} \right ) | u_s |^2 \, dx \, d s.
\end{equation*}
Since $\frac{1}{A^2} \phi^{(4)}(y/A) = \Delta_y ( \phi^{(2)}(y/A) )$, we can integrate by parts twice and use the H\"older inequality to deduce that
\begin{align*}
\left | I_{\leq \Lambda} \right | & \lesssim \| \phi^{(2)} \|_{L^\infty} \int_{s=0}^\Lambda \sqrt{s} \left ( \| \Delta u_s \|_{L^2} \| u_s \|_{L^2} + \| \nabla u_s \|_{L^2}^2  \right ) \, d s \\
& \lesssim \int_{s=0}^\Lambda \sqrt{s} \left ( \left \| \frac{-\Delta}{-\Delta + s} u \right \|_{L^2} \left \| \frac{1}{-\Delta + s} u \right \|_{L^2} + \left \| \frac{ \nabla}{-\Delta + s} u \right \|_{L^2}^2  \right ) \, d s \\
& \lesssim \left ( \int_{s=0}^\Lambda \frac{ds}{s^{1/2}} \right ) \| u \|_{L^2}^2  \lesssim \sqrt{\Lambda} \| u \|_{L^2}^2. 
\end{align*}
To estimate $I_{\geq \Lambda}$, we simply use the bound $\| u_s \|_{L^2} \lesssim s^{-1} \| u \|_{L^2}$which shows that
\begin{equation}
\left | I_{\geq \Lambda} \right | \lesssim \frac{1}{A^2} \| \phi^{(4)} \|_{L^\infty} \left ( \int_{s = \Lambda}^{+\infty}  \frac{ds}{s^{3/2}} \right )  \| u \|_{L^2}^2  \lesssim \frac{1}{A^2} \frac{1}{\sqrt{\Lambda}} \| u \|_{L^2}^2.
\end{equation}
Thus, we have shown that, for arbitrary $\Lambda > 0$, 
\begin{equation}
\left | \mbox{LHS of \eqref{eq:phi4}} \right | \lesssim \left ( \sqrt{\Lambda} + \frac{1}{A^2} \frac{1}{\sqrt{\Lambda}} \right ) \| u \|_{L^2}^2 . 
\end{equation}
By minimizing this bound with respect to $\Lambda$, we obtain the desired estimate. \end{proof}

We conclude this section with the following coercivity estimate for $L=(L_-, L_+)$.

\begin{lemma}[Coercivity estimate] \label{lem:coerc}
There exists some universal constant $c_0 > 0$ such that, for any $\eps = \eps_1 + i\eps_2 \in H^{1/2}(\R)$, we have that
$$
(L_+ \eps_1, \eps_1) + (L_-\eps_2,  \eps_2)  \geq c_0 \| \eps \|_{H^{1/2}}^2   - \frac{1}{c_0} \left \{ (\eps_1, Q)^2 + (\eps_1, S_1)^2 + (\eps_1, G_1)^2 + (\eps_2, \rho_1)^2 \right \} .
$$
Here $S_1$ and $G_1$ are the unique functions such that $L_- S_1 = \Lambda Q$ with $S_1 \perp Q$ and $L_- G_1= -\nabla Q$ with $G_1 \perp Q$, respectively, and the function $\rho_1$ is defined in \eqref{def:rho}.
\end{lemma}

\begin{proof}
From \cite{FrLe2012} we recall the key fact that the nullspaces of $L_+$ and $L_+$ are given by
\begin{equation}
\mathrm{ker} \, L_+ = \mathrm{span} \, \{ \nabla Q \}, \quad \mathrm{ker} \, L_- = \mathrm{span} \, \{ Q \} .
\end{equation}
Then, by following arguments in \cite{We1985} ifor ground states for nonlinear Schr\"odinger equations, we deduce the standard coercivity estimate
\begin{equation} \label{ineq:coercstandard}
(L_+ \eps_1,  \eps_1) + (L_- \eps_2, \eps_2) \geq c_1 \| \eps \|_{H^{1/2}}^2 - \frac{1}{c_1} \left \{ (\eps_1, \phi_+)^2 + (\eps_1, \nabla Q)^2 + (\eps_2, Q)^2 \right \}
\end{equation}
for all $\eps = \eps_1 + i \eps_2 \in H^{1/2}(\R)$, where $c_1 > 0$ is some universal constant. Here $\phi_+ = \phi_+(x) > 0$ with $\| \phi_+ \|_{L^2}=1$ denotes the unique ground state eigenfunction of $L_+$, and we have $L_+ \phi_+ = e_+ \phi_+$ with some $e_+ < 0$. (We refer to \cite{FrLe2012} for a detailed discussion of the spectral properties of $L_+$ and $L_-$.) 

To derive the coercivity estimate in Lemma \ref{lem:coerc} from an estimate of the form \eqref{ineq:coercstandard}, we can use some arguments that, e.\,g., can be found in \cite{MeRa2006} in the context of NLS. For the reader's convenience, we provide the details of the adaptation to  our case.  To prove the desired coercivity estimate, we can that assume $\eps = \eps_1 + i \eps_2 \in H^{1/2}(\R)$ satisfies
$$
(\eps_1, S_1) = (\eps_1, G_1) = (\eps_2, \rho_1) = 0.
$$  
Define the function $\hat{\eps} = \hat{\eps}_1 + i \hat{\eps}_2 \in H^{1/2}(\R)$ by setting 
$$
\hat{\eps} = \eps - \alpha \Lambda Q - i \beta Q - \gamma \nabla Q,
$$
where $\alpha, \beta, \gamma \in \R$ are chosen such that
$$
(\hat{\eps}_1, \phi_+) = (\hat{\eps}_2, Q) = (\hat{\eps}_1, \nabla Q) = 0 .
$$
Indeed, we see that 
$$
\alpha = \frac{(\eps_1, \phi_+)}{(\Lambda Q, \phi_+)}, \quad \beta = \frac{(\eps_2, Q)}{(Q,Q)} , \quad \gamma = \frac{(\eps_1, \nabla Q)}{(\nabla Q, \nabla Q)} ,
$$
where we also used  that $(\Lambda Q, \nabla Q) = 0$ holds, since $Q$ is even, and $(\phi_+, \nabla Q) =0$ since $\nabla Q \in \mathrm{ker} \, L_+$ and $\phi_+ \in \mathrm{ran} \, L_+$. Next, recall that $L_+ \phi_+ = e_+ \phi_+$ with $e_+ < 0$ and $L_+ \Lambda Q = -Q$. Hence $(\Lambda Q, \phi_+) = -\frac{1}{e_+} (Q, \phi_+)   >0$, by the strict positivity of  $Q > 0$ and $\phi_+ > 0$. On the other hand, the orthogonality conditions satisfied by $\eps = \eps_1 + i \eps_2$ imply that
$$
\alpha = - \frac{(\hat{\eps}_1, S_1)}{(\Lambda Q, S_1)}, \quad \beta = - \frac{(\hat{\eps}_2, \rho_1)}{(Q, \rho_1)} , \quad \gamma = - \frac{(\hat{\eps}_1, G_1)}{(\nabla Q, G_1)} ,
$$
where we also use that $(\Lambda Q, G_1) = ( \nabla Q, S_1) = 0$, since $Q$ and $S_1$ are even and $G_1$ is odd. Note that $L_- S_1 = \Lambda Q$ and hence $(\Lambda Q, S_1) = (L_- S_1, S_1)  \neq 0$, and $(\nabla Q, G_1) = -(L_- G_1, G_1) < 0$ because of $L_- G_1 = -\nabla Q$. Furthermore, recall that $L_+ \rho_1 = S_1$ and $L_+ \Lambda Q =- Q$. Thus $(Q, \rho_1) = -(\Lambda Q, S_1) = (L_- S_1, S_1) > 0$ again. In summary, we find 
$$
\frac{1}{K} \| \eps \|_{H^{1/2}} \leq \| \hat{\eps} \|_{H^{1/2}} \leq K \| \eps \|_{H^{1/2}} ,
$$
with some universal constant $K> 0$. Now, since $(\Lambda Q, Q)=(\nabla Q, Q) = 0$ and $L_+ \Lambda Q =-Q$ as well as $L_+ \nabla Q =0$ and $L_- Q =0$, we obtain
$$
(\hat{\eps}_1, Q) = (\eps_1, Q), \quad (L_+ \hat{\eps}_1, \hat{\eps}_1) = (L_+ \eps_1, \eps_1) +  \alpha ( \eps_1, Q) , \quad (L_- \hat{\eps}_2, \hat{\eps}_2) = (L_- \eps_2, \eps_2).
$$
By the previous relations and estimate \eqref{ineq:coercstandard}, we conclude
\begin{align*}
 (L_+ \eps_1, \eps_1) + (L_- \eps_2, \eps_2) & = (L_+ \hat{\eps}_1, \hat{\eps}_1) + (L_- \hat{\eps}_2, \hat{\eps}_2) - \alpha(\eps_1, Q) \\
& \geq  c_1 \| \hat{\eps} \|_{H^{1/2}}^2 -  \alpha (\eps_1, Q) \geq c_0 \| \eps \|_{H^{1/2}}^2 - \frac{1}{c_0} (\eps_1, Q)^2,
\end{align*}
with some sufficiently small universal constant $c_0>0$.
\end{proof}

\section{On the Modulation Equations}
\label{sec:modeqn}

Here we collect some results and estimates regarding the modulation theory used in Section \ref{sec:modestimates}.

\subsection{Uniqueness of Modulation Parameters}
First, we show that the parameters $\{ b, v, \lambda, \alpha, \gamma\}$ are uniquely determined if $\eps = \eps_1 + i \eps_2 \in H^{1/2}(\R)$ is sufficiently small and satisfies the orthogonality conditions \eqref{eq:ortho1}--\eqref{eq:ortho5}. Indeed, this follows from an implicit function argument, which we detail here.
  
For $\delta > 0$, let $W_\delta = \{ w \in H^{1/2}(\R) : \| w- Q \|_{H^{1/2}} < \delta \}$. Consider approximate blowup profiles $Q_\Pa$ with $|\Pa|=|(b,v)| < \eta$, where $\eta > 0$ is a small constant. For $w \in W_\delta$, $\lambda_1 > 0$, $y_1 \in \R$, $\gamma_1 \in \R$ and $|\Pa| < \eta$, we define
$$
\eps_{\lambda_1, y_1, \gamma_1, b,v} (y) = e^{i \gamma_1} \lambda_1^{\frac 1 2} w( \lambda_1 y - y_1 ) - Q_\Pa . 
$$
Consider the map $\bm{\sigma} = (\sigma^1, \sigma^2, \sigma^3, \sigma^4, \sigma^5)$ defined by
\begin{align*}
\sigma^1 & =  ( (\eps_{\lambda_1, y_1, \gamma_1, b,v})_1, \Lambda \Theta_\Pa ) - ( (\eps_{\lambda_1, y_1, \gamma_1, b, v})_2, \Lambda \Sigma_\Pa ), \\
\sigma^2 & =  ( (\eps_{\lambda_1,  y_1, \gamma_1, b,v})_1, \partial_b \Theta_\Pa ) - ( (\eps_{\lambda_1, y_1, \gamma_1, b,v})_2, \partial_b \Sigma_\Pa ), \\
\sigma^3 & =  ( (\eps_{\lambda_1, y_1, \gamma_1, b,v})_1, \rho_2 ) - ( (\eps_{\lambda_1, y_1, \gamma_1, b,v})_2, \rho_1 ) , \\
\sigma^4 & =  ( (\eps_{\lambda_1, y_1, \gamma_1, b,v})_1, \nabla \Theta_\Pa) - ( (\eps_{\lambda_1, y_1, \gamma_1, b,v})_2, \nabla \Sigma_\Pa) , \\
\sigma^5 & = ( (\eps_{\lambda_1, y_1, \gamma_1, b,v})_1, \partial_v \Theta_\Pa) - ( (\eps_{\lambda_1, y_1, \gamma_1, b,v})_2, \partial_v \Sigma_\Pa) .
\end{align*}
Recall that $\rho = \rho_1 + i \rho_2$ was defined in \eqref{def:rho}. Taking the partial derivatives at $(\lambda_1, y_1, \gamma_1, b, v) = (1,0,0,0,0)$ yields that
$$
\frac{\partial \eps_{\lambda_1, y_1, \gamma_1, b, v}}{\partial \lambda_1} = \Lambda w, \quad \frac{\partial \eps_{\lambda_1, y_1, \gamma_1, b,v}}{\partial y_1} = -\nabla w, \quad  \frac{\partial \eps_{\lambda_1, y_1, \gamma_1, b,v}}{\partial \gamma_1} = i w,
$$
$$
 \quad \frac{\partial \eps_{\lambda_1, y_1, \gamma_1, b, v}}{\partial b} = - \partial_b Q_\Pa \big |_{\Pa=(0,0)} = - i S_1, \quad \frac{\partial \eps_{\lambda_1, y_1, \gamma_1, b, v}}{\partial v } = -\partial_v Q_\Pa \big |_{\Pa=(0,0)} = -i G_1,
$$
where we recall that $L_- S_1 = \Lambda Q$ and $L_- G_1 = -\nabla Q$. Note that $S_1$ is an even function, whereas $G_1$ is odd. At $(\lambda_1, y_1, \gamma_1, b, v,w) = (1,0,0,0,0,Q)$, the Jacobian of the map $\sigma$ is hence given by
\begin{align*}
\frac{\partial \sigma^1}{\partial \lambda_1} = 0, \quad \frac{\partial \sigma^1}{\partial y_1} = 0, \quad \frac{\partial \sigma^1}{\partial \gamma_1} = 0, \quad \frac{\partial \sigma^1}{\partial b} = - (S_1, L_- S_1), \quad \frac{\partial \sigma^1}{\partial v} = 0,  \\
\frac{\partial \sigma^2}{\partial \lambda_1} = -(L_- S_1,  S_1), \quad \frac{\partial \sigma^2}{\partial y_1} = 0 ,  \quad \frac{\partial \sigma^2}{\partial \gamma_1} = 0, \quad \frac{\partial \sigma^2}{\partial b} = 0, \quad \frac{\partial \sigma^2}{\partial v} = 0, \\
\frac{\partial \sigma^3}{\partial \lambda_1} = 0 , \quad \frac{\partial \sigma^3}{\partial y_1} = 0, \quad  \frac{\partial \sigma^3}{\partial \gamma_1} = -(Q,\rho_1), \frac{\partial \sigma^3}{\partial b}  =0, \quad \frac{\partial \sigma^3}{\partial v}Ê= 0, \\
\frac{\partial \sigma^4}{\partial \lambda_1} = 0 , \quad \frac{\partial \sigma^4}{\partial y_1} = 0, \quad  \frac{\partial \sigma^4}{\partial \gamma_1} = 0, \frac{\partial \sigma^4}{\partial b}  =0, \quad \frac{\partial \sigma^4}{\partial v}Ê= -(L_- G_1, G_1), \\
\frac{\partial \sigma^5}{\partial \lambda_1} = 0 , \quad \frac{\partial \sigma^5}{\partial y_1} = (L_- G_1, G_1), \quad  \frac{\partial \sigma^5}{\partial \gamma_1} = 0, \quad \frac{\partial \sigma^5}{\partial b}  =0, \quad \frac{\partial \sigma^5}{\partial v}Ê= 0. 
\end{align*}
Note that we also used here that $Q$ and $S_1$ are even functions, whereas $G_1$ is odd; e.\,g.,~we have $(Q, G_1) =0$ etc. Moreover, we note 
$$
-(Q, \rho_1) = (L_+ \Lambda Q, \rho_1) = -(\Lambda Q, L_+ \rho_1) = -(\Lambda Q,S_1) = - (L_- S_1, S_1).
$$
Therefore and since $(L_- S_1, S_1) > 0$ and $(L_- G_1, G_1) > 0$, the determinant of the functional matrix is non zero. By the implicit function theorem, we obtain existence and uniqueness for $(\lambda_1, y_1, \gamma_1, b, v, w)$ in some neighborhood around $(1,0,0,0,0,Q)$.

\subsection{Estimates for the Modulation Equations}
To conclude this section, we collect some estimates needed in the discussion of the modulation equations in Section \ref{sec:modestimates}.

\begin{lemma} \label{lem:modequations}
The following estimate hold.
\begin{align}
& (M_-(\eps) - b \Lambda \eps_1 + v \cdot \nabla \eps_1, \Lambda \Theta_\Pa ) + (M_+(\eps) + b \Lambda \eps_2 - v \cdot \nabla \eps_1, \Lambda \Sigma_\Pa) \label{eq:inner1}  \\  & = - \Re(\eps, Q_\Pa) + \Or (\Pa^2 \| \eps \|_{L^2} ), \nonumber \\
& (M_-(\eps) - b \Lambda \eps_1 + v \cdot \nabla \eps_1, \partial_b \Theta_\Pa ) + (M_+(\eps) + b \Lambda \eps_2 - v \cdot \nabla \eps_2,  \partial_b \Sigma_\Pa)  \label{eq:inner2} \\
& =  \Or (\Pa^2 \| \eps \|_{L^2} ), \nonumber \\
& (M_-(\eps) - b \Lambda \eps_1 + v \cdot \nabla \eps_1, \rho_2 ) + (M_+(\eps) + b \Lambda \eps_2 - v \cdot \nabla \eps_2, \rho_1) \label{eq:inner3} \\
& = \Or (\Pa^2 \| \eps \|_{L^2} ), \nonumber \\
& (M_-(\eps) - b \Lambda \eps_1 + v \cdot \nabla \eps_1, \nabla \Theta_\Pa ) + (M_+(\eps) + b \Lambda \eps_2 - v \cdot \nabla \eps_2, \nabla \Sigma_\Pa) \label{eq:inner4} \\
& = \Or ( \Pa^2 \| \eps \|_{L^2} ) , \nonumber \\
& (M_-(\eps) - b \Lambda \eps_1 + v \cdot \nabla \eps_1, \partial_v \Theta_\Pa ) + (M_+(\eps) + b \Lambda \eps_2 - v \cdot \nabla \eps_2, \partial_v \Sigma_\Pa) \label{eq:inner5} \\
& = \Or ( \Pa^2 \| \eps \|_{L^2} ) , \nonumber  
\end{align} 
\end{lemma}

\begin{proof}

First, we recall that
\begin{align*}
M_+(\eps) &= L_+ \eps_1 - 2 \Sigma_\Pa \Theta_\Pa \eps_2 + \Or ( \Pa^2 \eps), \\
M_-(\eps) &= L_- \eps_2 - 2 \Sigma_\Pa \Theta_\Pa \eps_1 + \Or ( \Pa^2 \eps) .
\end{align*}
We divide the proof of \eqref{eq:inner1}--\eqref{eq:inner5} as follows.

\medskip
{\bf Proof of estimate \eqref{eq:inner1}.}
Furthermore, we notice the identity
\begin{equation} \label{eq:infin_poho}
L_- \Lambda S_1 = -S_1 + 2 (\Lambda Q) Q S_1 + \Lambda Q + \Lambda^2 Q. 
\end{equation}
To see this relation, we recall that $L_- S_1 =Ê\Lambda Q$ and hence
\begin{align*}
L_- \Lambda S_1 & =  [L_-, \Lambda] S_1 + \Lambda L_- S_1 = D S_1 + 2 x Q' Q S_1 + \Lambda^2 Q \\
& = - S_1 + Q^2 S_1 + \Lambda Q + 2x Q' Q S_1 + \Lambda^2 Q  \\
& = -S_1 + 2 (\Lambda Q) Q S_1 + \Lambda Q + \Lambda^2 Q,
\end{align*}
as claimed. In a similar fashion, we deduce from $L_- G_1 = -\nabla Q$ that
\begin{equation} \label{eq:infin_poho2}
L_- \Lambda G_1 = -G_1 -\nabla Q + 2 (\Lambda Q) Q G_1 - \Lambda \nabla Q.
\end{equation}
Next, we recall that
$$
\Lambda \Sigma_\Pa = \Lambda Q + \Or (\Pa^2), \quad \Lambda \Theta_\Pa = b \Lambda S_1 + v \Lambda G_1 + \Or(\Pa^2),
$$
Combining \eqref{eq:infin_poho} and \eqref{eq:infin_poho2} with this fact and using that $L_+ \Lambda Q = -Q$, we find that
\begin{align*}
\mbox{LHS of \eqref{eq:inner1}} & =  ( \eps_1, L_+ \Lambda Q) + b (\eps_2, L_- \Lambda S_1) + v (\eps_2, L_- \Lambda G_1) \\
& \quad  - 2b (Q S_1 \eps_2, \Lambda Q) - 2v (Q G_1 \eps_2, \Lambda Q) - b (\eps_2, \Lambda^2 Q)  + v (\eps_2, \nabla \Lambda Q) \\
& \quad + \Or (\Pa^2 \| \eps \|_{L^2} ) \\
& = - (\eps_1, Q) - b (\eps_2, S_1) - v(\eps_2, G_1)  + b (\eps_2, \Lambda Q)  - v(\eps_2, \nabla Q)  + \Or (\Pa^2 \| \eps \|_{L^2} ) \\
&= -  \Re (\eps, Q_b) + \Or ( \Pa^2 \| \eps \|_{L^2}) .
\end{align*}
Here we also used that $b (\eps_2, \Lambda Q) = \Or (\Pa^2 \| \eps \|_{L^2})$ and $v (\eps_2, \nabla Q) = \Or(\Pa^2 \| \eps \|_{L^2}) $, which follows from the orthogonality conditions \eqref{eq:ortho1} and \eqref{eq:ortho4}, respectively. This completes the proof of \eqref{eq:inner1}.

\medskip
{\bf Proof of estimate \eqref{eq:inner2}.}
Here we argue as follows. From the proof of Proposition \ref{prop:Qb_existence} we recall that
$$
\partial_b \Sigma_\Pa = 2b T_2 + v  F_2, \quad \partial_b \Theta_\Pa = S_1 + \Or (b^2) ,
$$
where
$$
 L_+ T_2 = \frac{1}{2} S_1 -\Lambda S_1 + S_1^2 Q, \quad L_+ F_2 = G_1-\Lambda G_1 + \nabla S_1 + 2 G_1 S_1 Q
$$
Using these facts, we compute
\begin{align*}
\mbox{LHS of \eqref{eq:inner2}} & =  (\eps_2, L_- S_1) - 2b (S_1 Q \eps_1, S_1) - 2v (\eps_1 G_1 Q, S_1) + b(\eps_1, \Lambda S_1) - v(\eps_1, \nabla S_1) \\
& \quad + 2b (\eps_1, L_+ T_2) + v (\eps_1, L_+ F_2) + \Or(\Pa^2 \| \eps \|_{L^2} ) \\
& = (\eps_2, \Lambda Q) - 2b (\eps_1, S^2_1 Q) - 2v (\eps_1, Q G_1 S_1 ) + b (\eps_1, \Lambda S_1) - v (\eps_1, \nabla S_1) \\
& \quad + 2b (\eps_1, \frac{1}{2} S_1 - \Lambda S_1 + S_1^2 Q ) + v (\eps_1, G_1 - \Lambda G_1 + \nabla S_1 + 2 G_1 S_1 Q) \\
& \quad + \Or(\Pa^2 \| \eps \|_{L^2} ) \\
& = (\eps_2, \Lambda Q) - b (\eps_1, \Lambda S_1)  - v (\eps_1, \Lambda G_1)  + v( \eps_1, G_1) - + \Or (\Pa^2 \| \eps \|_{L^2} )  \\
& = (\eps_2, \Lambda \Sigma_\Pa) - (\eps_1, \Lambda \Theta_\Pa) + \Or (\Pa^2 \| \eps \|_{L^2} ).
\end{align*}
In the last step we also used that $v(\eps_1, G_1) = \Or(\Pa^2 \| \eps \|_{L^2})$ thanks to the orthogonality condition \eqref{eq:ortho5}. This completes the proof of \eqref{eq:inner2}.

\medskip
{\bf Proof of estimate \eqref{eq:inner3}.}
We now turn to the proof of estimate \eqref{eq:inner3}. Indeed, by recalling \eqref{def:rho}, we find that
\begin{align*}
\mbox{LHS of \eqref{eq:inner3}} & = (\eps_2, L_- \rho_2) + (\eps_1, L_+ \rho_1) - 2b (\eps_2, Q S_1 \rho_1) - 2v (\eps_2, Q G_1 \rho_1) - b (\eps_2, \Lambda \rho_1) \\
& \quad + v (\eps_2, \nabla \rho_1) + \Or (\Pa^2 \| \eps \|_{L^2} ) \\
& = 2b (\eps_2,  Q S_1 \rho_1) + b (\eps_2, \Lambda \rho_1) - 2b (\eps_2, T_2) + 2v (\eps_2, Q G_1 \rho_1) - v (\eps_2, \nabla \rho_1) - v ( \eps_2, F_2) \\
& \quad + (\eps_1, S_1) - 2b (\eps_2, Q S_1 \rho_1) - 2v (\eps_2, Q G_1 \rho_1) - b (\eps_2, \Lambda \rho_1)  + v (\eps_2, \nabla \rho_1) \\
& \quad + \Or (\Pa^2 \| \eps \|_{L^2} ) \\
& = -2b (\eps_2, T_2) - v( \eps_2, F_2) + (\eps_1, S_1) + \Or (\Pa^2 \| \eps \|_{L^2} ) \\
& = - (\eps_2, \partial_b \Sigma_\Pa) + (\eps_1, \partial_b \Theta_\Pa) + \Or (\Pa^2 \| \eps \|_{L^2}) = \Or (\Pa^2 \| \eps \|_{L^2} ),
\end{align*}
using the orthogonality condition \eqref{eq:ortho2}. The proof of \eqref{eq:inner3} is now complete.

\medskip
{\bf Proof of estimate \eqref{eq:inner4}.}
First, we note that
\begin{equation}
\nabla \Sigma_\Pa = \nabla Q + \Or ( \Pa^2), \quad \nabla \Theta_\Pa = b \nabla S_1 + v \nabla G_1 + \Or(\Pa^2) .
\end{equation}
Moreover, we have the relations
\begin{equation}
L_+ \nabla Q = 0, \quad L_- \nabla S_1 = 2 (\nabla Q) Q S_1 + \nabla \Lambda Q, \quad L_- \nabla G_1 = 2 (\nabla Q) Q G_1 - \nabla^2 Q,
\end{equation}
which are obtained in an analogous way as done to show \eqref{eq:infin_poho} and \eqref{eq:infin_poho2}. Thus we obtain
\begin{align*}
\mbox{LHS of \eqref{eq:inner4}} & = b (\eps_2, L_- \nabla S_1) + v (\eps_2, L_- \nabla G_1) +  (\eps_1, L_+ \nabla Q) \\
& \quad - 2b (\eps_2 Q S_1, \nabla Q) - 2v (\eps_2 Q G_1 , \nabla Q) - b (\eps_2, \Lambda \nabla Q) + v(\eps_2, \nabla^2 Q) + \Or (\Pa^2 \| \eps \|_{L^2} ) \\
& = 2 b (\eps_2, (\nabla Q) Q S_1) + b (\eps_2,  \nabla \Lambda Q) + 2v(\eps_2, (\nabla Q) Q G_1) - v (\eps_2, \nabla^2 Q) \\
& \quad - 2b (\eps_2 Q S_1, \nabla Q) - 2v (\eps_2 Q G_1, \nabla Q) - b(\eps_2, \Lambda \nabla Q) + v (\eps_2, \nabla^2 Q) + \Or (\Pa^2 \| \eps \|_{L^2} ) \\
& =b  (\eps_2, [\nabla, \Lambda] Q) + \Or(\Pa^2 \| \eps \|_{L^2})  = b (\eps_2, \nabla Q) + \Or (\Pa^2 \| \eps \|_{L^2}) \\
& = \Or (\Pa^2 \| \eps \|_{L^2} ),
\end{align*}
since $b (\eps_2, \nabla Q) = \Or(\Pa^2 \| \eps \|_{L^2})$ due to condition \eqref{eq:ortho4}. This shows that \eqref{eq:inner4} holds.

\medskip
{\bf Proof of estimate \eqref{eq:inner5}.}
Here we notice that
\begin{equation}
\partial_v \Sigma_\Pa = b F_2 + 2 v H_2, \quad \partial_v \Theta_\Pa = G_1 ,
\end{equation}
where
\begin{equation}
L_+ H_2 = \nabla G_1 + G_1^2 Q .
\end{equation}
Using the relations above, we thus obtain 
\begin{align*}
\mbox{LHS of \eqref{eq:inner5}} & = (\eps_2, L_- G_1) - 2b (\eps_1 Q S_1, G_1) - 2v (\eps_1 Q G_1, G_1) + b(\eps_1, \Lambda G_1) \\
& \quad -v(\eps_1, \nabla G_1)  + b ( \eps_1, L_+ F_2 ) + 2v (\eps_2, L_+ H_2) + \Or (\Pa^2 \| \eps \|_{L^2} ) \\
& = -(\eps_2,\nabla Q) - 2b ( \eps_1 Q S_1, G_1) - 2v (\eps_1 Q G_1, G_1) + b (\eps_1, \Lambda G_1) \\
& \quad -v(\eps_1, \nabla G_1) + b (\eps_1, G_1 - \Lambda G_1 + \nabla S_1 + 2 G_1 S_1 Q)  \\
& \quad +2v (\eps_1, \nabla G_1 + G_1^2 Q ) + \Or (\Pa^2 \| \eps \|_{L^2} ) \\
& = -(\eps_2, \nabla Q) + b (\eps_1, \nabla S_1) + v (\eps_1, \nabla G_1) + \Or (\Pa^2 \| \eps \|_{L^2} ) \\
& = -(\eps_2, \nabla \Sigma_\Pa) + (\eps_1, \nabla \Theta_\Pa) + \Or (\Pa^2 \| \eps \|_{L^2}),
\end{align*}
thanks to the orthogonality condition \eqref{eq:ortho4}. This completes the proof of \eqref{eq:inner5} and hence we have proven that Lemma \ref{lem:modequations} holds.
\end{proof}

\section{The Cauchy Problem}

\label{sec:cauchy}

We have the following local well-posedness result concerning the Cauchy problem for the $L^2$-critical half-wave equation \eqref{eq:wave}. In fact, the proof of the following well-posedness result for problem \eqref{eq:wave} can be deduced in a verbatim fashion as for the so-called {\em cubic Szeg\"o equation} treated in \cite{GeGr2010}. We have the following result, where we only consider forward times, which is no restriction due to the time-reversibility of \eqref{eq:wave}.

\begin{thm} \label{thm:lwp}
Let $s \geq 1/2$ be given. For every initial datum $u_0 \in H^s(\R)$, there exists a unique solution $u \in C^0([t_0,T); H^s(\R))$ of problem \eqref{eq:wave}. Here $t_0 < T(u_0) \leq +\infty$ denotes its maximal time of existence (in forward time). Moreover, we have the following properties.
\begin{itemize}
\item[(i)] {\bf Conservation of $L^2$-mass, energy and linear momentum:} It holds that
$$
 M(u) = \int |u|^2 , \ \ E(u) = \frac{1}{2} \int |D^{\frac 1 2} u|^2 - \frac{1}{4} \int |u|^2, \ \ P(u) = \int \overline{u} (-i \partial_x u) ,
 $$
are conserved along the flow.\\[0.5ex]

\item[(ii)] {\bf Blowup alternative in $H^{1/2}$:} Either $T(u_0) = +\infty$ or if $T(u_0) < +\infty$ then $\| u(t) \|_{H^{1/2}} \to +\infty$ as $t \to T^-$.\\[0.5ex]

\item[(iii)] {\bf Continuous dependence:} If $s > 1/2$, then the flow map $u_0 \mapsto u(t)$ is Lipschitz continuous on bounded subsets of $H^s(\R)$.\\[0.5ex]

\item[(iv)] {\bf Global Existence for Small Data:} If $u_0 \in H^s(\R)$ satisfies $\| u_0 \|_{L^2} < \| Q \|_{L^2}$, then $T(u_0) = +\infty$ holds true.\\[0.5ex]  
\end{itemize}
\end{thm}

\begin{proof} Without loss of generality we assume that $t_0=0$ holds. Consider the corresponding integral equation
\begin{equation} \label{eq:duhamel}
u(t) = e^{-i t D} u_0 - i \int_0^t e^{-i(t-t') D} |u(t')|^2 u(t') \,dt' .
\end{equation}
We discuss the cases of initial data in $H^{s}(\R)$ with $s > 1/2$ first. Below, we indicate how to treat the borderline case $s=1/2$.
 
\medskip
{\bf Case $s > 1/2$.} First, we suppose that $s > 1/2$ holds. In this case, the Sobolev embedding $\| u \|_{L^\infty} \leq C_s \| u \|_{H^{s}}$ in $\R$ shows that the nonlinearity $u \mapsto |u|^2 u$ is Lipschitz on bounded subsets of $H^s(\R)$. Hence, local existence and uniqueness of $u \in C^0([0,T); H^s(\R))$ follows from a simple fixed point argument, provided that $s > 1/2$ holds. Also, continuous dependence of $u(t)$ with respect to the initial datum $u_0$ in $H^s(\R)$ as expressed in (iii) follows by standard arguments, using that $u \mapsto |u|^2 u$ is locally Lipschitz on $H^s(\R)$. To prove (i), we note that a calculation shows $\frac{d}{dt} E(u(t)) =0$ and $\frac{d}{dt} M(u(t))=0$, assuming that we have initial data in $H^2(\R)$ so that $E(u(t))$ and $M(u(t)$) are $C^1$ in $t$. By a standard approximation argument and local wellposedness in $H^s(\R)$ for $s > 1/2$, we conclude that $E(u(t))$ and $M(u(t))$ are also conserved for initial data in $H^s(\R)$ for $s > 1/2$. 

To complete the proof of Theorem \ref{thm:lwp} for the case $s > 1/2$, we have to show that property (ii) holds. Indeed, this can be seen as follows. From standard theory of semilinear evolution equations with locally Lipschitz perturbations, we have the blowup alternative in $H^s(\R)$. That is, if $u \in C^0([0,T); H^s(\R))$ has the maximal time of existence $T(u_0) < +\infty$, then $\| u(t) \|_{H^s} \to +\infty$ as $t \to T^-$. Suppose now that $T(u_0) < +\infty$ and assume that $K= \sup_{t \in [0,T)} \| u(t) \|_{H^{1/2}} < +\infty$ holds. We show that this implies $\tilde{K} = \sup_{t \in [0,T)} \| u(t) \|_{H^s} <+\infty$ as well, which would prove that (ii) holds. In fact, from \eqref{eq:duhamel} and invoking Lemma \ref{lem:brezis}, we conclude that
\begin{align*}
\| u(t) \|_{H^s} & \leq \| u_0 \|_{H^s} + \int_0^t \| |u(t')|^2 u(t') \|_{H^s} \, dt'  \leq \| u_0 \|_{H^s} + C \int_0^t \| u(t') \|_{L^\infty}^2 \| u(t') \|_{H^s} \, dt' \\
& \leq \| u_0 \|_{H^s} + C K^2 \int_0^t \left [ \log \left ( 2 + \frac{\| u(t') \|_{H^s}}{K}   \right ) \right ] \| u(t') \|_{H^s} \, dt' .
\end{align*}
Note here the fact that $z^2 \log (1+ a/z) \leq K^2 \log (1 + a/K)$ if $0 \leq z \leq K$ and $a \geq 0$. If we let $f(t) := \| u(t) \|_{H^s}/K$, we obtain the integral inequality
\begin{equation}
f(t) \leq f(0) + C \int_0^t \left [ \log ( 2 + f(t')) \right ] f(t') \, dt'.
\end{equation}  
By Gronwall's lemma, this implies
\begin{equation}
2 + f(t) \leq (2+f(0))^{e^{C t}}, \quad \mbox{for $t \in [0,T)$},
\end{equation}
which shows that $\sup_{t \in [0,T)} \| u(t) \|_{H^s} < +\infty$ holds. 

Finally, by recalling \eqref{eq:gwpbound}, it is easy to see that initial data $\| u_0 \|_{L^2} < \| Q \|_{L^2}$ are a-priori bounded in $H^{1/2}$ and hence $u(t)$ extends globally in time, thanks to the blowup alternative shown above. This completes the proof of Theorem \ref{thm:lwp} for $s > 1/2$.

\medskip
{\bf Case $s=1/2$.} In the limiting case when $s=1/2$ holds, we need a more refined analysis of the problem. In fact, this can be done in an similar fashion as for the Cauchy problem for the cubic Szeg\"o equation mentioned above; see \cite{GeGr2010}. For the reader's convenience, we give a brief sketch of the main arguments that treat the borderline case $s=1/2$ as follows.

First, we can obtain a weak solution $u \in C_w([0,T); H^{1/2}(\R))$ by an approximation and compactness argument.

Then, we show uniqueness by an argument basically due to Judovic \cite{Ju1963}; see also \cite{Og1990}. More precisely, by using Lemma \ref{lem:trud} below, the quantity $g(t) = \| u(t) - \tilde{u}(t) \|_{L^2}^2$ is found to satisfy
$$
|g'(t)| \lesssim \left ( \| u(t) \|_{L^{2 (p+1)}}^{2(1+ \frac 1 p )} + \| \tilde{u}(t)Ê\|_{L^{2 (p+1)}}^{2(1+ \frac 1 p)} \right ) \| u(t) - \tilde{u}(t) \|_{L^2}^{2 ( 1- \frac 1 p) }  \leq C p g(t)^{1- \frac 1 p},
$$
for any exponent $p > 2$ and where $C > 0$ is some constant depending only on the bound $\sup_{t \in I} \{ \| u(t) \|_{H^{1/2}}, \| \tilde{u}(t) \|_{H^{1/2}} \}$ with $I$ being any compact time interval of existence including $t=0$. Thus if $g(0)=0$, we obtain that 
$$g(t) \leq (C t)^p,$$ 
by integrating the previous bound. In particular, we see that $g(t) \to 0$ for any $t < 1/C$ as $p \to +\infty$. Hence we deduce that $g(t) \equiv 0$ for $t < 1/C$, provided that $g(0)=0$. Repeating the argument in time if necessary, we deduce uniqueness of the weak solution $u \in C_w([0,T); H^{1/2}(\R))$ solving \eqref{eq:wave}.

Finally, we upgrade $u \in C_w([0,T); H^{1/2}(\R))$ to $u \in C^0([0,T); H^{1/2}(\R)$ by a standard argument using weak convergence and the time reversibility of the flow. Also, the proof of continuous dependence in $H^{1/2}(\R)$ follows from standard arguments. This completes our sketch of the proof of Lemma \ref{thm:lwp}. \end{proof}

We conclude the present section with some fundamental estimates related for the space $H^{1/2}(\R)$. (See also \cite{GeGr2010} for similar statements and proofs in the periodic setting.) 

\begin{lemma} \label{lem:brezis}
For  $s > 1/2$ and $u \in H^s(\R)$, we have 
$$
\| u \|_{L^\infty} \leq C_s \| u \|_{H^{1/2}} \left [ \log \left ( 2 + \frac{\| u \|_{H^s}}{\| u \|_{H^{1/2}}} \right ) \right ]^{1/2} ,
$$
where $C_s > 0$ is some constant that only depends on $s > 1/2$. 
\end{lemma}

\begin{proof} This follows from standard arguments in the literature. For the reader's convenience, we reproduce the proof here. For every $\Lambda > 0$ fixed, we deduce that
\begin{align*}
\| u \|_{L^\infty} & \lesssim  \int_{|\xi| \leq \Lambda} |\hat{u}(\xi)| \, d \xi + \int_{|\xi| \geq \Lambda} |\hat{u}(\xi)| \, d \xi \\
& \lesssim \int_{|\xi| \leq \Lambda} ( 1+ |\xi| )^{1/2} \frac{|\hat{u}(\xi)|}{(1+ |\xi|)^{1/2}} \, d \xi + \int_{|\xi| \geq \Lambda} (1+ |\xi|)^{s} \frac{|\hat{u}(\xi)|}{(1+|\xi|)^{s}} \, d \xi \\
& \lesssim \| u \|_{H^{1/2}} \left ( \int_{|\xi| \leq \Lambda} \frac{d \xi}{1+ |\xi|} \right )^{1/2} + \| u \|_{H^s} \left ( \int_{|\xi| \geq \Lambda} \frac{d \xi}{(1+|\xi|)^{2s}} \right )^{1/2} \\
& \lesssim  \left ( \| u \|_{H^{1/2}} \log (\Lambda +1)^{1/2} + \| u \|_{H^s} \Lambda^{-s+1/2} \right ).
\end{align*}
By minimizing this bound with respect to $\Lambda > 0$, we obtain the desired inequality.
\end{proof}

\begin{lemma} \label{lem:trud}
For any $u \in H^{1/2}(\R)$ and $2 < p < +\infty$, it holds that
$$
\| u \|_{L^p} \leq C p^{1/2} \| u \|_{H^{1/2}},
$$
where the constant $C > 0$ is independent of $p$ and $u$.
\end{lemma}

\begin{proof} This follows from standard arguments in the literature. For the reader's convenience, we present the details. Let $\mu(\cdot)$ denote the Lebesgue measure on $\R$. We have the general formula
$$
\| u \|_{L^p}^p = p \int_0^\infty t^{p-1} \mu ( \{ x : |u(x)|  \geq t \} ) \, dt.
$$ 
Without loss of generality, we will  assume that $\| u \|_{H^{1/2}} = 1$ in what  follows. Next, we write $u=u_{\leq \Lambda} + u_{\geq \Lambda}$ where $u_{\leq \Lambda}(x) = \frac{1}{\sqrt{2 \pi}} \int_{|\xi| \leq \Lambda} \hat{u}(\xi) e^{i \xi x} \, d \xi$. For any $t \geq 0$, let us choose $\Lambda = \Lambda_t \geq 0$ such that $\| u_{\leq \Lambda} \|_{L^\infty} \leq t/2$. Indeed, note that
\begin{align*}
\| u_{\leq \Lambda} \|_{L^\infty} & \lesssim \int_{|\xi| \leq \Lambda} |\hat{u}(\xi)| \, d \xi \lesssim \left ( \int_{|\xi| \leq \Lambda} |\hat {u}(\xi) |^2 \, d \xi \right )^{1/2} \cdot \log (\Lambda+1)^{1/2} \\
& \lesssim \| u \|_{H^{1/2}} \log (\Lambda+1)^{1/2} = c \log (\Lambda + 1)^{1/2},
\end{align*}
where $c > 0$ is some universal constant (and note that $\| u \|_{H^{1/2}} = 1$ by assumption). Hence, for any $t \geq 0$, we can always find $\Lambda = \Lambda_t$ to ensure that $\| u_{\leq \Lambda} \|_{L^\infty} \leq t/2$. Making this choice, we find that
\begin{align*}
\| u \|_{L^p}^p & \leq p \int_{0}^\infty t^{p-1} \mu ( \left \{ x : |u_{\geq \Lambda_t}| \geq t/2 \right \} ) \, dt  \leq p \int_0^\infty t^{p-3} \| u_{\geq \Lambda_t} \|_{L^2}^2  \, dt   \\
& \leq p \int_0^\infty t^{p-3} \int_{|\xi| \geq \Lambda_t} |\hat{u}(\xi)|^2 \, d \xi \, dt \leq p \int \left ( \int_0^{2 \log (|\xi|+1)^{1/2}} t^{p-3} \, dt \right ) |\hat{u}(\xi)|^2 \, d \xi \\
& \leq \frac{p}{p-2} \int (\log (|\xi|+1))^{(p-2)/2} |\hat{u}(\xi)|^2 \, d\xi \\
& \lesssim \frac{p}{p-2} \left ( \frac{p-2}{2} \right )^{\frac{p-2}{2}} \int (|\xi|^2 +1)^{1/2} |\hat{u}(\xi)|^2 \, d \xi \lesssim p^{p/2} \| u \|_{H^{1/2}}^2 \lesssim p^{p/2}.
\end{align*}
Here we used the bound $(\log (|\xi|+1))^\ell \lesssim \ell^\ell (|\xi|^2 +1)^{1/2}$ for $\ell > 0$. By taking the $1/p$-th power on both side, we obtain the claimed inequality. \end{proof}

\section{Completion of the Proof of Lemma~\ref{lem:back}}

\label{sec:higherHs}

Here we improve the bound \eqref{ineq:cont_bound1.5}, thus completing {\bf Step 6} in the proof of Lemma~\ref{lem:back}. We achieve this by using a Fourier-theoretic method.  Our point of departure is again the identity
\[
i\partial_t\tilde{u}= D \tilde{u}- |\tilde{u}|^2\tilde{u} - \psi - F,
\]
where we have 
\[
F = |\tilde{u}+\tilde{Q}|^2(\tilde{u}+\tilde{Q}) - |\tilde{Q}|^2\tilde{Q} - |\tilde{u}|^2\tilde{u},\quad \tilde{Q}: = \frac{1}{\lambda^{\frac 1 2}(t)} Q_\Pa \left ( (\frac{x-\alpha(t)}{\lambda(t)} \right )e^{i\gamma(t)}
\]
We plan to obtain a $H^{\frac{1}{2}+\eps}$-bound on $\tilde{u}$ for $\eps>0$ sufficiently small, taking advantage of the a priori bounds at time $t_1$ and those assumed for $t\in [t_0, t_1]$. Consider 
\begin{equation}\label{eq:bound0}\begin{split}
i\frac{d}{dt}\left ( D^{\frac{1}{2}+\eps}\tilde{u}, D^{\frac{1}{2}+\eps}\tilde{u} \right ) 
&=i\Im\left (  D^{\frac{1}{2}+\eps}\big[ D \tilde{u}- |\tilde{u}|^2\tilde{u} - \psi - F\big], D^{\frac{1}{2}+\eps}\tilde{u}\right ) \\
&=i\Im\left (  - D^{\frac{1}{2}+\eps}\big[|\tilde{u}|^2\tilde{u} +\psi + F\big], D^{\frac{1}{2}+\eps}\tilde{u}\right ). \\
\end{split}\end{equation}
We commence with the contribution of that part of $F$ which is linear in $\tilde{u}$. Thus we have to estimate the expression 
\[
\Im \left (  D^{\frac{1}{2}+\eps}\big[2\Re(\tilde{u}\overline{\tilde{Q}})\tilde{Q} + |\tilde{Q}|^2\tilde{u}\big],  D^{\frac{1}{2}+\eps}\tilde{u}\right ) .
\]
In order to control this, we need to bound expressions of the form 
\[
D^{\alpha}(f g) - f(D^{\alpha}g),\quad \alpha\in [0,1] .
\]
We claim the bound 
\[
\|D^{\alpha}(f g) - f(D^{\alpha}g)\|_{L^2}\lesssim \|D^\alpha f\|_{L^2} \|\hat{g}\|_{L^1},\quad \alpha\in [0,1] .
\]
This follows from Plancherel's theorem and the identity 
\[
\widetilde{D^{\alpha}(f g)}(\xi) - \widetilde{f(D^{\alpha}g)}(\xi) = \int_{\R}(|\xi|^\alpha - |\eta|^\alpha)\hat{f}(\xi-\eta)\hat{g}(\eta)\,d\eta ,
\]
and we have 
\[
\big| |\xi|^\alpha - |\eta|^\alpha\big|\leq |\xi - \eta|^\alpha,\quad \alpha\in [0,1] .
\]
In particular, we find 
\[
\big|\int_{\R}(|\xi|^\alpha - |\eta|^\alpha)\hat{f}(\xi-\eta)\hat{g}(\eta)\,d\eta\big|\leq \int_{\R}|\xi - \eta|^\alpha |\hat{f}|(\xi-\eta)|\hat{g}|(\eta)\,d\eta,\quad \alpha\in [0,1] ,
\]
whence 
\[
\|\int_{\R}(|\xi|^\alpha - |\eta|^\alpha)\hat{f}(\xi-\eta)\hat{g}(\eta)\,d\eta\|_{L_{\xi}^2}\leq \min\{\|D^{\alpha}f\|_{L^2}\|\hat{g}\|_{L^1},\,\|\widehat{D^{\alpha}f}\|_{L^1}\| g\|_{L^2}\},\quad \alpha\in [0,1] .
\]
Further, we recall the elementary fractional Leibniz rule 
\[
\| D^\alpha(fg)\|_{L^2}\lesssim \| D^{\alpha}f\|_{L^2} \|g\|_{L^\infty} + \| D^{\alpha}g\|_{L^2} \|f\|_{L^\infty},\quad \alpha \geq 0.
\]
We immediately infer that 
\begin{equation}\label{eq:bound1}\begin{split}
&\big|\Im \left (  D^{\frac{1}{2}+\eps} \big(|\tilde{Q}|^2\tilde{u}\big),  D^{\frac{1}{2}+\eps}\tilde{u}\right ) \big| \\&= 
\big|\Im \left ( \big(|\tilde{Q}|^2 D^{\frac{1}{2}+2\eps} \tilde{u}\big),  D^{\frac{1}{2}}\tilde{u}\right ) \big| + O(\|\widetilde{D^{\frac{1}{2}+2\eps}|\tilde{Q}|^2}\|_{L^1}\|\tilde{u}\|_{L^2}\|D^{\frac{1}{2}}\tilde{u}\|_{L^2}) .\\
\end{split}\end{equation}
We can estimate the right-hand term by 
\[
O(\|\widetilde{D^{\frac{1}{2}+2\eps}|\tilde{Q}|^2}\|_{L^1}\|\tilde{u}\|_{L^2}\|D^{\frac{1}{2}}\tilde{u}\|_{L^2})\lesssim \lambda^{-\frac{3}{2}-2\eps}\lambda \lambda^{\frac{1}{2}}\lesssim \lambda^{-2\eps},
\]
which is integrable for $\eps$ small enough. Next, consider the more delicate term 
\[
\big|\Im \left ( \big(|\tilde{Q}|^2 D^{\frac{1}{2}+2\eps} \tilde{u}\big),  D^{\frac{1}{2}}\tilde{u}\right ) \big| .
\]
Here the key is to exploit a cancellation: Writing $\widehat{D^{\frac{1}{2}}\tilde{u}}(\xi) = f(\xi)$, we find 
\begin{align*}
2i\Im\mathcal{F}\big(D^{\frac{1}{2}+2\eps} \tilde{u}\overline{D^{\frac{1}{2}}\tilde{u}}\big)(\xi)& = \int_{\R}|\xi-\eta|^{2\eps}[\hat{f}(\xi-\eta)\overline{\hat{f}(\eta)} - \overline{\hat{f}}(\xi-\eta)\hat{f}(\eta)] \,d\eta\\
&=\int_{\R}[|\xi-\eta|^{2\eps} - |\eta|^{2\eps}]\hat{f}(\xi-\eta)\overline{\hat{f}(\eta)}\,d\eta .
\end{align*}
It follows from Plancherel's theorem that
\begin{equation}\label{eq:bound2}\begin{split}
\big|\Im \left ( \big(|\tilde{Q}|^2 D^{\frac{1}{2}+2\eps} \tilde{u}\big),  D^{\frac{1}{2}}\tilde{u}\right ) \big|
&=\big|\int_{\R}\mathcal{F}(|\tilde{Q}|^2)(\xi)\int_{\R}[|\xi-\eta|^{2\eps} - |\eta|^{2\eps}]\hat{f}(\xi-\eta)\overline{\hat{f}(\eta)}\,d\eta d\xi\big|\\
&\leq \||\xi|^{2\eps}\mathcal{F}(|\tilde{Q}|^2)(\xi)\|_{L_{\xi}^1}\|f\|_{L^2}^2\lesssim \lambda^{-1-2\eps}\lambda = \lambda^{-2\eps}.
\end{split}\end{equation}

Next, we consider the term 
\[
\Im \left (  D^{\frac{1}{2}+\eps}\big[2\Re(\tilde{u}\overline{\tilde{Q}})\tilde{Q}\big],  D^{\frac{1}{2}+\eps}\tilde{u}\right ) 
\]
The challenge consists again in moving the extra $2\eps$ derivatives away from the function $\tilde{u}$. To this end, we write 
\begin{align*}
\Im \left (  D^{\frac{1}{2}+\eps}\big[2\Re(\tilde{u}\overline{\tilde{Q}})\tilde{Q}\big],  D^{\frac{1}{2}+\eps}\tilde{u}\right )  = \Im \left ( \big[2 D^{\frac{1}{2}+\eps}\Re(\tilde{u}\overline{\tilde{Q}})\tilde{Q}\big],  D^{\frac{1}{2}+\eps}\tilde{u}\right )  + \text{error}_1 .
\end{align*}
In order to estimate the error term, introduce $f = 2\Re(\tilde{u}\overline{\tilde{Q}})$, $g = \tilde{Q}$, 
$h = D^{\frac{1}{2}}\tilde{u}$. Then using Plancherel's theorem, we find 
\begin{align*}
\text{error}_1 = \Im \int_{\R}|\xi|^{\eps}\int_{\R}\hat{f}(\xi - \eta)(|\xi|^{\alpha} - |\xi-\eta|^\alpha)\hat{g}(\eta)\overline{\hat{h}(\xi)}\,d\eta d\xi,\quad \alpha = \frac{1}{2}+\eps,
\end{align*}
and so we infer the bound 
\begin{equation}\label{eq:bound3}\begin{split}
|\text{error}_1|&\leq (\|D^\eps f\|_{L^2}\|\widehat{D^\alpha g}\|_{L^1} + \|f\|_{L^2}\|\widehat{D^{\alpha+\eps}g}\|_{L^1})\|h\|_{L^2}\\
&\lesssim (\lambda^{1-\eps}\lambda^{-\frac{1}{2}}\lambda^{-1-\eps} + \lambda\lambda^{-\frac{1}{2}}\lambda^{-1-2\eps})\lambda^{\frac{1}{2}}\lesssim \lambda^{-2\eps} .
\end{split}\end{equation}
We further obtain 
\begin{align*}
\Im \left ( \big[2 D^{\frac{1}{2}+\eps}\Re(\tilde{u}\overline{\tilde{Q}})\tilde{Q}\big],  D^{\frac{1}{2}+\eps}\tilde{u}\right )  = \Im \left ( \big[2 \Re(D^{\frac{1}{2}+\eps}\tilde{u}\overline{\tilde{Q}})\tilde{Q}\big],  D^{\frac{1}{2}+\eps}\tilde{u}\right )  + \text{error}_2,
\end{align*}
and we can estimate with $f = \tilde{u}$, $g = \overline{\tilde{Q}}$, $h = \overline{\tilde{Q}}D^{\frac{1}{2}+\eps}\tilde{u}$,
\begin{equation}\nonumber\begin{split} 
|\text{error}_2| = \big|\int_{\R}\int_{\R}\hat{f}(\xi-\eta)(|\xi|^{\alpha} - |\xi-\eta|^\alpha)\hat{g}(\eta)\overline{\hat{h}(\xi)}\,d\eta d\xi\big| .
\end{split}\end{equation}
Then since we have 
\[
\| h  - D^{\eps}(\overline{\tilde{Q}}D^{\frac{1}{2}}\tilde{u})\|_{L^2}\lesssim \|\widehat{D^\eps\overline{\tilde{Q}}}\|_{L^1}\|D^{\frac{1}{2}}\tilde{u}\|_{L^2}\lesssim \lambda^{-\frac{1}{2}-\eps}\lambda^{\frac{1}{2}},
\]
we find 
\begin{equation}\label{eq:bound4}\begin{split}
|\text{error}_2| &\lesssim \|f\|_{L^2}\|\widehat{D^{\alpha}g}\|_{L^1}\| h  - D^{\eps}(\overline{\tilde{Q}}D^{\frac{1}{2}}\tilde{u})\|_{L^2}\\
&+\|D^\eps f\|_{L^2}\|\widehat{D^{\alpha}g}\|_{L^1}\|\overline{\tilde{Q}}D^{\frac{1}{2}}\tilde{u}\|_{L^2}\\
&+\|f\|_{L^2}\|\widehat{D^{\alpha+\eps}g}\|_{L^1}\|\overline{\tilde{Q}}D^{\frac{1}{2}}\tilde{u}\|_{L^2}\\
&\lesssim \lambda \lambda^{-1-\eps}\lambda^{-\frac{1}{2}-\eps}\lambda^{\frac{1}{2}} + \lambda^{1-\eps}\lambda^{-1-\eps}\lambda^{-\frac{1}{2}}\lambda^{\frac{1}{2}}\\
&\lesssim \lambda^{-2\eps} .
\end{split}\end{equation}
We have thus far reduced estimating the term 
\[
\Im \left (  D^{\frac{1}{2}+\eps}\big[2\Re(\tilde{u}\overline{\tilde{Q}})\tilde{Q}\big],  D^{\frac{1}{2}+\eps}\tilde{u}\right )  
\]
to estimating the term
\begin{align*}
&\Im \left (  \big[2\Re(D^{\frac{1}{2}+\eps}\tilde{u}\overline{\tilde{Q}})\tilde{Q}\big],  D^{\frac{1}{2}+\eps}\tilde{u}\right ) \\
&= \Im \left (  \big[2D^{\frac{1}{2}}\Re(D^{\eps}\tilde{u}\overline{\tilde{Q}})\tilde{Q}\big],  D^{\frac{1}{2}+\eps}\tilde{u}\right )  + \text{error}_3\\
&=\Im \left (  \big[2\Re(D^{\eps}\tilde{u}\overline{\tilde{Q}})\tilde{Q}\big],  D^{1+\eps}\tilde{u}\right )  + \text{error}_3 + \text{error}_4 . \\
\end{align*}
Denoting $f = D^\eps\tilde{u}$, $g = \overline{\tilde{Q}}$, $h = \overline{\tilde{Q}}D^{\frac{1}{2}}\tilde{u}$, $h_1 = D^{\eps}h - \overline{\tilde{Q}}D^{\frac{1}{2}+\eps}\tilde{u}$, we find with Plancherel's theorem
\begin{equation}\label{eq:bound5}\begin{split}
|\text{error}_3| &\lesssim \int_{\R}|\hat{f}|(\xi-\eta)\big||\xi|^{\frac{1}{2}} - |\xi-\eta|^{\frac{1}{2}}\big||\hat{g}(\eta)| |\xi|^{\eps}|\hat{h}|(\xi)\,d\eta d\xi\\
&+\int_{\R}|\hat{f}|(\xi-\eta)\big||\xi|^{\frac{1}{2}} - |\xi-\eta|^{\frac{1}{2}}\big||\hat{g}(\eta)||\widehat{h_1}|(\xi)\,d\eta d\xi\\
&\lesssim [\|D^{\eps}f\|_{L^2}\|\widehat{D^{\frac{1}{2}}g}\|_{L^1} +  \|f\|_{L^2}\|\widehat{D^{\frac{1}{2}+\eps}g}\|_{L^1}]\|h\|_{L^2}\\
&+\|f\|_{L^2}\|\widehat{D^{\frac{1}{2}}g}\|_{L^1}\|h_1\|_{L^2}\\
&\lesssim (\lambda^{1-2\eps}\lambda^{-1} + \lambda^{1-\eps}\lambda^{-1-\eps})\lambda^{-\frac{1}{2}}\lambda^{\frac{1}{2}}
+\lambda^{1-\eps}\lambda^{-1}\lambda^{-\frac{1}{2}-\eps}\lambda^{\frac{1}{2}}\lesssim \lambda^{-2\eps} .
\end{split}\end{equation}
Further, we find with $f = 2\Re(D^{\eps}\tilde{u}\overline{\tilde{Q}})$, $g = \tilde{Q}$, $h = D^{\frac{1}{2}}\tilde{u}$,
\begin{equation}\label{eq:bound6}\begin{split}
|\text{error}_4| &\lesssim \int_{\R}|\hat{f}|(\xi-\eta)\big||\xi|^{\frac{1}{2}} - |\xi-\eta|^{\frac{1}{2}}\big| |\hat{g}|(\eta)|\xi|^{\eps}|\hat{h}|(\xi)\,d\eta d\xi\\
&\lesssim (\| D^\eps f\|_{L^2}\|\widehat{D^{\frac{1}{2}}g}\|_{L^1} + \|f\|_{L^2}\|\widehat{D^{\frac{1}{2}+\eps}g}\|_{L^1})\|h\|_{L^2}\\
&\lesssim (\lambda^{1-2\eps}\lambda^{-\frac{1}{2}}\lambda^{-1} + \lambda^{1-\eps}\lambda^{-\frac{1}{2}-\eps}\lambda^{-1}+\lambda^{1-\eps}\lambda^{-\frac{1}{2}}\lambda^{-1-\eps})\lambda^{\frac{1}{2}}\\
&\lesssim \lambda^{-2\eps} .
\end{split}\end{equation}
We have now reduced things to the term 
\begin{align*}
&\Im \left (  \big[2\Re(D^{\eps}\tilde{u}\overline{\tilde{Q}})\tilde{Q}\big],  D^{1+\eps}\tilde{u}\right ) \\
&= - \Re \left (  \big[2\Re(D^{\eps}\tilde{u}\overline{\tilde{Q}})\tilde{Q}\big],  D^{\eps}\partial_t\tilde{u}\right ) 
-\Im \left (  \big[2\Re(D^{\eps}\tilde{u}\overline{\tilde{Q}})\tilde{Q}\big],  D^{\eps}G\right ) \\
&=- \Re \left (  2\Re(D^{\eps}\tilde{u}\overline{\tilde{Q}}), \partial_t\big(\overline{\tilde{Q}}D^{\eps}\tilde{u}\big)\right )  + \Re \left (  2\Re(D^{\eps}\tilde{u}\overline{\tilde{Q}}), (\partial_t\overline{\tilde{Q}})D^{\eps}\tilde{u}\right ) \\
&-\Im \left (  \big[2\Re(D^{\eps}\tilde{u}\overline{\tilde{Q}})\tilde{Q}\big],  D^{\eps}G\right ) , \\
\end{align*}
where we put $G = |\tilde{u}|^2\tilde{u} + \psi + F$. We finally estimate the contributions of these three terms: for the first term after the last equality sign, we have 
\begin{equation}\label{eq:bound7}\begin{split}
\int_t^{t_1} -\Re \left (  2\Re(D^{\eps}\tilde{u}\overline{\tilde{Q}}), \partial_t\big(\overline{\tilde{Q}}D^{\eps}\tilde{u} \right )\,dt = \|\Re(D^{\eps}\tilde{u}\overline{\tilde{Q}})|_{t}^{t_1}\|_{L^2}^2\lesssim (\lambda^{\frac{1}{2}-\eps})^{2} .
\end{split}\end{equation}
Next, we find 
\begin{equation}\label{eq:bound8}\begin{split}
&\big|\Re \left (  2\Re(D^{\eps}\tilde{u}\overline{\tilde{Q}}), (\partial_t\overline{\tilde{Q}})D^{\eps}\tilde{u}\right ) \big|\\
&\lesssim \|D^{\eps}\tilde{u}\overline{\tilde{Q}})\|_{L^2}\|\partial_t\overline{\tilde{Q}}\|_{L^\infty}\|D^{\eps}\tilde{u}\|_{L^2}\lesssim (\lambda^{1-\eps})^2\lambda^{-1} .
\end{split}\end{equation}
Finally, we estimate the third term above involving the expression $G_n$. We have schematically 
\[
|\tilde{u}|^2\tilde{u} + F = \tilde{Q}^2\tilde{u} + \tilde{Q}\tilde{u}^2 + \tilde{u}^3
\]
Using Lemma \ref{lem:brezis}, it follows that 
\begin{align*}
\|D^{\eps}[|\tilde{u}|^2\tilde{u} + F]\|_{L^2}&\lesssim \|D^{\eps}\tilde{u}\|_{L^2}(\lambda^{-1}+\lambda^{-\frac{1}{2}}\lambda^{\frac{1}{2}}\log^{\frac{1}{2}}\big(\frac{\|\tilde{u}\|_{H^{\frac{1}{2}+\eps}}}{\|\tilde{u}\|_{H^{\frac{1}{2}}}}\big) +\lambda\log\big(\frac{\|\tilde{u}\|_{H^{\frac{1}{2}+\eps}}}{\|\tilde{u}\|_{H^{\frac{1}{2}}}}\big))\\
&+\|\tilde{u}\|_{L^2}(\lambda^{-1-\eps} + \lambda^{-\frac{1}{2}-\eps}\lambda^{\frac{1}{2}}\log^{\frac{1}{2}}\big(\frac{\|\tilde{u}\|_{H^{\frac{1}{2}+\eps}}}{\|\tilde{u}\|_{H^{\frac{1}{2}}}}\big))\\
&\lesssim \lambda^{-\eps} + \lambda^{1-\eps}\log(\|\tilde{u}\|_{H^{\frac{1}{2}+\eps}}) .
\end{align*}
We conclude that 
\begin{equation}\label{eq:bound9}\begin{split}
&\big|\Im \left (  \big[2\Re(D^{\eps}\tilde{u}\overline{\tilde{Q}})\tilde{Q}\big],  D^{\eps}G\right ) \big|\\&\lesssim \|\Re(D^{\eps}\tilde{u}\overline{\tilde{Q}})\tilde{Q}\|_{L^2}[\|D^\eps|[\tilde{u}|^2\tilde{u} + F]\|_{L^2} + \|D^{\eps}\psi\|_{L^2}]\\
&\lesssim \lambda^{-\eps}[\lambda^{-\eps} + \lambda^{1-\eps}\log(\|\tilde{u}\|_{H^{\frac{1}{2}+\eps}})] .
\end{split}\end{equation}
The inequalities \eqref{eq:bound3} - \eqref{eq:bound9} complete the estimate of the term 
\[
\Im \left (  D^{\frac{1}{2}+\eps}\big[2\Re(\tilde{u}\overline{\tilde{Q}})\tilde{Q}\big],  D^{\frac{1}{2}+\eps}\tilde{u}\right )  .
\] 
We continue with the remaining interactions in 
\[
i\Im\left (  -D^{\frac{1}{2}+\eps}[|\tilde{u}|^2\tilde{u} + F],\,D^{\frac{1}{2}+\eps}\tilde{u}\right )  .
\]
We write the higher order terms in $[|\tilde{u}|^2\tilde{u} + F]$ schematically in the form 
\[
\tilde{Q}\tilde{u}^2 + |\tilde{u}|^2\tilde{u}
\]
We get 
\begin{equation}\label{eq:bound10}\begin{split}
&\big|i\Im\left (  -D^{\frac{1}{2}+\eps}[|\tilde{u}|^2\tilde{u}+ F],\,D^{\frac{1}{2}+\eps}\tilde{u}\right ) \big|\\
&\lesssim \|D^{\frac{1}{2}+\eps}\tilde{u}\|_{L^2}\big(\lambda^{-\frac{1}{2}}\|\tilde{u}\|_{H^{\frac{1}{2}}}\log^{\frac{1}{2}}\big(\frac{\|\tilde{u}\|_{H^{\frac{1}{2}+\eps}}}{\|\tilde{u}\|_{H^{\frac{1}{2}}}}\big)\|D^{\frac{1}{2}+\eps}\tilde{u}\|_{L^2}\\
&\hspace{2.5cm}+\|\tilde{u}\|_{H^{\frac{1}{2}}}^2\log\big(\frac{\|\tilde{u}\|_{H^{\frac{1}{2}+\eps}}}{\|\tilde{u}\|_{H^{\frac{1}{2}}}}\big)\|D^{\frac{1}{2}+\eps}\tilde{u}\|_{L^2}\big)
\end{split}\end{equation}
By combining \eqref{eq:bound0} -- \eqref{eq:bound10} and denoting $Y(t): = \|D^{\frac{1}{2}+\eps}\tilde{u}\|_{L^2}^2$, we deduce
\[
|Y'(t)|\lesssim  \lambda^{-2\eps} + \lambda^{1-\eps}\log\big(\frac{Y^{\frac{1}{2}}(t)+1}{\|\tilde{u}\|_{H^{\frac{1}{2}}}}\big)+Y(t)\lambda^{-\frac{1}{2}}\|\tilde{u}\|_{H^{\frac{1}{2}}}\log^{\frac{1}{2}}\big(\frac{Y^{\frac{1}{2}}(t)+1}{\|\tilde{u}\|_{H^{\frac{1}{2}}}}\big).
\]
In view of the fact that $\lambda\sim t^2$, $Y(t_1)\lesssim \lambda^{\frac{1}{2}-2\eps},$ and $\|\tilde{u}\|_{H^{\frac{1}{2}}}\lesssim \lambda^{\frac{1}{2}}$, a Gronwall-lemma type argument implies that 
\[
Y(t)\lesssim \lambda^{\frac{1}{2}-2\eps},
\]
provided $\eps<\frac{1}{4}$, for $t$ sufficiently small, which is the desired a-priori bound. This completes the {\bf Step 6} in the proof of Lemma \ref{lem:back}.

\section{Fractional Leibniz Type Formula}

\begin{lemma} \label{lem:leibniz}
Suppose $N \geq 1$ and let $\phi : \R^N \to \R$ be such that $\nabla \phi$ and $\Delta \phi$ belong to $L^\infty(\R^N)$. Then we have
$$
\left | \int_{\R^N} \bar{u}(x) \nabla \phi(x)  \cdot \nabla u(x) \right | \lesssim \| \nabla \phi \|_{L^\infty} \| u \|_{\dot{H}^{1/2}}^2 + \| \Delta \phi  \|_{L^\infty} \| u \|_{L^2}^2.
$$
\end{lemma}

\begin{proof}
By density, it suffices to prove this bound for any Schwartz function $u \in \mathcal{S}(\R^N)$. Let
$$
(f, T g) = \int_{\R^d} \bar{f}(x) \nabla \phi (x) \cdot \nabla g(x) , \quad \mbox{for $f,g \in \mathcal{S}(\R^N)$}.
$$
Define $a = \| \nabla \phi \|_{L^\infty}$ and $b = \| \Delta \phi  \|_{L^\infty}$, where we suppose that $b > 0$ (and hence $a > 0$) holds. (Otherwise, the arguments below can be trivially modified in this case.)  We define the norm $\| \cdot \|_{H^1_{a,b}}$ by setting
$$
\| u \|_{H^{1}_{a,b}}^2 = a^2 \| \nabla u \|_{L^2}^2 + b^2 \| u \|_{L^2}^2 .
$$
By the Cauchy--Schwarz inequality, we immediately find that
$$
\left |(f,Tg) \right | \lesssim \| \nabla \phi  \|_{L^\infty} \|f \|_{L^2} \| \nabla g \|_{L^2} \lesssim \|f \|_{L^2} \| g \|_{H^1_{a,b}} .
$$
On the other hand, if we integrate by parts and apply Cauchy--Schwarz again, we obtain that
\begin{align*}
\left | (f,Tg) \right | & = \left | \int \nabla \bar{f}(x) \cdot \nabla \phi(x) g(x) + \bar{f}(x) \Delta \phi(x) g(x) \right |  \\
& \lesssim \| \nabla \phi \|_{L^\infty} \| \nabla f \|_{L^2} \| g \|_{L^2} + \| \Delta \phi  \|_{L^2} \| f \|_{L^2} \| g \|_{L^2}  \lesssim \| g \|_{L^2} \| f \|_{H^{1}_{a,b}}.
\end{align*}
Combining the previous estimates, we deduce the operator bounds
\begin{equation} \label{ineq:inter1}
\| T \|_{L^2 \to H^{-1}_{a,b}} \lesssim 1 \quad \mbox{and} \quad  \| T \|_{H^{1}_{a,b} \to L^2} \lesssim 1,
\end{equation}
where the space $H^{-1}_{a,b}$ denotes the dual of $H^{1}_{a,b}$ equipped with the dual norm $\| u \|_{H^{-1}_{a,b}} = \sup \{ \left | (v,u) \right |$. 

Now, we are ready to use standard interpolation theory to complete the proof. Indeed, let 
$$X^\vartheta(\R^N) = \left [L^2(\R^N), H^{1}_{a,b}(\R^N)  \right ]_{2, \theta}$$ 
denote the real interpolation of $L^2(\R^N)$ and $H^{1}_{a,b}(\R^N)$ with exponent $\vartheta \in (0,1)$. Using Plancherel's theorem and the equivalence $(a^2 |\xi|^2 + b^2)^{\vartheta} \sim a^{2 \theta} |\xi|^{2 \vartheta} + b^{2 \vartheta}$ and applying standard interpolation arguments (see e.\,g.~\cite[Lemma 23.1]{Ta2007}), we deduce with equivalence of norms that
$$
X^{\vartheta}(\R^N) \simeq H^{\vartheta}_{a,b}(\R^N), \quad \mbox{for $\vartheta \in (0,1)$},
$$ 
where the norm $\| \cdot \|_{H^{\vartheta}_{a,b}}$ is given by 
$$
\| \cdot \|_{H^{\vartheta}_{a,b}}^2 = a^{2 \vartheta} \| D^{\vartheta} u \|_{L^2}^2 + b^{2 \vartheta} \| u \|_{L^2}^2.
$$
From interpolation theory we deduce from \eqref{ineq:inter1} the bound
$$
\left | (f,Tg) \right | \lesssim \| f \|_{H^{\vartheta}_{a,b}} \| g \|_{H^{1-\vartheta}_{a,b}}, \quad \mbox{for $\vartheta \in (0,1)$}.
$$
By taking $\vartheta = 1/2$ and $f=\bar{u}$ and $g=u$, we complete the proof of Lemma \ref{lem:leibniz}. \end{proof}

\end{appendix}

\bibliographystyle{amsplain}


\def\cprime{$'$}
\providecommand{\bysame}{\leavevmode\hbox to3em{\hrulefill}\thinspace}
\providecommand{\MR}{\relax\ifhmode\unskip\space\fi MR }
\providecommand{\MRhref}[2]{%
  \href{http://www.ams.org/mathscinet-getitem?mr=#1}{#2}
}
\providecommand{\href}[2]{#2}

\end{document}